\documentclass[11pt,  reqno]{amsart}

\usepackage{amsmath,amssymb,amscd,amsthm,amsxtra, esint}

\usepackage{hyperref}

\usepackage{cases}

\allowdisplaybreaks[2]

\newtheorem{theorem}{Theorem} [section]

\newtheorem{lemma}[theorem]{Lemma}
\newtheorem{proposition}[theorem]{Proposition}
\newtheorem{remark}[theorem]{Remark}

\newtheorem{corollary}[theorem]{Corollary}

\DeclareMathOperator*{\intt}{\int}

\DeclareMathOperator*{\supp}{supp}

\newcommand{\I}{\hspace{0.5mm}\text{I}\hspace{0.5mm}}
\newcommand{\II}{\text{I \hspace{-2.8mm} I} }
\newcommand{\III}{\text{I \hspace{-2.9mm} I \hspace{-2.9mm} I}}

\newcommand{\noi}{\noindent}
\newcommand{\Z}{\mathbb{Z}}
\newcommand{\R}{\mathbb{R}}

\newcommand{\T}{\mathbb{T}}

\newcommand{\N}{\mathcal{N}}
\newcommand{\RR}{\mathcal{R}}

\newcommand{\F}{\mathcal{F}}
\renewcommand{\P}{\mathbf{P}}

\newcommand{\Nf}{\mathfrak{N}}

\newcommand{\al}{\alpha}
\newcommand{\be}{\beta}
\newcommand{\dl}{\delta}

\newcommand{\eps}{\varepsilon}

\newcommand{\g}{\gamma}

\newcommand{\ld}{\lambda}
\newcommand{\Ld}{\Lambda}
\newcommand{\s}{\sigma}
\newcommand{\ft}{\widehat}

\newcommand{\wt}{\widetilde}
\newcommand{\cj}{\overline}
\newcommand{\dx}{\partial_x}
\newcommand{\dt}{\partial_t}
\newcommand{\dd}{\partial}
\newcommand{\invft}[1]{\overset{\vee}{#1}}
\newcommand{\lrarrow}{\leftrightarrow}

\newcommand{\LLRA}{\Longleftrightarrow}

\renewcommand{\l}{\ell}

\newcommand{\les}{\lesssim}
\newcommand{\ges}{\gtrsim}

\newcommand{\jb}[1]
{\langle #1 \rangle}

\newcommand{\ind}{\mathbf 1}

\renewcommand{\S}{\mathcal{S}}

\numberwithin{equation}{section}
\numberwithin{theorem}{section}

 \DeclareMathAlphabet{\mathpzc}{OT1}{pzc}{m}{it}

\begin{document}


\title[Non-existence for  the periodic cubic NLS below $L^2$]
{Non-existence of solutions for the periodic cubic NLS below $L^2$}

\author{Zihua Guo and  Tadahiro Oh}

\address{Zihua Guo\\
School of Mathematical Sciences\\
Monash University\\
VIC 3800, Australia, 
and 
LMAM, School of Mathematical Sciences, Peking University, Beijing 100871, China} 

\email{zihua.guo@monash.edu}

\address{
Tadahiro Oh\\
School of Mathematics\\
The University of Edinburgh, 
and The Maxwell Institute for the Mathematical Sciences\\
James Clerk Maxwell Building\\
The King's Buildings\\
 Peter Guthrie Tait Road\\
Edinburgh\\ 
EH9 3FD\\ 
United Kingdom} 

\email{hiro.oh@ed.ac.uk}



\begin{abstract}
We prove non-existence of solutions for the cubic nonlinear Schr\"odinger
equation (NLS) on the circle if initial data belong to $H^s(\T) \setminus L^2(\T)$ for some $s \in (-\frac18, 0)$.
The proof is based on establishing an a priori bound
on solutions to a renormalized cubic NLS in negative Sobolev spaces
via the short-time Fourier restriction norm method.

\end{abstract}

\maketitle

\tableofcontents

\section{Introduction}
\label{SEC:1}

\subsection{The cubic nonlinear Schr\"odinger equation on $\T$ and $\R$}
We consider the Cauchy problem for the cubic nonlinear Schr\"odinger equation
(NLS) on the circle $\T = \R/(2\pi \Z)$:
\begin{equation}
    \label{NLS0}
    \begin{cases}
        i  \dt u - \dx^2  u   \pm |u|^2 u = 0 \\
        u|_{t= 0} = u_0,
    \end{cases}
    \qquad (x, t)  \in \T \times \R,
\end{equation}

\noi
where $u$ is a complex-valued function.
The Cauchy problem  \eqref{NLS0} has been studied extensively
from both theoretical and applied points of view.
It is known to be one of the simplest partial differential equations (PDEs)
with complete integrability \cite{AKNS, AM, GK}.
In the following, however, 
we only discuss analytical aspects
of \eqref{NLS0} without using the complete integrable structure of the equation.

It is well known that \eqref{NLS0} enjoys several symmetries.
In the following, we concentrate on the scaling symmetry and the Galilean symmetry.
The scaling symmetry states that if $u(x, t)$ is a solution to \eqref{NLS0}
on $\R$ with initial condition $u_0$, then
$u^\ld(x, t) = \ld^{-1} u (\ld^{-1}x, \ld^{-2}t)$
is also a solution to \eqref{NLS0} with the $\ld$-scaled initial condition 
$u_0^\ld(x) = \ld^{-1} u_0 (\ld^{-1}x)$.
Associated to the scaling symmetry, 
there is a  scaling-critical Sobolev index $s_c$
such that the homogeneous $\dot{H}^{s_c}$-norm is invariant
under the dilation symmetry.
It is commonly conjectured that a PDE is ill-posed in $H^s$ for $s < s_c$.
In the case of the one-dimensional cubic NLS, 
 the scaling-critical Sobolev index is $s_c = -\frac{1}{2}$
 and 
Christ-Colliander-Tao \cite{CCT2b} exhibited
a norm inflation phenomenon in  $H^s(\R)$
for $s< -\frac{1}{2}$.\footnote{This norm inflation also holds
at the scaling critical regularity  $s = - \frac 12$   on both $\R$ and $\T$.
See Kishimoto \cite{Kishimoto} and Oh-Wang \cite{OH}.
See also a recent result by Carles-Kappeler \cite{CK} for 
a norm inflation in the Fourier-Lebesgue space $\F L^{s, p}(\T)$
for  $ s< -\frac 23$ and $p \in [1, \infty]$.
These norm inflation results also hold for the Wick ordered cubic NLS defined in \eqref{NLS1}.} 
While there is no dilation symmetry in the periodic setting, 
the heuristics provided by the scaling argument plays an important role
even in the periodic setting. 
For example, 
a norm inflation result 
for  \eqref{NLS0} on $\T$
analogous to \cite{CCT2b} also holds  
for $s < s_c = - \frac 12$.

The Galilean symmetry states that if $u(x, t)$ is a solution to \eqref{NLS0}
on $\R$ with initial condition $u_0$, then
$u^\beta(x, t) = e^{i\frac{\beta}{2}x}e^{i\frac{\beta^2}{4}t}  u (x+\beta t, t)$
is also a solution to \eqref{NLS0} with the modulated initial condition 
$u_0^\beta(x) = e^{i\frac{\beta}{2}x}  u_0 (x)$.
This symmetry also holds on the circle for $\be \in 2 \Z$.
Note that the $L^2$-norm is invariant under the Galilean symmetry.\footnote{In fact, 
any Fourier-Lebesgue $\F L^{0, p}$-norm defined in \eqref{FL} is invariant under the Galilean symmetry.}
This induces another critical regularity $s_\text{crit}^\infty  =  0$.
Indeed, there is a strong dichotomy in the behavior of solutions for  $s\geq 0$ and $s < 0$.

Bourgain \cite{BO1} introduced the so-called 
Fourier restriction norm method via the $X^{s, b}$-spaces  (see \eqref{Xsb1} below)
and proved local well-posedness
of \eqref{NLS0} in $L^2(\T)$.
Thanks to the $L^2$-conservation, 
this result was immediately extended to global well-posedness in $L^2(\T)$. 
On the other hand, \eqref{NLS0}
is known to be ill-posed below $L^2(\T)$.
Burq-G\'erard-Tzvetkov \cite{BGT} showed that 
the solution map $\Phi_t: u_0\in H^s \mapsto u(t) \in H^s$
is not locally uniformly continuous if $s < 0$.  See also \cite{CCT1}.
Moreover, 
Christ-Colliander-Tao \cite{CCT2} and Molinet \cite{MOLI}
proved that the solution map is indeed discontinuous
if $s < 0$.
Our main result in this paper
states an even stronger form of ill-posedness
holds true for \eqref{NLS0}
in negative Sobolev spaces.

\begin{theorem}[Non-existence of  solutions for the cubic NLS]\label{THM:2}

Let $s \in ( -\frac{1}{8}, 0)$
and $u_0 \in H^s (\T)\setminus L^2(\T)$.
Then, for any $T>0$, 
there exists no distributional  solution $u \in C([-T, T]; H^s(\T))$ to 
the cubic  NLS \eqref{NLS0}
such that

\begin{itemize}
\item[\textup{(i)}] $u|_{t = 0} = u_0$, 

\smallskip

\item[\textup{(ii)}] There exist smooth global solutions $\{u_n\}_{n\in \mathbb{N}}$ to \eqref{NLS0} such that 
$u_n \to u$ in $ C([-T, T]; H^s(\T))$ as $n \to \infty$. 
\end{itemize}

\end{theorem}

Note that the condition (ii) is a natural condition to impose
since it would follow from  the continuity of the solution map\,$: u_0 \in H^s(\T) \mapsto u \in C([-T, T];H^s(\T))$,
which is one of the essential components in 
the well-posedness theory of evolution equations.
On the one hand,  the previous ill-posedness results
\cite{BGT, CCT1, CCT2, CCT2b, MOLI}
treated 
 smooth functions, at least in $L^2(\T)$, 
to construct examples.
On the other hand, 
the proof of Theorem \ref{THM:2} is based on 
establishing a priori bound on solutions to a renormalized cubic NLS (see \eqref{NLS1} below)
in negative Sobolev spaces.
See Theorem \ref{THM:1} in Subsection \ref{SUBSEC:1.2}.
 
In making sense of \eqref{NLS0} on $\T \times [-T, T]$ in the distributional sense, 
we need to assume that a solution $u(t)$ is a priori in $L^3(\T)$
for almost every $t \in [-T, T]$.
Thus, the usual construction of solutions
in low regularities
employs
an auxiliary space-time function space
$X_T$.\footnote{
There is a more relaxed notion of weak solutions 
in the extended sense \cite{CH1, CH2}.
In \cite{GKO}, we constructed
weak solutions in the extended sense  in $L^2(\T)$
without any auxiliary function space.
Moreover, the result in \cite{GKO} yields unconditional uniqueness in $H^s(\T)$
for $s \geq \frac 16$.
}
For example, the construction in \cite{BO1}
assumes a priori that solutions are in 
 $L^4_{x, t}(\T\times[-T, T]) \subset X^{0, \frac{1}{2}+}([-T, T])$.
Theorem \ref{THM:2}
asserts non-existence of solutions below $L^2(\T)$
{\it even if} we assume a priori that 
$u$ belongs to some auxiliary function space $X_T$.

Recently, there has been a significant development
in probabilistic construction of local-in-time and global-in-time solutions to 
dispersive and hyperbolic PDEs. 
See, for example, \cite{BTI, CO}. 
Such probabilistic arguments
often allowed us to go below
certain regularity thresholds such as a scaling critical regularity, 
below which the equations are known to be ill-posed deterministically.
We would like to point out that 
the ill-posedness result stated in  
Theorem \ref{THM:2}
is much stronger than ordinary ill-posedness
results such as the failure of (local uniform) continuity
of a solution map.
In particular, 
Theorem \ref{THM:2}
states that
 it is {\it not} possible
 to construct solutions
to 
the cubic NLS \eqref{NLS0}
for initial data below $L^2(\T)$
{\it even probabilistically}.

Lastly, let us compare the situation with the non-periodic case. 
Tsutsumi \cite{Tsutsumi} proved global well-posedness
of \eqref{NLS0} in $L^2(\R)$.
There are several results \cite{KPV, CCT1}
showing that 
the cubic NLS on $\R$ is `mildly  ill-posed' below $L^2$
in the sense that the solution map is not locally uniformly continuous if $s < 0$.
There is, however,  no ill-posedness result below $L^2$
for the cubic NLS on $\R$, contradicting either existence, uniqueness, or continuous dependence.
For example, 
while  Molinet's ill-posedness result \cite{MOLI} for the  cubic NLS on $\T$
is based on showing 
 discontinuity of the solution map
 on  $L^2(\T)$ endowed with weak topology, 
Goubet-Molinet
\cite{GM}
showed that 
the solution map for 
the cubic NLS on $\R$ is  weakly continuous on $L^2(\R)$.
Moreover, 
Christ-Colliander-Tao \cite{CCT3}
and Koch-Tataru \cite{KT1, KT2}
proved 
existence (without uniqueness)
of solutions to  the cubic NLS on $\R$ in negative Sobolev spaces.
Note that 
solutions constructed in \cite{CCT3, KT1, KT2}
satisfy the conditions (i) and (ii) in Theorem \ref{THM:2}, 
showing 
a sharp contrast between the periodic and non-periodic cases.

\subsection{The Wick ordered cubic NLS on $\T$}
\label{SUBSEC:1.2}
Given a  global solution $u \in C(\R; L^2(\T))$ to \eqref{NLS0}, 
we define the following invertible gauge transformation:
\begin{equation}
u(t) \mapsto
\mathcal{G}(u)(t) : = e^{\mp 2 i t \mu(u)} u(t),
\label{gauge}
\end{equation}

\noi
where $\mu(u)$ is defined by
\[ \mu (u) = \fint |u(x, t)|^2 dx := \frac{1}{2\pi} \int_\T |u(x, t)|^2 dx. \]

\noi
Thanks to the $L^2$-conservation, 
$\mu(u)$ is defined, independently of $t \in \R$,
as long as $u_0 \in L^2(\T)$.
A direct computation shows that the gauged function,  which we still denote by $u$, 
solves the following
 {\it Wick ordered cubic NLS}:
\begin{equation}
	\label{NLS1} 
	\begin{cases}
		i \dt u - \dx^2 u \pm ( |u|^2 -2  \fint |u|^2 dx) u = 0 \\
		u|_{t= 0} = u_0,
	\end{cases}
\quad 	(x, t)  \in \T \times \R.
\end{equation}

\noi
Conversely, given a global solution $u \in C(\R; L^2(\T))$
to \eqref{NLS1}, 
we see that $\mathcal{G}^{-1}(u)$ solves the original cubic NLS \eqref{NLS0}.
Such a gauge transformation, however, 
does not make sense below $L^2(\T)$
and thus
we cannot freely convert solutions of \eqref{NLS1} into solutions of \eqref{NLS0}.
In other words, 
the Wick ordered  cubic NLS \eqref{NLS1} arises from 
 the standard  cubic NLS \eqref{NLS0}
by choosing another gauge on the phase space
such that they are equivalent (only) in $L^2(\T)$. 
Note that this renormalization of the nonlinearity in \eqref{NLS1}
is canonical, appearing 
in Euclidean quantum field theory, 
and 
\eqref{NLS1}
first appeared in the work of Bourgain \cite{BO96}
for studying  the invariant Gibbs measure for the defocusing cubic NLS on $\T^2$.

This specific choice of  gauge for \eqref{NLS1}
 removes a certain singular component from the cubic nonlinearity.
 As a result, 
the Wick ordered cubic NLS \eqref{NLS1}  
is known to behave better than  the cubic NLS \eqref{NLS0} outside $L^2(\T)$.
In fact, the standard cubic NLS on $\R$
and the Wick ordered cubic NLS \eqref{NLS1} on $\T$ share many common features.
For example, just like the cubic NLS on $\R$ \cite{KPV, CCT1, GM}, 
the solution map 
for the Wick ordered cubic NLS \eqref{NLS1} on $\T$
is known
to be weakly continuous in $L^2$ \cite{OS}, 
while it is `mildly ill-posed' 
in the sense that the solution map is not
locally uniformly continuous below $L^2$ \cite{CO}.
See \cite{OS} for more discussion on this issue.
In particular, we proposed in \cite{OS} that 
the Wick ordered cubic NLS \eqref{NLS1}
is the right model to study outside $L^2(\T)$.
As with any renormalization procedure or gauge choice, we stress that this is a matter of choice, 
since \eqref{NLS0} and \eqref{NLS1}
are not equivalent outside $L^2$.
The examples in \cite{OS} and Theorem \ref{THM:2} above provide
supporting evidences for our choice.

Our main goal below is to establish
an a priori estimate on solutions to \eqref{NLS1}
and prove existence of solutions to \eqref{NLS1}
in negatives Sobolev spaces.

\begin{theorem}\label{THM:1}
Let $s \in (-\frac{1}{8}, 0)$.
Given $u_0\in H^s(\T)$, there exist
$T = T(\|u_0\|_{H^s})>0$
and a solution $u \in C([-T, T];  H^s(\T))$ to
the Wick ordered cubic NLS \eqref{NLS1}. 
\end{theorem}

Previously, 
Christ \cite{CH2} and Gr\"unrock-Herr \cite{GH}
proved local well-posedness of \eqref{NLS1}
in $\F L^{s, p}(\T)$ for $s = 0$ and $ p < \infty$,
where the Fourier-Lebesgue space  $\F L^p(\T)$ 
is defined by the norm 
\begin{align}
\| f\|_{\F L^{s, p}(\T)} = \| \jb{n}^s \ft f(n)\|_{\l^p(\Z)}.
\label{FL}
\end{align}

\noi
Note that $\F L^{0, p}(\T) \supsetneq L^2(\T)$ for $p > 2$.
In the context of negative  Sobolev spaces $H^s(\T)$, $s < 0$, 
Colliander-Oh \cite{CO} 
proved almost sure local well-posedness
for $s > -\frac 13$
and almost sure global well-posedness
for $s > -\frac {1}{12}$
with respect to the canonical Gaussian
measures on $H^s(\T)$. 
In Theorem \ref{THM:1}, 
we claim only existence of solutions 
to \eqref{NLS1}
in negatives Sobolev spaces, i.e.~the uniqueness
issue remains open.
This is analogous to the situation for 
the standard cubic NLS on $\R$
\cite{CCT3, KT1, KT2}.
The proof of Theorem \ref{THM:1}
occupies most of the remaining part of this paper:
Sections \ref{SEC:notations}-\ref{SEC:existence}.
Once we prove Theorem \ref{THM:1}, 
the proof of Theorem \ref{THM:2} 
immediately follows
from the a priori bound on solutions to \eqref{NLS1}
in negative Sobolev spaces
and inverting the gauge transformation \eqref{gauge}.
See Section \ref{SEC:non-existence}.

The proof of Theorem \ref{THM:1} is based
on establishing an a priori bound
on (smooth) solutions to \eqref{NLS1}
in negative Sobolev spaces
via the short-time Fourier restriction norm method.
Then, we use a compactness argument to construct 
solutions to \eqref{NLS1} (without uniqueness).
Here,  the short-time Fourier restriction norm method
simply means that we use 
dyadically defined  $X^{s, b}$-type spaces
with suitable localization in time,
depending on the dyadic size of spatial frequencies.
A precursor of this method
appears in the work of Koch-Tzvetkov \cite{KTz}, 
where 
localization in time was combined with the Strichartz norms.
This method has been very effective 
in establishing a priori bounds on solutions
in low regularities (yielding even uniqueness in some cases), 
in particular, where a solution map
is known to fail to be locally uniformly continuous.
See \cite{CCT3, IKT, KT1, Guo1, KT2, GKK, KP}. 

Let us briefly discuss the difference between the usual Fourier restriction norm method
and the short-time Fourier restriction norm method.
Consider the following model evolution equation: $u_t - L u = \N(u)$,
where $\N(u)$ is a homogenous nonlinearity of degree $p$.
Then, 
the usual Fourier restriction norm method
requires the following two estimates:
\begin{align*}
   \text{Linear:}& 
&    \| u \|_{X^s(T)}  &  \les \|u_0\|_{H^s} +  \| \N(u)\|_{N^s(T)},  & \\
   \text{Nonlinear:}& 
&  \| \N(u) \|_{N^s(T)}  & \les \|u\|_{X^s(T)}^p,& 
\end{align*}

\noi
where  $X^s(T) = X^{s, b}([-T, T])$
and $N^s(T)= X^{s, b-1}([-T, T])$,  $b > \frac 12$, 
in the simplest setting.
Then, a fixed point argument yields well-posedness in $H^s$.\footnote{Strictly speaking, 
one needs to gain a smallness factor $T^\theta$, $\theta > 0$, 
or assume small data to close the argument.
For simplicity, however, we present only the essential part of the estimates.}
With the short-time Fourier restriction norm method, 
one now needs to establish the following three estimates:
\begin{align}
   \text{Linear:}& 
&    \| u \|_{F^s(T)}  &  \les \|u\|_{E^s(T)} +  \| \N(u)\|_{N^s(T)},  \label{Alinear}
 & \\
   \text{Nonlinear:}& 
&  \| \N(u) \|_{N^s(T)}  & \les \|u\|_{F^s(T)}^p,  \label{Anonlinear}
& \\
   \text{Energy:}& 
&    \| u \|_{E^s(T)}^2  &  \les \|u_0\|_{H^s}^2 + \|u\|_{F^s(T)}^{p+1}, \label{Aenergy}  
&
\end{align}

\noi
where $E^s(T) \thickapprox L^\infty([-T, T]; H^s(\T))$.
Then, these three estimates yield an a priori bound
on solutions in $H^s$, which allows us to prove existence of solutions 
by 
a compactness argument.
In order to improve the regularity $s$, 
one may apply a normal form reduction
and establish a better energy estimate 
at the expense of introducing higher order nonlinear terms. 
This is analogous to 
considering higher ordered modified energies
in the $I$-method developed by Colliander-Keel-Staffilani-Takaoka-Tao
\cite{CKSTT2, CKSTT1}.
See Section \ref{SEC:energy}.
In the following, we use the formulation introduced by
Ionescu-Kenig-Tataru \cite{IKT}.
In Section \ref{SEC:spaces}, 
we provide
 the precise definitions of the function spaces
 $F^s(T), N^s(T)$, and $E^s(T)$.

\begin{remark}\rm 
In order to establish uniqueness of the solutions constructed in Theorem \ref{THM:1}, 
one needs to establish an energy estimate $\| u - v\|_{E^s(T)}$
on the difference of two solutions $u$ and $v$.
On the one hand, 
the conservation of the $L^2$-norm
hints an almost conservation of the energy $\|u\|_{E^s(T)}$ for $s < 0$, 
analogously to 
 the $I$-method.
On the other hand, the $L^2$-norm of 
the difference of two solutions $u$ and $v$
is not conserved under \eqref{NLS1}.
This provides an essential difficulty in establishing 
an energy estimate on the difference of two solutions for $s < 0.$

\end{remark}

\begin{remark}\rm
For $u_0 \notin L^2(\T)$, 
the renormalized nonlinearity in \eqref{NLS1}
can be formally written as
\[(|u|^2 -2 \cdot \infty ) u = 0 .\]

\noi
What is important in this renormalization procedure
is to {\it subtract off the  right amount of infinity}.
For example, a gauge transformation:
$\mathcal{G}(u)(t) : = e^{\mp  i\g t  \mu(u)} u(t)$
with a parameter $\g\in \R$
gives rise to the following renormalized cubic NLS:
\begin{align}
\textstyle i \dt u - \dx^2 u \pm ( |u|^2 -\g  \fint |u|^2 dx) u = 0. 
\label{IntroWNLS}
\end{align}

\noi
For $u_0\in L^2(\T)$, 
\eqref{IntroWNLS} is equivalent to \eqref{NLS0} and \eqref{NLS1}.
If $u_0 \notin L^2(\T)$, however, 
they are not equivalent.
In particular, 
if $u_0 \in H^s(\T)\setminus L^2(\T)$ for some $s \in (-\frac 18, 0)$, 
then 
a slight modification of the proof of Theorem \ref{THM:2}
yields non-existence of solutions 
for \eqref{IntroWNLS} unless $\g = 2$.
This shows that ``$2\cdot \infty$'' is the right amount to subtract
in the renormalization procedure.	
See also \cite{OS} 
for the quantum field theoretic derivation of \eqref{NLS1}.

\end{remark}

\begin{remark}\label{REM:4NLS}
\rm

One can easily modify 
the argument in this paper
and show that Theorem \ref{THM:2} also holds for 
 the following fourth order dispersive analogue
of the cubic NLS:
\begin{equation}
    \label{4NLS1}
        i  \dt u +  \dx^{4} u   \pm |u|^2 u = 0, 
    \qquad (x, t)  \in \T \times \R.
\end{equation}

\noi
On the one hand, by establishing a relevant $L^4$-Strichartz estimate
and combining with the $L^2$-conservation, 
one can show that \eqref{4NLS1} is 
globally well-posed in $L^2(\T)$.
On the other hand, 
the example in \cite{BGT, CCT1}
basically yields the failure of local uniform continuity of the solution map in $H^s(\T)$
to \eqref{4NLS1}
if $s < 0$.
By considering the  Wick ordered version of the fourth order cubic NLS:
\begin{equation}
    \label{4NLS2}
\textstyle        i  \dt u +  \dx^{4} u 
        \pm ( |u|^2 -2  \fint |u|^2 dx) u = 0, 
\end{equation}

\noi
one can show that 
Theorem \ref{THM:1} also holds for \eqref{4NLS2} (even for lower values of $s<0$).
This essentially  follows from the following algebraic observation:
under $n = n_1 - n_2 + n_3$, we have 
\begin{align}
\Phi_4(\bar{n}):&  = - n^{4} + n_1^{4} - n_2^{4} + n_3^{4} \notag \\
&  = -  (n - n_1)(n-n_3) 
\big( n_1^2 +n_2^2 +n_3^2 +n^2 + 2(n_1 +n_3)^2\big).
\label{PhiX}
\end{align}

\noi
Compare this with the phase function $\Phi(\bar n)$  for the cubic NLS defined in \eqref{Phi} below.
In particular, 
the only resonant contribution, i.e. $\Phi_4(\bar{n}) = 0$, 
comes from $n = n_1$ or $n = n_3$, just like the cubic NLS.
Moreover, 
we have  $|\Phi_4(\bar{n})| \ges |\Phi(\bar{n})|$
for any $(n_1, n_2, n_3, n) \in \Z^4 $
with $n = n_1 - n_2 + n_3$.
As a corollary to Theorem \ref{THM:1} for \eqref{4NLS2}, 
we see that 
Theorem \ref{THM:2} also holds for the fourth order cubic NLS \eqref{4NLS1}.
See \cite{OT, OW} for more discussions on \eqref{4NLS1}
and \eqref{4NLS2}.
Lastly, consider the higher order dispersive analogue of \eqref{NLS0}:
\begin{equation*}
        i  \dt u +  (i \dx)^{k} u   \pm |u|^2 u = 0. 
\end{equation*}

\noi
When $k$ is an odd integer, 
there is (at least) another resonant contribution
 coming from $n = -n_2$
 and thus Theorems \ref{THM:2} and \ref{THM:1}
are not directly applicable in this case.
When $k$ is an even integer, 
it may be of interest to investigate 
if  the phase function 
$\Phi_k(\bar{n}):  = (-1)^\frac{k}{2}\big(-  n^{k} + n_1^{k} - n_2^{k} + n_3^{k}\big)$
admits a factorization similar to \eqref{PhiX} and \eqref{Phi}
such that 
results analogous to 
Theorems \ref{THM:2} and \ref{THM:1} may hold.
We, however, do not pursue  this issue further  here.

\end{remark}

This paper is organized as follows.
In 
Sections \ref{SEC:notations} and \ref{SEC:spaces}, 
we introduce
notations and the function spaces
along with their basic properties.
We prove the linear estimate \eqref{Alinear} in 
Section
\ref{SEC:linear}
and the key trilinear estimate \eqref{Anonlinear}, with $p = 3$, in 
Sections
\ref{SEC:Strichartz}
and 
\ref{SEC:trilinear}.
The energy estimate \eqref{Aenergy}
with a higher order correction term is established in 
Section
\ref{SEC:energy}.
In Section
\ref{SEC:existence}, 
we combine the estimates \eqref{Alinear}, \eqref{Anonlinear},
and \eqref{Aenergy} and prove Theorem \ref{THM:1}.
Finally, we use the a priori estimate on solutions from Theorem \ref{THM:1}
and 
prove Theorem \ref{THM:2}
in 
Section
\ref{SEC:non-existence}.

Some of the estimates presented in this paper,
in particular those in  
 Sections \ref{SEC:spaces} and \ref{SEC:linear}, 
are by now standard.
For the sake of completeness of the paper, we present the proofs in details.
Since our argument is of local-in-time nature, 
there is no difference between
the focusing and defocusing cases.
Hence, 
we assume that \eqref{NLS1} is defocusing (with the $+$ sign in \eqref{NLS0} and \eqref{NLS1})
in the following.

\smallskip

\noi
{\bf Acknowledgments:} 
The authors would like to thank Carlos Kenig 
and Didier Pilod
for helpful discussions, in particular on 
Lemma \ref{LEM:Kconti}.
T.O.~would like to thank Nikolay Tzvetkov
for our discussion on the fourth order NLS, which motivated the addition of Remark \ref{REM:4NLS}. 
The authors are grateful to the anonymous referees 
for their helpful comments.
Z.G.~was supported in part by NNSF  of China (No.11371037, No.11271023) and Beijing Higher Education Young Elite Teacher Project.
T.O.~was supported by the European Research Council (grant no.~637995 ``ProbDynDispEq'').

\section{Notations}
\label{SEC:notations}

For $a, b > 0$, we use $a\lesssim b$ to mean that
there exists $C>0$ such that $a \leq Cb$.
By $a\sim b$, we mean that $a\lesssim b$ and $b \lesssim a$.
We also use $a+$ (and $a-$) to denote $a + \eps$ (and $a - \eps$), respectively, for arbitrarily small $\eps \ll 1$.

Given $u \in \mathcal{S}'$, we
use $\ft{u}$ and $\F (u)$
to denote the space-time Fourier transform of $u$ given by
\[ \ft{u}(n, \tau) = \int_{\T\times\R} e^{- inx} e^{-it\tau} u(x, t) dx dt.\]

\noi
Moreover, we use $\F_x$ and $\F_t$
to denote the Fourier transforms with respect
to the spatial and temporal variables, respectively.
When there is no confusion,
we may simply use $\ft{u}$ or $\F(u)$
to denote
the spatial, temporal, or space-time Fourier transform
of $u$, depending on the context.
In dealing with the spatial Fourier transform,
we often denote $\ft{u}(n, t)$
by $\ft{u}_n(t)$.

For $k \in \Z_+: = \Z\cap [0, \infty)$,
let $I_0 = \{\xi: |\xi| < 1\}$
and
$I_k = \{ \xi : 2^{k-1} \leq |\xi|  < 2^k\}$ if $k \geq 1$.
For $k \in \Z_+$ and $j \geq 0$, let
\[ D_{k, j} = \big\{ (n, \tau) \in \Z\times \R: \ n \in I_k, \ \tau - n^2 \in I_j \big\} \]

\noi
and $D_{k, \leq j} = \bigcup_{j' \leq j} D_{k, j'}$.
Lastly, we define $D_{\leq j}$ by
\begin{equation*}
D_{\leq j} =  \bigcup_{k = 0}^\infty D_{k, \leq j}.
\end{equation*}

We use $\P_k$ to
denote the projection operator on $L^2(\T)$
defined by
$\ft{\P_k u}(n) = \ind_{I_k}(n) \ft{u}(n)$.
With a slight abuse of notation, we also use
 $\P_k$ to denote the projection operator on $L^2(\T \times \R)$
given by $\F(\P_k u)(n, \tau) = \ind_{I_k}(n) \F(u)(n, \tau)$.
For $\ell \in \Z$, let
\[ \P_{\leq \ell } = \sum_{0\leq k \leq \ell}\P_k
\qquad \text{and}\qquad \P_{\geq \ell } = \sum_{ k \geq \ell}\P_k.\]

Let $\eta_0: \R \to [0, 1]$ be an even smooth cutoff function
supported on $[-\frac{8}{5}, \frac{8}{5}]$
such that $\eta_0 \equiv 1$ on $[-\frac{5}{4}, \frac{5}{4}]$.
We define $\eta$ by $\eta(\xi) = \eta_0( \xi) - \eta_0(2\xi)$,
and set  $\eta_k(\xi) = \eta(2^{-k}\xi)$ for $k \in \Z$.
Namely, $\eta_k$ is supported on
$\{ \frac{5}{4} \cdot 2^{k-1}  \leq |\xi|\leq \frac{8}{5} \cdot 2^k \}$.
As before, we define
$\eta_{\leq \ell} = \sum_{k \leq \ell} \eta_k$, etc.

Given a set of indices such as
$j_i$ and $k_i$, $i = 1, \dots, 4$,
we use
 $j^*_i$ and $k^*_i$ to denote the decreasing rearrangements
of these indices.
Also, given a set of frequencies
$n_i$, $i = 1, \dots, 4$,
we use
 $n^*_i$ to  denote the decreasing rearrangements
of $|n_i|$, $i = 1, \dots, 4$.

In the following, we use $S(t)= e^{-i t \dx^2}$
to denote the solution operator to
the linear Schr\"odinger equation: $i \dt u - \dx^2 u = 0$.
Namely, for $\phi \in L^2(\T)$, we have
\[ S(t) \phi = \sum_{n\in \Z} e^{inx +in^2 t} \ft{\phi}(n).\]

Lastly, let us discuss the renormalized nonlinearity in \eqref{NLS1}.
The nonlinearity on the right-hand side of \eqref{NLS1} can be written as
\begin{align*}  
\textstyle  \mathfrak N(u)= \mathfrak N(u, u, u)
 :&=  \big( |u|^2 - 2 \fint \ |u|^2 dx\big)u\notag\\
 &=:  \N (u, u, u) - \RR (u, u, u),
\end{align*}

\noi
where the non-resonant part $\N$ and the resonant part $\RR$ are defined by
\begin{align}
\label{NN1}
& \N(u_1, u_2, u_3) (x, t) = \sum_{n_2 \ne n_1, n_3} \ft{u}_1(n_1, t)\cj{\ft{u}_2(n_2, t)}\ft{u}_3(n_3, t) e^{i(n_1-n_2+n_3)x}, \\
\label{NN2}
& \RR(u_1, u_2, u_3) (x, t) = \sum_n \ft{u}_1(n, t)\cj{\ft{u}_2(n, t)}\ft{u}_3(n, t) e^{inx}.
\end{align}

\noi
Here, the condition $n_2 \ne n_1, n_3$ in the sum for $\mathcal{N}(u)$
is a shorthand notation for $n_2 \ne n_1$ and $n_2 \ne n_3$.

In establishing an energy estimate in Section \ref{SEC:energy}, we  
use the following interaction representation $a$ (of $u$) on $\T\times \R$:
\begin{align}
a(t) := S(-t) u(t) = e^{it\dx^2} u(t).
\label{interaction}
\end{align}

\noi
On the Fourier side, we have  $\ft{a}_n(t) = e^{-it n^2} \ft u_n(t)$, $n \in \mathbb{Z}$.
For simplicity of notations, we use $a_n(t)$ to denote $\ft a_n(t)$ in the following.
The use of the interaction representation 
allows us to illustrate
the connection between the Poincar\'e-Dulac normal form
reduction discussed in \cite{GKO}
and the method of adding correction terms in the $I$-method \cite{CKSTT2, CKSTT1}.
With this notation, 
\eqref{NLS1} can be written as
\begin{align} \label{NLS2}
\dt a_n & = i
\sum_{\substack{n = n_1 - n_2 + n_3\\ n_2\ne n_1, n_3} }
e^{- i \Phi(\bar{n})t }
a_{n_1} \cj{a}_{n_2}a_{n_3}
- i|a_n|^2 a_n
 \notag \\
& =: i\,  \textsf{N}(a)(n, t) -i\,  \textsf{R}(a)(n, t),
 \end{align}

\noi
where the phase function $\Phi(\bar{n})$ is defined by
\begin{align}\label{Phi}
\Phi(\bar{n}):& = \Phi(n_1, n_2, n_3, n) = n^2 - n_1^2 + n_2^2- n_3^2 \notag \\
& = 2(n_2 - n_1) (n_2 - n_3)
= 2(n - n_1) (n - n_3).
\end{align}

\noi
Here,  the last two equalities hold under $n = n_1 - n_2 + n_3$.
Noting that 
\[\textsf{N}(a)(n, t) = e^{-in^2 t} \F(\N(u))(n, t)
\qquad 
\text{and}
\qquad 
\textsf{R}(a)(n, t) = e^{-in^2 t} \F(\RR(u))(n, t),\]

\noi
it follows from \eqref{Phi} that
$\N$ defined in \eqref{NN1} indeed corresponds to the non-resonant part (i.e. $\Phi(\bar{n})\ne 0$) of the nonlinearity
$\mathfrak{N}(u)$
and $\RR$ defined in \eqref{NN2} corresponds to the resonant part.

\section{Function spaces and their basic properties}
\label{SEC:spaces}

In \cite{BO1}, Bourgain introduced the weighted space-time Sobolev spaces
called the $X^{s, b}$-spaces via the norm:
\begin{align}\label{Xsb1}
\| u\|_{X^{s, b}(\T\times \R)} = 
\|\jb{n}^s \jb{\tau - n^2}^b \ft {u}(n, \tau)\|_{\l^2_n L^2_\tau(\Z\times\R)}.
\end{align}

\noi
In terms of the interaction representation defined in \eqref{interaction}, 
we simply have 
$\| u\|_{X^{s, b}(\T\times \R)} =
\| \jb{\dx}^s \jb{\dt}^b a\|_{L^2(\T\times \R)}$.
The $X^{s, b}$-spaces and their variants have been very effective in 
studying nonlinear evolution equations in low regularity settings.
In the following, we define the $X^{s, b}$-spaces
adapted to short time scales.
These spaces were first introduced by
Ionescu-Kenig-Tataru \cite{IKT}
in the context of the KP-I equation.
Also, see Christ-Colliander-Tao
\cite{CCT3} for similar definitions.

For $k \in \Z_+$, we define the dyadic $X^{s, b}$-type spaces $X_k(\Z\times \R)$:
\begin{align} \label{Xk}
X_k = \Big\{
& f_k \in L^2(\Z \times \R): f_k(n,\tau) \text{ is supported on } I_k \times \R, \notag \\
& \text{ and }
\|f_k \|_{X_k}:=\sum_{j=0}^\infty
2^\frac{j}{2} \|\eta_j(\tau-n^2) f_k(n ,\tau)\|_{\ell^2_n L^2_{\tau}}<\infty
\Big\}.
\end{align}

\noi
We list some properties of the space $X_k$.
It follows easily from the definition
that if  $f_k\in X_k$ for some $k \in \Z_+$, then
\begin{align}\label{Xk1}
\bigg\| \int_{\R} |f_k(n,\tau )| d\tau \bigg\|_{\ell^2_n}
\les
\|f_k\|_{X_k}.
\end{align}

\noi
Letting $g_k(n, \tau) = f_k(n, \tau + n^2)$, we have
\begin{equation}
\int_\R \|g_k(n, \tau)\|_{\l^2_n} d\tau \les\|f_k\|_{X_k}.
\label{Xk1a}
\end{equation}

\noi
Moreover, for  $k,\ell \in \Z_+$ and $f_k \in X_k$, we have
then
\begin{align}\label{Xk2}
\sum_{j=\ell+1}^\infty 2^\frac{j}{2}
\bigg\| \eta_j(\tau-n^2) \int_{\R} |f_k(n,\tau')|
\, 2^{-\l}(1+2^{-\l}|\tau-\tau'|)^{-4}d\tau'\bigg\|_{\l^2_n L^2_\tau } &\notag \\
+2^\frac{\ell}{2} \bigg\| \eta_{\leq \l}(\tau-n^2) \int_{\R} |f_k(n,\tau')|
\, 2^{-\l}(1+2^{-\l}|\tau-\tau'|)^{-4}d\tau'\bigg\|_{\l^2_n L^2_\tau }&\les
\|f_k\|_{X_k},
\end{align}

\noi
where the implicit constant is independent of $k$ and $\l$.
See \cite{Guo1} for the proof of \eqref{Xk2}.
In particular,
for $k,\ \l\in \Z_+$, $t_0\in \R$, $f_k \in X_k$ and
$\g\in \mathcal{S}(\R)$,
we have
\begin{align}\label{Xk3}
\big\| \F[\gamma(2^\l(t-t_0))\cdot \F^{-1}(f_k)]\big\|_{X_k}
\les \|f_k\|_{X_k}.
\end{align}

\noi
Note that the implicit constant in 
\eqref{Xk3}
is also independent of $k, \l$, and $t_0$.

At spatial frequencies $|n|\sim 2^k$,
we will use the $X^{s, b}$-structure given by the
$X_k$-norm, localized on the  time scale $\sim 2^{-[\alpha k]}$, where
$\alpha>0$ is to be determined later.
Here, $[x]$ denotes the integer part of $x$.
For $k\in \Z_+$ we define the
spaces $F_k^\al$ and $N^\al_k$ by
\begin{align*}
F_k^\al=\Big\{
& u \in L^2(\T\times\R): \ft{u }(n,\tau) \text{ is supported in }{I}_k\times\R, \\
& \text{ and }\|u\|_{F_k^\al}=\sup_{t_k\in \R}
\big\|\F[  \eta_0(2^{[\alpha k]}(t-t_k))\cdot u ]\big\|_{X_k}<\infty
\Big\},
\\
N_k^\al=\Big\{
& u \in L^2(\T\times\R): \ft{u}(n,\tau) \text{ is supported in }{I}_k\times\R, \\
& \text{ and }
\|u\|_{N_k^\al}=\sup_{t_k\in\R}
\big\|(\tau-n^2+i2^{[\alpha k]})^{-1}\F[ \eta_0(2^{[\alpha k]}(t-t_k)) \cdot u ]\big\|_{X_k}<\infty
\Big\}.
\end{align*}

\noi
Next, we  define local-in-time versions of these spaces
in the usual way.
 For $T\in (0,1]$,  we define the local-in-time spaces
 $F_k^\alpha(T)$ and $N_k^\alpha(T)$ by
\begin{align*}
F_k^\alpha(T)
& =\big \{ u \in C([-T,T];L^2(\T)):
\|u\|_{F_k^\al(T)}=\inf_{\wt{u}=u \text{ on } \T\times [-T,T]}\|\wt u \|_{F_k^\alpha}\big\}, \\
N_k^\alpha(T)
& =\big\{ u \in C([-T,T];L^2(\T)):
\|u\|_{N_k^\alpha(T)}=\inf_{\wt{u}= u \text{ on } \T\times [-T,T]}\|\wt u\|_{N_k^\alpha}\big\}.
\end{align*}

\noi
Here, the infimum is taken over all extensions $\wt u \in C_0(\R; L^2(\T))$.
So far, we have defined the dyadic function spaces.
We now assemble these dyadic spaces in a straight forward manner
using the Littlewood-Paley decomposition.
For $s\in \R$ and $T\in (0,1]$, we define the spaces
$F^{s,\alpha}(T)$ and $N^{s,\alpha}(T)$ by
\begin{align*}
F^{s,\alpha}(T)
& = \Big\{ u:\
\|u\|_{F^{s,\alpha}(T)}^2=\sum_{k=0}^{\infty}2^{2sk}\|\P_k u\|_{F_k^\alpha(T)}^2<\infty \Big\}, \\
N^{s,\alpha}(T)
& =\Big \{ u:\
\|u\|_{N^{s,\alpha}(T)}^2=\sum_{k=0}^{\infty}2^{2sk}\|\P_k u\|_{N_k^\alpha(T)}^2<\infty \Big\}.
\end{align*}

\noi
In order to deal with  these short-time spaces
$F^{s,\alpha}(T)$ and $N^{s,\alpha}(T)$, we need to
define the corresponding energy space.
For $s\in \R$ and $u\in
C([-T,T];H^{\infty}(\T))$, let
\begin{align*}
\|u\|_{E^{s}(T)}^2=\|\P_{0}u(0)\|_{L^2(\T)}^2
+\sum_{k\geq1}\sup_{t_k\in [-T,T]}2^{2sk}\|\P_k u(t_k)\|_{L^2(\T)}^2.
\end{align*}

\noi
Note that the energy space $E^s(T)$ is independent of the parameter $\alpha>0$.
This space is essentially the usual energy space $C([-T,T];H^s(\T))$ but
with a logarithmic difference.

\medskip

We conclude this section by stating basic embeddings involving the $F^{s, \al}$-spaces.
It follows immediately from \eqref{Xk3} that if $\alpha_1\geq \alpha_2$,
then we have $F^{s,\alpha_2}(T)\subset F^{s,\alpha_1}(T)$.

The following lemma shows that 
a smooth time cutoff supported on a interval of size $\sim 2^{-[\al k]}$
acts boundedly on $N^\al_k$.

\begin{lemma}\label{LEM:embed3}
Let $\al \geq 0$, $k\in \Z_+$, $t_k\in \R$, and $\g \in \mathcal{S}(\R)$.
Then, we have 
\begin{align}
\big\| (\tau - n^2 + i 2^{[\al k]})^{-1}\F [
\g(2^{[\al k]}(t-t_k))\cdot \F^{-1}(f_k)]\big\|_{X_k} 
\les \big\|(\tau - n^2 + i 2^{[\al k]})^{-1} f_k\big\|_{X_k}
\label{embed3}
\end{align}

\noi
for $f_k$ supported on $I_k \times \R$.
Here, the implicit constant is independent
of $\al, k,$ and $t_k$.
\end{lemma}

\begin{proof}
First, note that 
\begin{equation}\label{embed3a}
|\tau - n^2 + i 2^{[\al k]}|^{-1}
(1+2^{-[\al k]}|\tau - \tau'|)^{-1}
\les |\tau' - n^2 + i 2^{[\al k]}|^{-1}.
\end{equation}

\noi
Then, \eqref{embed3}
follows from \eqref{Xk2} and \eqref{embed3a}.
\end{proof}

The next lemma shows that the 
$F^{\al}_k$-norm controls
the $L^\infty_t L^2_x$-norm
of a dyadic piece.

\begin{lemma}\label{LEM:infty}
Let $u$ be a function on $\T\times \R$ such that
$\supp \ft u  \subset I_k \times \R$.
Then, we have
\begin{align}
\| u \|_{L^\infty_t L^2_x} \les \|u\|_{F^\al_k}
\label{infty1}
\end{align}

\noi
for any $\al > 0$.
Similarly, we have
\begin{align}
\big\| \F^{-1}[\eta_{\leq j}(\tau - n^2) \ft{u}] \big\|_{L^\infty_t L^2_x}
\les \|u\|_{F^\al_k}
\label{infty2}
\end{align}

\noi
for any $j \in \Z_+$.
Here, \eqref{infty2} also holds
when we replace $\eta_{\leq j}$ by $\eta_j$ or $\eta_{\geq j}$.
\end{lemma}

\begin{proof}
Let $t \in \R$. Then, by \eqref{Xk1},  
we have
\begin{align*}
\|u(x, t)\|_{L^2_x}
& = \sup_{t_k}\big\|\eta_0(2^{[\al k]} ( t - t_k))\cdot u(x, t)\big\|_{L^2_x}\\
& \les \sup_{t_k} \bigg\| \int \big|\F[\eta_0(2^{[\al k]} ( t - t_k))\cdot u](n, \tau)\big|d\tau\bigg\|_{\l^2_n}
\leq \|u\|_{F^\al_k}.
\end{align*}

\noi
The second estimate \eqref{infty2} follows from \eqref{infty1}
once we note that
\begin{equation*} 
\big\| \F^{-1}[\eta_{\leq j}(\tau - n^2) \ft{u}] \big\|_{L^\infty_t L^2_x}
\les \|u \|_{L^\infty_t L^2_x},
\end{equation*}

\noi
which follows from
$\F^{-1} [\eta_{\leq j} (\tau - n^2) ](t)
= 2^j \invft{\eta}_0(2^j t) e^{itn^2}$
and Young's inequality.
\end{proof}

As a corollary to Lemma \ref{LEM:infty}, we 
have the following control
of the 
$C([-T, T]; H^s)$-norm
by the $F^{s, \al}(T)$-norm.

\begin{lemma} \label{LEM:embed1}
Let $s\in \R$, $T\in (0,1]$, and $\alpha> 0$.
Then, we have
\begin{equation} \label{embed1}
\sup_{t\in [-T,T]}\|u(t)\|_{H^s} \les \|u\|_{F^{s,\alpha}(T)}.
\end{equation}

\end{lemma}

\begin{proof}
For $k \in \Z_+$, let 
$\wt{u}_k$  be an extension of $\P_ku$.
Then, 
by Lemma \ref{LEM:infty}, we have 
\begin{equation} \label{embed11}
\| \P_k u(t)\|_{L^2_x}
= \| \wt u_k(t)\|_{L^2_x}
\les \|\wt u_k\|_{F^\al_k}
\end{equation}

\noi
for $t \in [-T, T]$.
Then, \eqref{embed1} follows from \eqref{embed11}
by taking an infimum 
over  extensions
$\wt{u}_k$  of $\P_k(u)$, 
summing over dyadic blocks, 
and taking a supremum in $t \in [-T,T]$.
\end{proof}

So far, we defined the function spaces with
the modulation regularity $\frac{1}{2}$.
In the following, we define the corresponding function spaces
with the modulation regularity $b$.
For $k \in \Z_+$ and $b \in \R$,
let $X_k^b$ denote the dyadic $X^{s, b}$-type space
analogous to $X_k$,
whose norm is given by
\begin{align*}
 \|f_k \|_{X_k^b}:=\sum_{j=0}^\infty
2^{jb} \|\eta_j(\tau-n^2) f_k(n ,\tau)\|_{\ell^2_n L^2_{\tau}}
\end{align*}

\noi
for $f_k $ supported on $I_k \times \R$.
By definition, we have $X_k = X^\frac{1}{2}_k$.
Then, we define the spaces $F^{ b, \al}_k$ and $F^{s, b,  \al}(T)$
with $X_k^b$,
just as we defined $F^\al_k$ and $F^{s,  \al}(T)$ with $X_k$.

The following lemma shows that
we obtain a small power of time localization
at a slight expense of the regularity in modulation.
\begin{lemma} \label{LEM:timedecay}
Let $\al, T > 0$ and $b < \frac{1}{2}$.
Then, we have \begin{align*}
\|\P_k u \|_{F^{b, \al}_k}
\les T^{\frac{1}{2} - b -} \|\P_k u \|_{F^{\al}_k}
\end{align*}

\noi
for any function $u$ supported on $\T\times [-T, T].$
\end{lemma}

\begin{proof}
Let $\chi(t)$ be the  characteristic function on $[-1, 1]$
and $\chi_{_T}(t)  = \chi(T^{-1} t)$, i.e. $\chi_{_T}$ is the characteristic function of $[-T, T]$.
Note that 
\begin{align}
\|\ft \chi_{_T}\|_{L^q} \sim T^{1-\frac{1}{q}}
\label{time0}
\end{align}

\noi
for $ q > 1$.
For fixed $t_k\in \R$, let $v_k = \eta_0(2^{[\al k]} (t -t_k)) \cdot \P_k(u)$.
Then, we have
$v_k = \chi_{T}\cdot v_k$
and it suffices to show
\begin{align} \label{time1}
\big\|\F[\chi_{_T}\cdot  v] \big\|_{X^b_k} 
\les T^{\frac{1}{2} - b -} \|\F(v_k)\|_{X_k}.
\end{align}

In the following, we simply use $v$ to denote $v_k$.
Write $v = \sum_{j'\in\Z_+} v_{j'}$, where 
$v_{j'} = \F^{-1}[\eta_{j'} (\tau' - n^2) \ft v]$.
For any $\eps > 0$, we have 
\begin{align} 
\big\|\F[\chi_{_T}\cdot  v] \big\|_{X^b_k} 
\les \sup_{j\in\Z_+} 2^{(b+\eps)j}
\big\|\eta_j (\tau - n^2)\F[\chi_{_T}\cdot  v]  \big\|_{\l^2_n L^2_\tau}.  \label{time2}
\end{align}

\noi
Fix $j \in \Z_+$.
First, we consider the contribution from $j' \geq j - 5$.
By H\"older's  inequality (with $\theta = \frac{1}{2} - b-\eps>0$)
and Young's inequality with \eqref{time0}, we have 
\begin{align*}
\eqref{time2} 
& \les \sup_{j\in\Z_+} 2^{\frac{j}{2}}
\sum_{j' }
\big\|\eta_j (\tau - n^2)\F[\chi_{_T}\cdot  v_{j'}]  \big\|_{\l^2_n L^\frac{2}{1-2\theta}_\tau}
 \les  \sup_{j\in\Z_+}\sum_{j' \geq j - 5 } 2^{\frac{j'}{2}}
\big\|\F[\chi_{_T}\cdot  v_{j'}]  \big\|_{\l^2_n L^\frac{2}{1-2\theta}_\tau}  \notag \\
& \les  T^\theta
\sup_{j\in\Z_+}\sum_{j' \geq j - 5 } 
2^{\frac{j'}{2}}
\|\ft v_{j'}  \|_{\l^2_n L^2_\tau}  
\leq T^{\frac{1}{2}-b-\eps} \|\F(v)\|_{X_k}.
\end{align*}

\noi
Next, we consider the contribution from $j' < j - 5$.
With $\F(\chi_{_T}\cdot  v)(n, \tau) =
\int \ft \chi_{_T}(\tau - \tau') \ft v (n, \tau')d\tau'$,
we have $|\tau - \tau'|\sim 2^j$ in this case.
Then, by Young's inequality with \eqref{time0}, we have 
\begin{align*} 
\eqref{time2}
& \les \sup_{j\in\Z_+} 
\bigg\| \eta_j(\tau - n^2) \int_{|\tau - \tau'|\sim 2^j}
| \tau - \tau'|^{b + \eps}  |\ft \chi_{_T}(\tau - \tau')| |\ft v (n, \tau')|d\tau' \bigg\|_{\l^2_n L^2_\tau} \notag   \\
& \les \sup_{j\in\Z_+} 
\big\||\tau |^{b + \eps}  \ft \chi_{_T}\big\|_{L^2_{|\tau|\sim 2^j} }
\sum_{j'  < j - 5}
\|\ft v_{j'}  \|_{\l^2_n L^1_\tau}  
\les T^{\frac{1}{2}-b-\eps} \|\F(v)\|_{X_k}.
\end{align*}

\noi
This proves \eqref{time1}.
\end{proof}

As a corollary to the proof of Lemma \ref{LEM:timedecay}, 
we obtain the following lemma.
\begin{lemma} \label{LEM:sup}
Let $k \in \Z_+$.
Then, for any interval $I = [t_1, t_2] \subset \R$, we have 
\begin{align*}
 \sup_{j\in\Z_+} 2^\frac{j}{2}
\big\|\eta_j (\tau - n^2)\F[\ind_{I}(t) \cdot  u ]  \big\|_{\l^2_n L^2_\tau}
\les \| \F(u)\|_{X_k}, 
\end{align*} 

\noi
where the implicit constant is independent of $k$ and $I$.
\end{lemma}

\begin{proof}
Set $b + \eps = \frac{1}{2}$ in \eqref{time2}.
\end{proof}

As in \cite{IKT}, for any $k\in \Z_+$ we define the set $S^\al_k$ of
$k$-acceptable time multiplication factors:
\begin{align}
S_k^\alpha =\big\{m_k:\R \rightarrow \R:
\|m_k\|_{S_k^\alpha}= \sum_{j=0}^{10} 2^{-j[\al k]}
\|\dd^j m_k\|_{L^\infty}< \infty\big\} \label{Sk}.
\end{align}

\begin{lemma} \label{LEM:Sk}
Let $\al > 0$ and $m_k$ be a $k$-acceptable time multiplication factor.
Then, we have
\begin{align}
\| m_k(t) u_k \|_{F^\al_k} & \les \|m_k\|_{S^\al_k} \| u_k \|_{F^\al_k}, \label{Sk1}\\
\| m_k(t) u_k \|_{N^\al_k} & \les \|m_k\|_{S^\al_k} \| u_k \|_{N^\al_k} \label{Sk2}
\end{align}
\end{lemma}

\noi
for any function
 $u_k$ on $\T\times \R$ such that
$\supp \ft u_k \subset I_k \times \R$.

\begin{proof}
Fix $t_k \in \R$.
Let $\g: \R \to [0, 1] $ be a smooth cutoff function supported on $[-2, 2]$
such that $\g \equiv 1$ on $[-\frac{8}{5}, \frac{8}{5}]$.
Then, we have
$\eta_0(2^{[\al k]}(t - t_k))\cdot m_k u_k
= \g(2^{[\al k]}(t - t_k))\cdot m_k v_k$,
where $v_k(t) = \eta_0(2^{[\al k]}(t - t_k))\cdot u_k(t)$.
Integrating by parts several times, we have
\begin{align*}
\big|\F [\g (2^{[\al k]}(\cdot - t_k))\cdot m_k] (\tau)\big|
\les 2^{-[\al k]} (1 + 2^{-[\al k]}|\tau|)^{-4} \|m_k\|_{S^\al_k}.
\end{align*}

\noi
Then,  \eqref{Sk1} follows from  \eqref{Xk2}.
The second estimate \eqref{Sk2} follows
from a similar argument 
with \eqref{embed3a}.
This completes the proof of Lemma \ref{LEM:Sk}.
\end{proof}

The following lemma shows a kind of almost  orthogonality property of $X_k$.
Due to the $\l^1$-based Besov nature of the temporal regularity in $X_k$,
we  have the second terms on the right-hand sides in \eqref{a.o.} and \eqref{aa.o.}.
See Lemma 6.4 in \cite{CCT3}
for a related result.

\begin{lemma}[Almost orthogonality] \label{LEM:a.o.}
Let $\g:\R\rightarrow [0,1]$ be  a smooth function  supported on
$[-1,1]$ such that
\[\sum_{m \in \Z}\gamma(t-m)\equiv 1\] 

\noi
for all $ t\in \R$.

\smallskip

\noi
\textup{(a)} There exists $C> 0$ such that
\begin{align}
\|\F(u)\|_{X_k}
\leq & \  C \bigg(\sum_{m \in \Z}\big\|\F[\g(\ld t-m) \cdot u ]\big\|_{X_k}^2\bigg)^\frac{1}{2} \notag \\
& + C
\sum_{j=0}^\infty
\bigg( \sum_{m\in \Z} 2^j \big \| \eta_j(\tau - n^2)
\F[\g(\ld t-m) \cdot u ](n, \tau)
\big\|_{\l^2_n L^2_\tau}^2\bigg)^\frac{1}{2}
\label{a.o.}
\end{align}

\noi
for all $\ld \geq 1$.

\medskip

\noi
\textup{(b)} There exists $C> 0$ such that
\begin{align}
\big\| \eta_{\geq \l } & (\tau -   n^2) (\tau - n^2 +  i A)^{-1}\F(u)\big\|_{X_k}
\leq
C \bigg(\sum_{m \in \Z}\big\| (\tau - n^2 +i  A)^{-1}\F[\g(\ld t-m) \cdot u ]
\big\|_{X_k}^2\bigg)^\frac{1}{2} \notag \\
& + C
\sum_{j=0}^\infty
\bigg( \sum_{m\in \Z} 2^j \big \| \eta_j(\tau - n^2)
(\tau - n^2 + i A)^{-1}
\F[\g(\ld t-m) \cdot u ](n, \tau)
\big\|_{\l^2_n L^2_\tau}^2\bigg)^\frac{1}{2}
\label{aa.o.}
\end{align}

\noi
for all $A, \ld \geq 1$,
where $\l$ is the greatest integer such that $2^\l \leq \ld$.

\end{lemma}

\begin{remark}\rm

In view of Minkowski's integral inequality: $\l^1_j \l^2_m \subset \l^2_m \l^1_j$, 
we see that the first terms on the right-hand sides
of \eqref{a.o.} and \eqref{aa.o.} are controlled by the second terms
and thus are not needed.
We, however, keep the first terms 
on the right-hand sides
of \eqref{a.o.} and \eqref{aa.o.} 
as we can estimate many cases 
simply by them.
Namely, the second terms are the correction terms needed
to handle a few special cases.

\end{remark}

We present the proof of Lemma \ref{LEM:a.o.}
at the end of this section.
In the following, we first discuss corollaries to Lemma \ref{LEM:a.o.}.

\begin{corollary} \label{COR:a.o.}
Let $\g$ be as in Lemma \ref{LEM:a.o.}.
Then, for $b < \frac{1}{2}$,
there exists $C = C(b) > 0$ such that
for all functions $u$  
with $\supp \ft{u}\subset I_k\times \R$, we have 
\begin{align*}
\|\F(u)\|_{X_k^b}
\leq & \  C \bigg(\sum_{m \in \Z}\big\|\F[\g(\ld t-m) \cdot u ]\big\|_{X_k}^2\bigg)^\frac{1}{2}
\end{align*}

\noi
for all $\ld \geq 1$.

\end{corollary}

\noi
The proof of Corollary \ref{COR:a.o.} is immediate from \eqref{a.o.} and Cauchy-Schwarz inequality.
The next lemma shows a relation
between  the $F^{s, \al}$-spaces at different time scales.
Once again, a slight loss of the regularity in modulation
is necessary due to the $\l^1$-based Besov nature of the $X_k$-norm.

\begin{lemma}
\label{LEM:embed2}
Suppose that  $\alpha\geq \beta\geq 0$, $s,r\in \R$ with
$r\leq \frac{\beta-\alpha}{2}+s$.
Then, for $b<\frac{1}{2}$,
there exists $C = C(b) > 0$ such that
\begin{equation}\label{embed2}
\|u\|_{F^{r, b, \beta}(T)}\leq C \|u\|_{F^{s,\alpha}(T)}
\end{equation}

\noi
for all functions $u$ on $\T\times \R$ and $T \in (0, 1]$.

\end{lemma}

\begin{proof} 

From the definition of $F^{s,\alpha}$, it suffices to prove
\begin{align*} 
2^{kr}\|u\|_{F_{k}^{b, \beta}}\les  2^{ks}\|u\|_{F_k^\alpha}
\end{align*}

\noi
for $k \in \Z_+$.
From the support conditions for $\eta_0$ and $\g$, we have
\begin{align*}
 \eta_0(2^{[\beta k]}(t-t_k))\cdot u (t)
 =
 \sum_{|m|\leq C2^{k(\alpha-\beta)}}
\g(2^{[\alpha k]}(t-t_k)-m)\cdot \eta_0(2^{[\beta k]}(t-t_k))\cdot u(t)
 \end{align*}

\noi
for some $C>0$.
Then,
by Corollary \ref{COR:a.o.} and \eqref{Xk3}, we have
\begin{align*}
\|u\|_{F_k^{b, \beta}}
& =\sup_{t_k\in \R}\big\|\F[ \eta_0(2^{[\beta k]}(t-t_k))\cdot u]\big\|_{X_k^b}\\
& \leq C(b) \sup_{t_k\in \R}
\bigg(\sum_{|m|\leq C2^{k(\alpha-\beta)}}
\big\|
\F[ \g(2^{[\alpha k]}(t-t_k)-m)\cdot \eta_0(2^{[\beta k]}(t-t_k))\cdot u ]\big\|_{X_k}\bigg)^\frac{1}{2}\\
& \les  2^{\frac{k(\alpha-\beta)}{2}}\|u\|_{F_k^\alpha}.
\end{align*}

\noi
Hence, \eqref{embed2} follows as long as
$r\leq \frac{\beta-\alpha}{2}+s$.
\end{proof}

We conclude this section by presenting the proof of Lemma \ref{LEM:a.o.}.

\begin{proof} [Proof of Lemma \ref{LEM:a.o.}]
(a)
Let $u_m  =\g(\ld t-m  )\cdot u$. Then, we have
\begin{align}\label{a.o.1}
\|\F(u)\|_{X_k}
=\sum_{j =0}^\infty 2^\frac{j}{2}\Big\|\eta_j (\tau)\sum_{m \in\Z}
\ft{u}_m (n,\tau+n^2)\Big\|_{\l^2_n L^2_{\tau}}.
\end{align}

\noi
Let $v_m  (n,\tau)=\ft{u}_m  (n,\tau+ n^2)$.
Since $\eta_0 \equiv 1$ on the support of $\g$, we have
$u_m   = \eta_0 (\lambda t-m )\cdot u_m   $
and thus
 \[v_m =\int
v_m (n, s)e^{ - i\frac{(\tau-s)m }{\lambda} }
\ld^{-1}\ft{\eta}_0\Big(\frac{\tau-s}{\lambda}\Big)d s.\]

\noi
By Cauchy-Schwarz inequality, we have
\begin{align}\label{a.o.2}
\Big\|\eta_j (\tau)  \sum_{m \in\Z} & \ft{u}_m(n,\tau  +n^2)\Big\|_{\l^2_nL^2_{\tau}}^2 \notag \\
& =\sum_{m,m'\in \Z}
\sum_{n\in \Z} \int
\eta_j^2(\tau)\, v_m(n,s)e^{- i\frac{(\tau-s)m}{\lambda}}
\lambda^{-1}\ft{\eta}_0\Big(\frac{\tau-s}{\lambda}\Big)\notag\\
&\hphantom{XXXXXXXll|} \times \cj{v_{m'}}(n,s')
e^{i\frac{(\tau-s')m'}\lambda}\lambda^{-1}\cj{\ft{\eta}_0}\Big(\frac{\tau-s'}{\ld}\Big)
ds ds'  d\tau\notag\\
& \les \sum_{m,m'\in \Z} \int
\|v_m(\cdot,s)\|_{\l^2_n}\|v_{m'}(\cdot,s')\|_{\l^2_n}|K_\lambda(s,s',m-m')|dsds',
\end{align}

\noi
where
\begin{align}
K_\ld(s,s',m-m')
& =\ld^{-2}\int\eta_j^2(\tau)
e^{- i\frac{\tau(m-m')}{\ld}}
\ft{\eta}_0\Big(\frac{\tau-s}{\lambda}\Big)
\cj{\ft{\eta}_0}\Big(\frac{\tau-s'}{\ld}\Big)d\tau \label{a.o.30}\\
& =\ld^{-1}\int\eta_j^2(\lambda\zeta)
e^{- i\zeta(m-m')}\ft{\eta}_0({\zeta-\lambda^{-1}s})\cj{\ft{\eta}_0}({\zeta-\lambda^{-1}s'})d\zeta.
\label{a.o.3}
\end{align}

\medskip
\noi
$\bullet$ {\bf Case (i):}
First,  we consider the case
$\lambda\ges  2^j $.

Suppose that  $|m - m'| \les \frac{\ld}{2^j}$.
From \eqref{a.o.3}, we have $|K_\ld| \lesssim \frac{ 2^j }{\ld^2}$.
Then, using \eqref{Xk1a}, we estimate \eqref{a.o.2}  by
\begin{align*}
\sum_m  \sum_{m': |m - m'| \les \frac{\ld}{2^j}}
\frac{2^j}{\ld^2}
\bigg(\int \|v_m(\cdot,s)\|_{\l^2_n} ds \bigg)
\bigg(\int \|v_{m'}(\cdot,s')\|_{\l^2_n}ds'\bigg)
\les \ld^{-1} \sum_m \|\ft{u}_m\|_{X_k}^2.
\end{align*}

\noi
Hence, we obtain
\begin{align*}
\eqref{a.o.1}
\les \sum_{ 2^j \les \ld } 2^\frac{j}{2}
\bigg(\ld^{-1} \sum_m \|\ft{u}_m\|_{X_k}^2\bigg)^\frac{1}{2}
\sim \bigg(\sum_{m \in \Z}\big\|\F[\g(\ld t-m) \cdot u ]\big\|_{X_k}^2\bigg)^\frac{1}{2}.
\end{align*}

When $|m - m'| \gg \frac{\ld}{2^j}$,
we  integrate \eqref{a.o.3} by parts twice and obtain
$|K_\ld| \les 2^{-j}|m-m'|^{-2}$.
In this case, \eqref{a.o.2} is estimated by
\begin{align*}
2^{-j } \sum_m \sum_{m': |m - m'| \gg \frac{\ld}{2^j}}
\frac{1}{|m-m'|^2}  \|\ft{u}_m\|_{X_k}\|\ft{u}_{m'}\|_{X_k}
\les \ld^{-1} \sum_m \|\ft{u}_m\|_{X_k}^2.
\end{align*}

\noi
Hence, the same conclusion holds as before.

\medskip
\noi
$\bullet$ {\bf Case (ii):}
Next,  we consider the case $\ld\ll 2^j$.

First, we consider the contribution from $m = m'$.
In this case, the contribution to \eqref{a.o.1}
is bounded by
\begin{align}\label{a.o.4}
\sum_{2^j \gg \ld }2^\frac{j}{2}
\bigg( \sum_{m}
\bigg\| \eta_j(\tau)
\int \| v_m(n,s)\|_{\l^2_n}
\lambda^{-1}\ft{\eta}_0\Big(\frac{\tau-s}{\lambda}\Big)
ds \bigg\|_{L^2_\tau}^2 \bigg)^\frac{1}{2}.
\end{align}

\noi
When $|s| \ll |\tau|$ or $|s|\gg |\tau|$,  we have $|\tau - s| \ges 2^j$.
In this case, we have
$|\eta_0(\ld^{-1}(\tau - s))|
\les \ld^2 2^{-2j} \phi(\ld^{-1}(\tau - s))$
for some $\phi \in \S$,
and thus we have
\begin{align*}
\eqref{a.o.4}
\les \sum_{2^j \gg \ld }2^\frac{j}{2}
\frac{\ld}{2^{\frac{3}{2}j}}
\bigg\{ \sum_{m}
\bigg(
\int \| v_m(n,s)\|_{\l^2_n} ds \bigg)^2 \bigg\}^\frac{1}{2}
\les
\bigg( \sum_m \|\ft{u}_m\|_{X_k}^2\bigg)^\frac{1}{2}.
\end{align*}

\noi
When $|s|\sim |\tau|\sim 2^j$, it follows from  Young's inequality that
\begin{align*}
\eqref{a.o.4}
& \les \sum_{2^j \gg \ld }2^\frac{j}{2}
\bigg( \sum_{m}
\|\eta_j (s) v_m (n, s) \|_{\l^2_n L^2_s}^2 \bigg)^\frac{1}{2}\\
& \les
\sum_j
\bigg( \sum_m 2^j \| \eta_j(\tau - n^2) \ft{u}_m (n, \tau)
\|_{\l^2_n L^2_\tau}^2\bigg)^\frac{1}{2}.
\end{align*}

Next, we consider the contribution from $m \ne m'$.

\medskip

\noi
$\circ$ Subcase (ii.1): $|m -m'|\gg \frac{2^j}{\ld}$.
In this case, integrating by parts three times, we have
$|K_\ld| \les \ld^{-1} |m -m'|^{-3}$.
The contribution to  \eqref{a.o.2} is estimated by
\begin{align*}
\sum_m \sum_{m': |m - m'| \gg \frac{2^j}{\ld}}
\frac{1}{\ld |m - m'|^3}
\bigg(\int \|v_m(n,s)\|_{\l^2_n} ds \bigg)
\bigg(\int \|v_{m'}(n,s')\|_{\l^2_n}ds'\bigg)
\les \frac{\ld}{2^{2j}}  \sum_m \|\ft{u}_m\|_{X_k}^2.
\end{align*}

\noi
Hence, we obtain
\begin{align*}
\eqref{a.o.1}
\les \sum_{ 2^j \gg \ld } 2^\frac{j}{2}
\bigg(\frac{\ld}{2^{2j}} \sum_m \|\ft{u}_m\|_{X_k}^2\bigg)^\frac{1}{2}
\sim \bigg(\sum_{m \in \Z}\big\|\F[\g(\ld t-m) \cdot u ]\big\|_{X_k}^2\bigg)^\frac{1}{2}.
\end{align*}

\medskip
\noi
$\circ$ Subcase (ii.2): $|m - m'|\les \frac{2^j}{\ld}$.

When $|s| \ll |\tau|$ or $|s|\gg |\tau|$,  we have $|\tau - s| \ges 2^j$.
In this case, we have
$|K_\ld| \les \frac{\ld^2}{ 2^{3j}}$ from \eqref{a.o.30}.
The contribution to  \eqref{a.o.2} is estimated by
\begin{align*}
\sum_m \sum_{m': 0< |m - m'| \les \frac{2^j}{\ld}}
\frac{\ld^2}{2^{3j}}
\bigg(\int \|v_m(n,s)\|_{\l^2_n} ds \bigg)
\bigg(\int \|v_{m'}(n,s')\|_{\l^2_n}ds'\bigg)
\les \frac{\ld}{2^{2j}}  \sum_m \|\ft{u}_m\|_{X_k}^2.
\end{align*}

\noi
Then, the result follows as in Subcase (ii.1).
A similar argument holds when 
$|s'| \ll |\tau|$ or $|s'|\gg |\tau|$.

Lastly, when $|s|\sim |s'| \sim |\tau|\sim2^j$,
we integrate \eqref{a.o.3} by parts four  times
and obtain
\[|K_\ld| \les   |m- m'|^{-4}  \ld^{-1}(1+ \ld^{-1}|s - s'|)^{-4}\]

\noi
By Young's inequality, the contribution to  \eqref{a.o.2} is estimated by
\begin{align*}
\sum_m & \sum_{m': 0< |m - m'| \les \frac{2^j}{\ld}}
\frac{1}{ |m - m'|^4}
 \|\eta_j(s) v_m(n,s)\|_{\l^2_nL^2_s} \\
& \hphantom{XXXX}  \times
\bigg\|\int \eta_j(s')\|v_{m'}(n,s')\|_{\l^2_n}
\ld^{-1}(1+ \ld^{-1}|s - s'|)^{-4}ds'\bigg\|_{L^2_s}\\
& \les
\sum_m
 \|\eta_j(s) v_m(n,\tau)\|_{\l^2_nL^2_\tau}^2 .
\end{align*}

\noi
Hence, we obtain
\begin{align*}
\eqref{a.o.1}  \les
\sum_j
\bigg( \sum_m 2^j \| \eta_j(\tau - n^2) \ft{u}_m (n, \tau)
\|_{\l^2_n L^2_\tau}^2\bigg)^\frac{1}{2}.
\end{align*}

\medskip

\noi
(b) Proceeding as in Part (a), we have
\begin{align*}
\big\|\eta_{\geq \l}(\tau-n^2)(\tau - n^2 + iA)^{-1} \F(u)\big\|_{X_k}
=\sum_{j=\l}^\infty 2^\frac{j}{2}\bigg\|\eta_j(\tau)\sum_{m \in\Z}
\frac{\ft{u}_m (n,\tau+n^2)}{|\tau| + A} \bigg\|_{\l^2_n L^2_{\tau}}
\end{align*}

\noi
and
\begin{align*}
\bigg\|\eta_j(\tau)  \sum_{m\in\Z}
& \frac{\ft{u}_m (n,\tau+n^2)}{|\tau| + A} \bigg\|_{\l^2_nL^2_{\tau}}^2 \notag \\
& \les \sum_{m,m'\in \Z} \int
\|v_m(\cdot,s)\|_{\l^2_n}\|v_{m'}(\cdot,s')\|_{\l^2_n}|\wt{K}_\lambda(s,s',m-m')|dsds',
\end{align*}

\noi
where
\begin{align*}
\wt{K}_\ld(s,s',m-m')
& =\ld^{-2}\int \frac{\eta_j^2(\tau)}{(|\tau| + A)^2}
e^{- i\frac{\tau(m-m')}{\ld}}
\ft{\eta}_0\Big(\frac{\tau-s}{\lambda}\Big)
\cj{\ft{\eta}_0}\Big(\frac{\tau-s'}{\ld}\Big)d\tau .
\end{align*}

When $|\tau| + A \ges \max (|s|, |s'|)$,
we have
\[ \frac{1}{(|\tau| + A)^2 } \les  \frac{1}{|s|+A} \frac{1}{|s'|+A}, \]

\noi
and thus we can proceed as in Part (a).
Next, we consider the case  $|\tau| + A \ll \max (|s|, |s'|)$.
Without loss of generality, assume $|s| \geq |s'|$.
In this case, we have $|\tau - s| \sim |s| \sim |s| + A$.
Thus,
\begin{equation*} 
 \bigg|\ft{\eta}_0\Big(\frac{\tau-s}{\lambda}\Big) \bigg|
\les \frac{\ld}{|s|+A} \, \phi\Big(\frac{\tau-s}{\lambda}\Big)
\end{equation*}

\noi
for some $\phi \in \S$.
In particular,
since we have $\ld  \sim 2^\l \leq  2^j $, we have
\begin{equation} \label{aa.o.4}
\frac{1}{|\tau| + A} \frac{\ld} {|s|+A}
\les \frac{1}{|s|+A}.
\end{equation}

\noi
If $|s'| \gg |\tau|+A$,
then we use \eqref{aa.o.4} with  $|s'|$ in place of $|s|$.
If $|s'| \les |\tau|+A$,
we use
$ (|\tau| + A)^{-1}  \les  (|s'|+A)^{-1}$.
Then, we can proceed as in Case (ii) of Part (a) in this case.
This completes the proof of Lemma \ref{LEM:a.o.}.
\end{proof}

\section{Linear estimate}
\label{SEC:linear}

In this section, we present a linear estimate
associated to the Schr\"odinger equation.
In particular, we estimate solutions
to the nonhomogeneous linear Schr\"odinger equation.

\begin{proposition} \label{PROP:linear}
Let  $T\in (0,1]$.
Suppose that  $u \in C([-T,T];H^\infty(\T))$
is a solution to the following nonhomogeneous linear Schr\"odinger equation:
\begin{align*}
i\dt u-\dx^2 u=v \qquad \text{on } \T\times (-T,T),
\end{align*}

\noi
where $v \in C([-T,T];H^\infty(\T))$.
Then, for any $s\in \R$  and $\alpha\geq 0$,
we have
\begin{align}\label{linear}
\|u\|_{F^{s,\alpha}(T)}\les \
\|u\|_{E^{s}(T)}+\|v\|_{N^{s,\alpha}(T)}.
\end{align}
\end{proposition}

\begin{proof}
We follow the argument in \cite{IKT, Guo1}.
We first make some preliminary computations. Given $\phi_k \in L^2(\T)$
with $\supp \ft{\phi}_k \subset I_k$, we have
\begin{align}\label{linear1}
\big\|\F \big[ \eta_0(2^{[\al k] }(t-t_k) ) \cdot S(t - t^*) \phi_k\big] \big\|_{X_k}\les \|\phi_k\|_{L^2}
\end{align}

\noi
for any $\al \geq 0$ and  $t_k, t^* \in \R$.
This easily follows from
\[\F\big[\eta_0(2^{[\al k] }(t-t_k)) \cdot S(t - t^*) \phi_k\big ](n, \tau)
= e^{-it_k (\tau - n^2) } 2^{-[\al k]}\ft{\eta}_0(2^{-[\al k]}(\tau-n^2)) e^{-it^*n^2}\ft{\phi}_k(n).\]

\noi
Moreover,
a straightforward computation shows
\begin{align*}
\F\Big[ \eta_0  & (2^{[\alpha k]}t)  \int_0^t S(t - t')v_k(t') dt' \Big](n ,\tau)\\
& =
C\int \ft{v}_k(n ,\tau')\cdot 
2^{-[\alpha k]}\frac{\ft{\eta}_0(2^{-[\alpha k]}(\tau-\tau'))-
\ft{\eta}_0(2^{-[\al k]}(\tau-n^2))}{\tau'-n^2}d\tau'
\end{align*}

\noi
and
\begin{align*}
& \bigg| 2^{-[\alpha k]}\frac{    \ft{\eta}_0(2^{-[\al k]}(\tau-\tau'))-
\ft{\eta}_0(2^{-[\al k]}(\tau-n^2))}{\tau'-n^2}\cdot
(\tau'-n ^2+i2^{[\alpha k]})\bigg|\\
&\hphantom{XXXXXXXX}
\les \ 2^{-[\alpha k]}(1+2^{-[\al k]}|\tau-\tau'|)^{-4}+2^{-[\alpha k]}(1+2^{-[\al k]}|\tau-n ^2|)^{-4}.
\end{align*}

\noi
Hence, from \eqref{Xk1} and \eqref{Xk2}, we have
\begin{align} \label{linear2}
\bigg\|\F\Big[ \eta_0(2^{[\alpha k]} t) & \int_0^t S(t - t')v_k(t') dt' \Big]\bigg\|_{X_k}
\lesssim \big\|(\tau - n^2 + i 2^{[\al k]})^{-1} \cdot \F(v_k)\big\|_{X_k}
\end{align}

\noi
for $v_k$ on $\T\times \R$ with $\supp \ft{v}_k \subset I_k \times \R$.

\medskip
In view of the definitions, the square of the right-hand side of \eqref{linear} is equivalent to
\begin{align*}
\|\P_0u(0)\|_{L^2}^2+\|\P_0v\|_{N_0^\alpha(T)}^2
+\sum_{k\geq 1}\Big(\sup_{t_k\in [-T,T]}2^{2sk}\|\P_ku(t_k)\|_{L^2}^2
+2^{2sk}\|\P_kv\|_{N_k^\alpha(T)}^2\Big).
\end{align*}

\noi
Thus, it suffices to prove
\begin{align}\label{linear3}
\|\P_k u\|_{F_k^\alpha(T)} \les
\begin{cases}
\vphantom{\bigg|}
\|\P_0u(0)\|_{L^2}+\|\P_0(v)\|_{N_0^\alpha(T)}, & k = 0,\\
\displaystyle \sup_{t_k\in [-T,T]}\|\P_ku(t_k)\|_{L^2}
+\|\P_k v\|_{N_k^\alpha(T)}, & k \geq 1.
\end{cases}
\end{align}

For $k \in \Z_+$, let  $\wt v_k$ denote an extension of $\P_kv$ such that
\begin{align}
\|\wt v_k\|_{N_k^\alpha}\leq C\|\P_kv\|_{N_k^\alpha(T)}.
\label{linear3a}
\end{align}

\noi
By Lemma \ref{LEM:Sk},
we may assume that $\wt v_k$ is supported on
$\T\times [-T-2^{-[\al k]-10},T+2^{-[\alpha k]-10}]$.
Indeed, let $m(t)$ be a smooth cutoff function such that
\[m(t)=\begin{cases}
1, & \text{for }t\geq 1,\\
0, & \text{for } t\leq 0.
\end{cases}
\]

\noi
Then,  defining $m_{k, -} $ and $m_{k, +}$ by  
\begin{align*}
m_{k, -} (t) & = m(2^{[\alpha k]+10}(t+T+2^{-[\alpha k]-10})),\\
m_{k, +} (t) & = m(-2^{[\alpha k]+10}(t-T-2^{-[\alpha k]-10})),
\end{align*}

\noi
we have
$m_{k, -}, m_{k, +} \in S_k^\alpha$, where $S_k^\al$ is defined in \eqref{Sk}.
Note that 
$m_k(t) := m_{k, -}(t) \cdot m_{k, +}(t) $
is supported on $[-T-2^{-[\alpha k]-10},T+2^{-[\alpha k]-10}]$, and
is equal to $1$ on $[-T,T]$.
Since $\wt{v}_k$ is an extension of $\P_k v$,
$m_k \cdot \wt{v}_k$ is also an extension of $\P_k v$.
Moreover, from \eqref{Sk2}, we have
$\| m_k\cdot \wt v_k\|_{N_k^\alpha}
\les  \|\wt v_k\|_{N_k^\alpha}\leq C\|\P_k v\|_{N_k^\alpha(T)}$.
Hence, we assume that
$\wt v_k$ is supported on
$\T\times [-T-2^{-[\al k]-10},T+2^{-[\alpha k]-10}]$
in the following.

Next, we define an extension $\wt u_k$ of $\P_ku$.
For $k \geq 1$,
we define $\wt u_k$  by setting
\[\wt u_k (t)
= \begin{cases}
\displaystyle  \eta_0(2^{[\alpha k]+5}(t-T))
\Big[S(t-T)\P_ku(T)-i\int_T^tS(t-t') \wt v_k(t') dt'\Big],
 &
t \geq T, \\
\displaystyle
\eta_0(2^{[\alpha k]+5}(t+T))
\Big[S(t+T)\P_ku(-T) - i \int_{-T}^tS(t-t')\wt v_k (t')dt'\Big],
&  t \leq -T,
\end{cases}\]

\noi
and $\wt u_k (t)  = \P_ku (t)$ for $t \in [-T, T]$.
When $k = 0$, define $\wt u_0$ by
\[\wt u_0 (t)
= \begin{cases}
\displaystyle  \eta_0(2^{5}(t-T))
\Big[S(t)\P_0u(0)-i\int_0^tS(t-t') \wt v_0(t') dt'\Big],
 &
t \geq T, \\
\displaystyle
\eta_0(2^{5}(t+T))
\Big[S(t)\P_0u(0) - i \int_{0}^tS(t-t')\wt v_0 (t')dt'\Big],
&  t \leq -T,
\end{cases}\]

\noi
and $\wt u_0 (t)  = \P_0u (t)$ for $t \in [-T, T]$.

Using  \eqref{Xk3}, we can show that
\begin{align}\label{linear4}
\|\P_ku \|_{F^\al_k(T)}\les \sup_{t_k\in [-T,T]}\big\|\F[\eta_0(2^{[\alpha k]}(t-t_k)) \cdot \wt u_k ]\big\|_{X_k}.
\end{align}

\noi
Namely, with this choice of $\wt u_k$,
we can restrict the domain of the supremum
from $t_k \in \R$ to $t_k \in [-T, T]$.
In the following, we first prove \eqref{linear3}, assuming \eqref{linear4}.
We then present the proof of \eqref{linear4} at the end.

First, we consider the case $k = 0$.
Let $\g: \R \to [0,1]$ be a smooth cutoff function supported on $[-1, 1]$
such that
\[\sum_{m\in \Z} \g(t - m) \equiv 1, \qquad t \in \R, \]

\noi
as in Lemma \ref{LEM:a.o.}.
Note that  there exists $C>0$ such that
\[\sum_{|m| \leq C} \g (t - t_0 -m) = 1, \quad
\text{on } \big[-\tfrac{8}{5}-T, \tfrac{8}{5} + T\big]\]

\noi
for $t_0 \in [-T, T]$.
Then, by \eqref{linear4},  \eqref{linear1},  \eqref{linear2}, and \eqref{linear3a},
we have
\begin{align*}
\|\P_0u \|_{F^\al_0(T)}
& \leq \sup_{t_0 \in [-T, T]}
\bigg\|  \F \Big[\eta_0(t - t_0)\cdot \Big(
S(t) \P_0u(0) - i \int_0^t S(t - t') \wt v_0 (t') dt' \Big) \Big]\bigg\|_{X_0} \\
& \leq \|\P_0u(0)\|_{L^2} \\
& \hphantom{XXXXX}
+ \sup_{t_0 \in [-T, T]}
\sum_{|m| \leq C}
\big\|(\tau - n^2 + i )^{-1} \cdot \F[ \g(t - t_0 -m) \cdot \wt v_0]\big\|_{X_0}\\
& \les \|\P_0u(0)\|_{L^2}+\|\P_0 v\|_{N_0^\alpha(T)},
\end{align*}

\noi
yielding \eqref{linear3}.
Note that we used  $\g(t - t_0 -m) = \g(t - t_0 -m)\eta_0(t - t_0 -m)$
and \eqref{Xk3}
in the last step.

Next, we consider the case $k \geq 1$.
Given $t_k \in [-T, T]$, we write
\begin{align*}
\wt u_k(t) = \eta_{_T, k} (t)\big[S(t - t_k) \P_ku(t_k) - i \int_{t_k}^t S(t-t')\wt v_k(t') dt'\big], 
\end{align*}

\noi
where $\eta_{_T, k} (t)$ is defined by 
\[\eta_{_T, k} (t) = 
\eta_0(2^{[\alpha k]+5}(t-T)) \ind_{(T, \infty)}
+ \ind_{[-T, T]}
+ \eta_0(2^{[\alpha k]+5}(t+T)) \ind_{(-\infty, -T)}.\]

\noi
Then, noting that 
\[\eta_0(2^{[\alpha k]}(t-t_k))
= \eta_0(2^{[\alpha k]}(t-t_k))
\sum_{|m| \leq C} \g (2^{[\alpha k]}(t - t_k) -m), 
\]

\noi
we proceed with \eqref{linear1} and \eqref{linear2}
and obtain
\begin{align*}
\|\P_k u\|_{F^\al_k(T)}
& \les  \sup_{t_k\in [-T,T]}\|\P_ku(t_k)\|_{L^2}\\
& \hphantom{XXX}+ \sup_{t_k \in [-T, T]}
\sum_{|m| \leq C}
\big\|(\tau - n^2 + i 2^{[\al k]})^{-1} \cdot \F[ \g(2^{[\alpha k]}(t - t_k) -m) \cdot \wt v_0]\big\|_{X_0}\\
& \les  \sup_{t_k \in [-T, T]}
\|\P_ku(t_k)\|_{L^2}+  \|\P_k v\|_{N_k^\alpha(T)},
\end{align*}

\noi
where 
we used  $\g(2^{[\alpha k]}(t - t_k) -m) = \g(2^{[\alpha k]}(t - t_k) -m)\eta_0(2^{[\alpha k]}(t - t_k) -m)$
and \eqref{Xk3}
as before.

It remains to prove \eqref{linear4}.
It suffices to prove
\[
\sup_{t_k\in \R}\big\|\F[ \eta_0(2^{[\alpha k]}(t-t_k)) \cdot \wt u ]\big\|_{X_k}
\les \sup_{t_k\in [-T,T]}\big\|\F[ \eta_0(2^{[\alpha k]}(t-t_k)) \cdot \wt u ]\big\|_{X_k}.\]

\noi
For $t_k>T$, since $\wt{u}$ is supported in 
$[-T-\frac 85\cdot 2^{-[\alpha k]-5},T+\frac 85\cdot2^{-[\alpha k]-5}]$, it is easy to see that
\[\eta_0(2^{[\alpha k]}(t-t_k))\cdot \wt u(t)
=\eta_0(2^{[\alpha k]}(t-T))\eta_0(2^{[\alpha k]}(t-t_k))\cdot \wt u (t).\]

\noi
Therefore, from \eqref{Xk3}, we obtain
\[\sup_{t_k>T}\big\|\F[\eta_0(2^{[\alpha k]}(t-t_k)) \cdot \wt u ]\big\|_{X_k}
\les \big\|\F[ \eta_0(2^{[\alpha k]}(t-T)) \cdot \wt u ]\big\|_{X_k}.\]

\noi
A similar argument holds for  $t_k<-T$,
and hence we obtain
\eqref{linear4} as desired.
\end{proof}

\section{Strichartz and related multilinear estimates}
\label{SEC:Strichartz}

Recall  the following periodic $L^4$- and $L^6$-Strichartz estimates
due to Bourgain \cite{BO1}:
\begin{equation}\label{Stri1}
\|u\|_{L^4_{x, t}(\T\times \R)} \lesssim \|u\|_{X^{0, \frac{3}{8}}}
\quad \text{and} \quad
\|S(t) \phi \|_{L^6_{x, t}(\T\times \R)} \leq C_\eps  |I|^\eps \|\phi\|_{L^2}
\end{equation}

\noi
for any $\eps > 0$,
where $\phi$ is a function on $\T$ such that $\supp \ft{\phi}$
is contained in an interval $I$ of length $|I|$.

\begin{lemma}\label{LEM:L46}
Let $k,j\in \Z_+$.
Suppose that  $ u \in L^2(\T\times \R)$ with $\supp \ft{u} \subset
D_{\leq j}$. Then
\begin{align}
\|u \|_{L_{x,t}^4}\les 2^{\frac{3}{8}j } \|u \|_{L^2_{x, t}}.
\label{Stri2}
\end{align}

\noi
Moreover, if, in addition,  we assume that
$\supp \ft u \subset I\times \R $ for some interval $I$,  
then we have
\begin{align}
 \quad \|u \|_{L_{x,t}^6}\leq C_\eps |I|^\eps 2^{\frac{j}{2} }\|u \|_{L^2_{x, t}}
\label{Stri3}
\end{align}

\noi
for any $\eps > 0$.
\end{lemma}

\begin{proof}
The $L^4$-estimate \eqref{Stri2} is a direct consequence of \eqref{Stri1}.
Writing
\begin{align*}
u(x, t)  = c \sum_n \int \ft u(n, \tau) e^{inx} e^{it\tau} d\tau
  = c \sum_n \int \ft{u}(n, \tau+ n^2 ) e^{i(nx+ n^2t)} e^{it\tau} d\tau, 
\end{align*}

\noi
it follows from \eqref{Stri1} that 
\begin{align*}
\|u\|_{L^6_{x, t}}
& \les \int \eta_{\leq j} (\tau) \Big\| \sum_n  \ft u(n, \tau+ n^2 ) e^{i(nx+ n^2t)} \Big\|_{L^6_{x, t}} d\tau \\
& \les |I|^\eps
\int \eta_{\leq j} (\tau) \Big( \sum_n | \ft u(n, \tau+ n^2 )|^2  \Big)^\frac{1}{2} d\tau
\les |I|^\eps 2^\frac{j}{2} \|u\|_{L^2_{x, t}}.
\end{align*}

\noi
This completes the proof of the $L^6$-estimate \eqref{Stri3}.
\end{proof}

As an immediate consequence of Lemma \ref{LEM:L46}, we have the following
multilinear estimates.

\begin{lemma}\label{LEM:L4L6}
Let $u_{k_i, j_i}$ be a function on $\T \times \R$
such that $\supp \ft{u}_{k_i, j_i} \subset D_{k_i, \leq j_i}$, $i = 1, \dots, 4$.
Then, we have 
\begin{align}
& \bigg|\int_{\T\times \R }  u_{k_1, j_1} \cj u_{k_2, j_2}u_{k_3, j_3}\cj u_{k_4, j_4} dx dt\bigg|
\lesssim \prod_{i = 1}^4 2^\frac{3 j_i}{8}
\|\F(u_{k_i, j_i})\|_{\l^2_n L^2_\tau},  \label{L4}\\
& \bigg|\int_{\T\times \R}  u_{k_1, j_1} \cj u_{k_2, j_2}u_{k_3, j_3}\cj u_{k_4, j_4} dx dt\bigg|
\lesssim  2^{-\frac{ j_1^*}{2}}2^{\eps k_3^*}\prod_{i = 1}^4 2^\frac{ j_i}{2} \|\F(u_{k_i, j_i})\|_{\l^2_n L^2_\tau},
\label{L6}
\end{align}

\noi
for any  $\eps > 0$.
Here, $j^*_i$ and $k^*_i$ denote the decreasing rearrangements
of $j_i$ and $k_i$.

\end{lemma}

\begin{proof}
The first estimate \eqref{L4} follows from $L^4_{x, t}, L^4_{x, t}, L^4_{x, t}, L^4_{x, t}$-H\"older inequality
and \eqref{Stri2}.

Next, we prove the second estimate \eqref{L6}.
Without loss of generality, assume $k_1 \geq k_2 \geq k_3 \geq k_4$.
By writing  $I_{k_1} = \bigcup_{\l_1} J_{1\l_1}$
and $I_{k_2} = \bigcup_{\l_1} J_{2\l_2}$,
where $J_{1\l_1}$  and $J_{2\l_2}$
are intervals of length $\sim |I_{k_3}|$, 
we can decompose 
$\ft u_{k_i, j_i} $, $i = 1, 2$, as 
\[\ft u_{k_i, j_i} = \sum_{\l_i} \ft u_{k_i, j_i, \l_i},\]

\noi
where $\ft u_{k_i, j_i, \l_i} (n_i, \tau_i) = \ind_{J_{i\l_i}}(n_i) \ft u_{k_i, j_i}(n_i, \tau_i)$.
Given $n_1 \in J_{1\l_1}$ for some $\l_1$,
there exist $O(1)$ many possible values for $\l_2 = \l_2(\l_1)$ such that $n_2 \in J_{2\l_2}$
under $n_1 - n_2 + n_3 - n_4=0$.
Note that the number of possible values of $\l'$ is independent of $\l$.

By $L^2_{x, t}, L^6_{x, t}, L^6_{x, t}, L^6_{x, t}$-H\"older inequality
(placing the term with the highest modulation in $L^2_{x, t}$)
and \eqref{Stri3}, we obtain
\begin{align*}
\text{LHS of } \eqref{L6}
& =
\sum_{\l_1} \sum_{\l_2}
\sum_{n_1 - n_2 + n_3 -n_4 = 0}
\intt_{\tau_1 - \tau_2 + \tau_3 - \tau_4 = 0}\ft u_{k_1, j_1, \l_1}(n_1, \tau_1) \\
&\hphantom{XXXXXX}
\times \cj{ \ft u}_{k_2, j_2, \l_2}(n_2, \tau_2)
\ft u_{k_3, j_3}(n_3, \tau_3)\cj{ \ft u}_{k_4, j_4} (n_4, \tau_4) d\tau_1d\tau_2d\tau_3\\
& \lesssim
\sum_{\l_1} \sum_{\l_2}
2^{-\frac{ j_1^*}{2}}2^{\eps k_3^*}\bigg( \prod_{i = 1}^4 2^\frac{ j_i}{2}\bigg)
\|\ft u_{k_1, j_1, \l_1}\|_{\l^2_nL^2_\tau}
\\
&\hphantom{XXXXXXXX}
\times \|{ \ft u}_{k_2, j_2, \l2}\|_{\l^2_nL^2_\tau}
\|\ft u_{k_3, j_3}\|_{\l^2_nL^2_\tau}\| \ft u_{k_4, j_4} \|_{\l^2_nL^2_\tau}.
\end{align*}

\noi
Then, \eqref{L6} follows from Cauchy-Schwarz inequality in $\l_1$.
\end{proof}

The proof of \eqref{L6}
was based on the $L^6$-Strichartz estimate \eqref{Stri3}
after applying H\"older inequality.
By refining the analysis, we can obtain the following lemmata,
which can be viewed as an refinement of the $L^6$-Strichartz estimate
(in the multilinear setting) in certain  cases.

\begin{lemma}\label{LEM:L62}
Let $\al \in [0, 1]$.
Let $u_i$ be a function on $\T \times \R$
such that $\supp \ft{u}_{i} \subset D_{k_i, j_i}$.
Assume that $k_1,  k_2,  k_3 \geq k_1^* - 5 \geq k_4 + 5 $ and  $j_1, j_2, j_3 \geq [\al k_1^*]$.
Then, we have
\begin{align}
& \bigg|\int_{\T\times \R}  u_{1} \cj u_{2}u_{3}\cj u_{4} dx dt\bigg|
\lesssim  M\prod_{i = 1}^4 2^\frac{ j_i}{2} \|\F(u_i)\|_{\l^2_n L^2_\tau}, 
\label{L62}
\end{align}

\noi
where $M$ is given by 
\[
M = 
\begin{cases}
2^{-\frac{ j_1^*}{2}- \frac{1}{2} \al k_1^* + \frac{1}{2}k_4 },
& \text{if } j_1^* \ne j_2, \\
2^{-\frac {j_1^*}2- \frac{1}{4} \al k_1^* + \frac{1}{2}k_4 },
& \text{if } j_1^* = j_2.
\end{cases}
\]
\end{lemma}

\begin{proof}
Letting $f_i = \ft{u}_i $ for $i = 1, 3$
and $f_i = \cj{\ft{u}}_i $ for $i = 2, 4$, we have
\begin{align}
\text{LHS of } \eqref{L62}
= c \intt_{\tau_1 - \tau_2 + \tau_3 - \tau_4 = 0}\sum_{n_1 - n_2 + n_3 -n_4 = 0}
\prod_{i = 1}^4 f_i(n_i, \tau_i) d\tau_1d\tau_2d\tau_3.
\label{L621}
\end{align}

\noi
$\bullet$ {\bf Case (a):} $j_4 = j_1^*$.

By symmetry, assume $k_1 \geq k_3$.
Then, we can assume $k_1 \geq k_2^* \geq k_1^* - 3$
under $n_1 - n_2 + n_3 -n_4 = 0$.
With $g_i(n, \tau) = f_i (n, \tau + n^2)$, we have
\begin{align*}
\eqref{L621}
& \les \int
\sum_{n_4} |g_4(n_4, \tau_4)|
\sum_{n_1, n_2} |g_1(n_1, \tau_1)|
|g_2(n_2, \tau_2)| \notag \\
& \hphantom{XXXXXX} \times
  \big|g_3\big(-n_1 + n_2 + n_4,
 h_3(n_1, n_2, n_4, \tau_1, \tau_2, \tau_4) \big)\big| d\tau_1d\tau_2d\tau_4,
\end{align*}

\noi
where $h_3(n_1, n_2, n_4, \tau_1, \tau_2, \tau_4)$ is defined by
\begin{equation*}
 h_3(n_1, n_2, n_4, \tau_1, \tau_2, \tau_4)
  = -\tau_1+\tau_2 + \tau_4 -  n_1^2 + n_2^2 + n_4^2- (-n_1 + n_2 + n_4)^2.
\end{equation*}

\noi
For fixed $n_1, n_4, \tau_1, \tau_2$, and $\tau_4$,
define the set $E_{32}= E_{32}(n_1, n_4, \tau_1, \tau_2, \tau_4)$ by
\begin{align*}
E_{32} = \{ n_2 \in \Z: h_3(n_1, n_2, n_4, \tau_1, \tau_2, \tau_4) = O(2^{j_3})\}.
\end{align*}

\noi
Noting that $|n_1 - n_4| \sim 2^{k_1^*}$ and that
\begin{align*}
-  n_1^2 + n_2^2 + n_4^2- (-n_1 + n_2 + n_4)^2
= 2n_2 (n_1 - n_4) - 2n_1^2 + 2n_1n_4,
\end{align*}

\noi
we conclude that
\begin{align}
|E_{32}| \lesssim 1+ 2^{j_3 - k_1^*}.
\label{L62b}
\end{align}

\noi
Then, by
Cauchy-Schwarz inequality in $n_2, n_1, n_4$, 
we obtain
\begin{align*}
\eqref{L621}
& \les 
(1+ 2^{j_3 - k_1^*})^\frac{1}{2}
\int \sum_{n_4} |g_4(n_4, \tau_4)|
\sum_{n_1} |g_1(n_1, \tau_1)|
\bigg(\sum_{n_2} |g_2(n_2, \tau_2)|^2 \notag \\
& \hphantom{XXXXXX} \times
 \big| g_3\big(-n_1 + n_2 + n_4,
 h_3(n_1, n_2, n_4, \tau_1, \tau_2, \tau_4) \big)\big|^2\bigg)^\frac{1}{2} d\tau_1d\tau_2d\tau_4\\
& \les
(1+ 2^{j_3 - k_1^*})^\frac{1}{2} 2^{\frac {k_4}2}
 \|g_4(n_4, \tau_4)\|_{\l^2_{n_4}L^2_{\tau_4}}
\sup_{n_4} \int \|g_1(n_1, \tau_1)\|_{\l^2_{n_1}}
\bigg(\sum_{n_2} |g_2(n_2, \tau_2)|^2 \notag \\
& \hphantom{XXXXXX} \times
\sum_{n_1} \big\|  g_3\big(-n_1 + n_2 + n_4,
 h_3(n_1, n_2, n_4, \tau_1, \tau_2, \tau_4) \big)\big\|^2_{L^2_{\tau_4}}
 \bigg)^\frac{1}{2} d\tau_1d\tau_2\\
\intertext{Noting that $h_3$ is linear in $\tau_4$
and applying Cauchy-Schwarz inequality in $\tau_1$ and $\tau_2$,} 
& \leq
(1+ 2^{j_3 - k_1^*})^\frac{1}{2} 2^{\frac {k_4}2}
\bigg(\prod_{i_1 = 1}^2  \|g_{i_1}\|_{L^1_{\tau_{i_1}} \l^2_{n_{i_1}}}\bigg)
\bigg(\prod_{i_2 = 3}^4
 \|g_{i_2}\|_{\l^2_{n_{i_2}}L^2_{\tau_{i_2}}}\bigg)\\
& \les  2^{-\frac{ j_1^*}{2}}
(1+ 2^{j_3 - k_1^*})^\frac{1}{2}2^{-\frac{j_3}{2}}\cdot 2^\frac{k_4}{2}
\prod_{i = 1}^4 2^\frac{ j_i}{2} \|\F(u_i)\|_{\l^2_n L^2_\tau},
\end{align*}

\noi
 yielding \eqref{L62}.

\medskip

\noi
$\bullet$ {\bf Case (b):} $j_1 = j_1^*$.
(The case $j_3 = j_1^*$ follows in the same manner by symmetry.)

By Cauchy-Schwarz inequality, we have
\begin{align*}
\eqref{L621}
& \les \int \sum_{n_4}
|g_4(n_4, \tau_4)|
\sum_{n_1, n_3}\int
|g_1(n_1, \tau_1)|| g_3(n_3, \tau_3)| \\
& \hphantom{XXXXXXXXXXX}
\times
\big|g_2\big(n_1+n_3-n_4, h_2(n_1, n_3, n_4, \tau_1, \tau_3, \tau_4)\big) \big| d \tau_1d\tau_3d\tau_4
\end{align*}

\noi
where $h_2(n_1, n_3, n_4, \tau_1, \tau_3, \tau_4)$ is defined by
\begin{equation*}
 h_2(n_1, n_3, n_4, \tau_1, \tau_3, \tau_4)
  =  \tau_1 + \tau_3-\tau_4 +n_1^2 + n_3^2 - n_4^2 - (n_1 +n_3 - n_4)^2.
\end{equation*}

\noi
For fixed $n_1, n_4, \tau_1, \tau_3$, and $\tau_4$,
define the set $E_{23}= E_{23}(n_1, n_4, \tau_1, \tau_3, \tau_4)$ by
\begin{align*}
E_{23} = \{ n_3 \in \Z: h_2(n_1, n_3, n_4, \tau_1, \tau_3, \tau_4) = O(2^{j_2})\}.
\end{align*}

\noi
Then, from
$ n_1^2 + n_3^2 -  n_4^2 - (n_1 + n_3 - n_4)^2
= - 2n_3 (n_1 - n_4)  + 2n_1 n_4 - 2n_4^2$
and $|n_1 - n_4| \sim 2^{k_1^*}$,
we conclude that
$|E_{23}| \lesssim 1+ 2^{j_2 - k_1^*}$.
Then, proceeding as before with 
Cauchy-Schwarz inequality
and noting that $h_2$ is linear in $\tau_1$, we have
\begin{align*}
\eqref{L621}
& \les
(1+ 2^{j_2 - k_1^*})^\frac{1}{2} 2^{\frac {k_4}2}
\int  \|g_4(n_4, \tau_4)\|_{\l^2_{n_4}}
\|g_1(n_1, \tau_1)\|_{\l^2_{n_1}L^2_{\tau_1}}
\sup_{n_4} \bigg(\sum_{n_3} |g_3(n_3, \tau_3)|^2
 \notag \\
& \hphantom{XXXXXX} \times
\sum_{n_1} \big\|
g_2\big(n_1+n_3-n_4, h_2(n_1, n_3, n_4, \tau_1, \tau_3, \tau_4)\big)
  \big\|^2_{L^2_{\tau_1}}
 \bigg)^\frac{1}{2} d\tau_3d\tau_4 \notag \\
& \les  2^{-\frac{ j_1^*}{2}}
(1+ 2^{j_2 - k_1^*})^\frac{1}{2}2^{-\frac{j_2}{2}}\cdot 2^\frac{k_4}{2}
\prod_{i = 1}^4 2^\frac{ j_i}{2} \|\F(u_i)\|_{\l^2_n L^2_\tau}.
\end{align*}

\medskip

\noi
$\bullet$ {\bf Case (c):} $j_2 = j_1^*$.

In this case, we do not use the multilinear argument.
The desired estimate 
follows from $L^4_{x, t}, L^2_{x, t}, L^4_{x, t}, L^\infty_{x, t}$-H\"older inequality
and the $L^4$-Strichartz estimate \eqref{Stri2}.
 \end{proof}

\begin{remark} \rm
In the periodic case, we have  a lower bound $|E| \geq 1$
unless $|E| = 0$.
See \eqref{L62b}.
Hence, our estimates are worse than those in 
the non-periodic setting.
Compare this with Propositions 4.1 and 4.2 in \cite{CCT3}.
\end{remark}

\begin{lemma}\label{LEM:L63}
Let $\al \in [0, 1]$.
Let $u_i$ be a function on $\T \times \R$
such that $\supp \ft{u}_{i} \subset D_{k_i,  j_i}$.
Assume that $k_3, k_4 \leq k_2^* - 10$, $j_1, j_2, j_3 \geq [\al k_1^*]$,
and $j_1^* \geq |\Phi (\bar n)|$,
where $\Phi(\bar n )$ is defined in \eqref{Phi}.
Then, we have
\begin{align}
& \bigg|\int_{\T\times \R}  u_{1} \cj u_{2}u_{3}\cj u_{4} dx dt\bigg|
\lesssim
M \prod_{i = 1}^4 2^\frac{ j_i}{2} \|\F(u_i)\|_{\l^2_n L^2_\tau},
\label{L63}
\end{align}

\noi
where $M$ is given by 
\[
M = \begin{cases}
2^{-\frac{ 1}{2}(1+\al) k_1^*}, & \text{if }|k_3-k_4|\geq  2,\\
 2^{-\frac{ 1}{2}(1+\al) k_1^* + \frac{k_3^*}{4}}, &
\text{otherwise}.
\end{cases}
\]

Moreover, 
when $j_3 = j_1^*$ and $k_3 \geq k_4 + 2$,
 \eqref{L63} holds with
\begin{equation}
 M = 2^{-\frac{k_1^*}{2}-\frac{k_3}{2}+ \frac{k_4}{2} -\frac{\be}{2} },
 \quad
\text{ where }\be = \min(\al k_1^*, k_3).
\label{L63a}
\end{equation}

\noi

\end{lemma}

\begin{proof}
First, we consider the case $|k_3 - k_4 | \geq 2$.
Then, we have $j_1^* \geq k_1^* + k_3^* - 5$
since $|\Phi(\bar n)| \sim |(n_2 - n_3)(n_3 - n_4)|
\sim n_1^*|n_3 - n_4|
\sim n_1^* n_3^*$.
When (a): $j_4 = j_1^*$
or (b): $j_1 = j_1^*$, 
the desired estimate follows from 
the corresponding cases in 
the proof of Lemma \ref{LEM:L62}.

Next, consider the case  (c)  $j_2 = j_1^*$.
With the notations from the proof of Lemma \ref{LEM:L62}, 
we have
\begin{align}
\text{LHS of } \eqref{L63}
& \les \int \sum_{n_4} |g_4(n_4, \tau_4)|
\sum_{n_2, n_3}
|g_2(n_2, \tau_2)| |g_3(n_3, \tau_3)| \notag\\
& \hphantom{XXXXX}
\times
\big|g_1\big(n_2-n_3+n_4, h_1(n_2, n_3, n_4, \tau_2, \tau_3, \tau_4)\big)\big|d\tau_2d\tau_3d\tau_4,
\label{L63a1}
\end{align}

\noi
where $h_1(n_2, n_3, n_4, \tau_2, \tau_3, \tau_4)$ is defined by
\begin{equation}
 h_1(n_2, n_3, n_4, \tau_2, \tau_3, \tau_4)
  =  \tau_2 - \tau_3+\tau_4 +n_2^2 - n_3^2  + n_4^2 -  (n_2 -n_3 +n_4 )^2.
\label{L63a2}
\end{equation}

\noi
For fixed $n_2, n_4, \tau_2, \tau_3$, and $\tau_4$,
define the set $E_{13}= E_{13}(n_2, n_3, \tau_2, \tau_3, \tau_4)$ by
\begin{align*}
E_{13} = \{ n_3 \in \Z: h_1(n_2, n_3, n_4, \tau_2, \tau_3, \tau_4) = O(2^{j_1})\}.
\end{align*}

\noi
Then, by writing $ n_2^2 - n_3^2  + n_4^2 -  (n_2 -n_3 +n_4 )^2
=  - 2n_3^2 +  2 (n_2 + n_4)n_3 - 2n_2n_4$, 
we have 
$|\dd_{n_3} h_1(n_3)| = |2(n_2 - 2n_3 + n_4)| \sim 2^{k_1^*}$
since $|n_2| \gg |n_3|, |n_4|$.
Hence, 
we conclude that
$|E_{13}| \lesssim 1+ 2^{j_1 - k_1^*}$.
Proceeding as in the proof of Lemma \ref{LEM:L62},
we  obtain
\begin{align*}
\eqref{L63a1}
& \les  2^{-\frac{ j_1^*}{2}}
(1+ 2^{j_1 - k_1^*})^\frac{1}{2}2^{-\frac{j_1}{2}}\cdot 2^\frac{k_4}{2}
\prod_{i = 1}^4 2^\frac{ j_i}{2} \|\F(u_i)\|_{\l^2_n L^2_\tau}.
\end{align*}

Hence, it remains to 
 consider the case $j_3 = j_1^*$.
With $g_i(n, \tau) = f_i(n, \tau+n^2)$, we have
\begin{align}
\text{LHS of } \eqref{L63}
& \les \int \sum_{n_3} |g_3(n_3, \tau_3)|
\sum_{n_2, n_4}
|g_2(n_2, \tau_2)| |g_4(n_4, \tau_4)| \notag\\
& \hphantom{XXXXX}
\times
\big|g_1\big(n_2-n_3+n_4, h_1(n_2, n_3, n_4, \tau_2, \tau_3, \tau_4)\big)\big|d\tau_2d\tau_3d\tau_4,
\label{L630}
\end{align}

\noi
where $h_1(n_2, n_3, n_4, \tau_2, \tau_3, \tau_4)$ is as in \eqref{L63a2}.
For fixed $n_2, n_3, \tau_2, \tau_3$, and $\tau_4$,
define the set $E_{14}= E_{14}(n_2, n_3, \tau_2, \tau_3, \tau_4)$ by
\begin{align*}
E_{14} = \{ n_4 \in \Z: h_1(n_2, n_3, n_4, \tau_2, \tau_3, \tau_4) = O(2^{j_1})\}.
\end{align*}

\noi
Then, from
$ n_2^2 - n_3^2  + n_4^2 -  (n_2 -n_3 +n_4 )^2
=  - 2n_4 (n_2 - n_3) - 2n_3^2 +2n_2 n_3$
and $|n_2 - n_3| \sim 2^{k_1^*}$,
we conclude that
$|E_{14}| \lesssim 1+ 2^{j_1 - k_1^*}$.
Proceeding as in the proof of Lemma \ref{LEM:L62},
we have
\begin{align*}
 \eqref{L630}
 & \les  2^{-\frac{ j_1^*}{2}}
(1+ 2^{j_1 - k_1^*})^\frac{1}{2}2^{-\frac{j_1}{2}} \cdot 2^\frac{k_3}{2}
\prod_{i = 1}^4 2^\frac{ j_i}{2} \|\F(u_i)\|_{\l^2_n L^2_\tau}\\
 & \les  2^{-\frac{ 1}{2}(1+\al) k_1^*}
\prod_{i = 1}^4 2^\frac{ j_i}{2} \|\F(u_i)\|_{\l^2_n L^2_\tau}.
\end{align*}

Next, we consider the case $|k_3 - k_4 |\leq 1$.
We separate the argument into two cases:
(i) $|n_3 - n_4| \ges 2^\frac{k_3^*}{2}$
and (ii) $|n_3 - n_4| \ll 2^\frac{k_3^*}{2}$.
In Case (i), we have  $j_1^* \geq k_1^* + \frac{k_3^*}{2} - 5$.
Hence, by repeating the previous argument, we obtain
\begin{align}
\text{LHS of } \eqref{L63}
 & \les  2^{-\frac{ 1}{2}(1+\al) k_1^*}
2^\frac{k_3^*}{4}
\prod_{i = 1}^4 2^\frac{ j_i}{2} \|\F(u_i)\|_{\l^2_n L^2_\tau}.
\label{L631}
\end{align}

In Case (ii), we write $I_{k_i} = \bigcup_{\l_i} J_{i\l_i}$, $i = 3, 4$, 
where $|J_{i\l_i}| = 2^\frac{k_3^*}{2}$.
As in the proof of Lemma \ref{LEM:L4L6}, 
it follows that, if $n_3 \in J_{3 \l_3}$ for some $\l_{3}$,
there are $O(1)$ many possible values for $\l_4 = \l_4(\l_3)$
such that  $n_4\in J_{4\l_4}$.
Then, by writing
\[\sum_{n_3} \sum_{n_4} = \sum_{\l_3}  \sum_{\l_4 = \l_4(\l_3)} \sum_{n_3 \in J_{3\l_3} }
\sum_{n_4 \in J_{4\l_4}}\]

\noi
and repeating the previous argument for each $\l_3$,
we only lose $|J_{i \l_i}| =  2^\frac{k_3^*}{4}$ by applying Cauchy-Schwarz inequality
in $n_3$ or $n_4$ at the end.
Finally, applying Cauchy-Schwarz inequality in $\l_3$, we obtain
\eqref{L631}.

The second claim with \eqref{L63a}
follows from Case (a) in the proof of Lemma \ref{LEM:L62}
by switching the indices $1\lrarrow 3$.
In this case, we have $|E_{12}| \les 1 + 2^{j_1 - k_3}$
and it suffices to note that $j_3 \geq k_1^* + k_3 - 5$
and $ (1 + 2^{j_1 - k_3})^\frac{1}{2}2^{-\frac{j_1}{2}}
\les 2^{-\frac{1}{2} \min(j_1, k_3)}
\les 2^{-\frac{1}{2} \min(\al k_1^*, k_3)}$.
\end{proof}

\begin{remark} \label{REM:L66}
\rm
It follows from the proof of Lemma \ref{LEM:L63}
that we lose $2^{\frac{k_3}{2}}$
in the case $j_3 = j_1^*$,
while we lose $2^\frac{k_4}{2}$ in other cases.
\end{remark}

\section{Trilinear estimates}
\label{SEC:trilinear}

In this section, we prove the main trilinear estimates
on the cubic nonlinear terms  $\N$ and $\RR$ in \eqref{NN1} and \eqref{NN2}
of the Wick ordered cubic NLS \eqref{NLS1}.

\begin{proposition} \label{PROP:3lin}
Let $s \in (-\frac{1}{8}, 0)$ and $T \in (0, 1]$.
Then, 
with $\alpha =  -4s+$,
there exists $\theta > 0$ such that
\begin{align}
\|\N(u_1,u_2,u_3)\|_{N^{s,\alpha}(T)}
& \les  T^\theta \prod_{i = 1}^3 \|u_i\|_{F^{s,\alpha}(T)},\label{3lin1}\\
\|\RR(u_1,u_2,u_3)\|_{N^{s,\alpha}(T)}
& \les  T^\theta  \prod_{i = 1}^3 \|u_i\|_{F^{s,\alpha}(T)},\label{3lin2}
\end{align}

\noi
where $\N(u_1, u_2, u_3)$ and
$\RR(u_1, u_2, u_3)$ are defined in \eqref{NN1} and \eqref{NN2}.

\end{proposition}

In the following, we first prove the preliminary lemmata
and then present the proof of Proposition \ref{PROP:3lin}
at the end of this section.
Let $\Phi(\bar n) = 2(n - n_1)(n - n_3)$ be the phase function as in \eqref{Phi}.
Then, under $\{n_1, n_3\} \ne \{n, n_2\}$, we have the following.

\medskip
\noi
{\bf Support conditions:}
\begin{itemize}
\item[(i)] If $|n|\sim |n_3| \gg |n_1|, |n_2|$, then
\begin{equation}\label{SC1}
|\Phi(\bar n)| \sim n_1^* |n_2 - n_1|.
\end{equation}

\item[(ii)] If $|n|\sim |n_2| \gg |n_1|, |n_3|$, then
\begin{equation}\label{SC2}
|\Phi(\bar n)| \sim (n_1^*)^2.
\end{equation}

\item[(iii)] If $|n|\sim |n_2|\sim |n_3| \gg |n_1|$, then
\begin{equation}\label{SC3}
|\Phi(\bar n)| \sim (n_1^*)^2.
\end{equation}

\item[(iv)] If $|\Phi(\bar n)|  \ll n_1^*$,  then
\begin{equation*}
|n| \sim|n_1|\sim|n_2|\sim|n_3|.
\end{equation*}

\end{itemize}

\noi
The conditions (i)--(iii) hold under the symmetries $n_1 \leftrightarrow n_3$
and $n \leftrightarrow n_2$,
and $\{n_1, n_3\} \leftrightarrow \{n, n_2\}$, respectively.
In the following, we assume that
 $u_i$ is a function on $\T\times \R$ such that $\supp\ft{u}_i \subset I_{k_i}\times \R$,
 $i = 1, 2, 3$.

\begin{lemma}[high $\times$ low $\times$ low $\to$ high]\label{LEM:HLL}
Let $\alpha\geq 0$.
If $k_4\geq 20$, $|k_1-k_4|\leq 5$, and $k_2,k_3\leq
k_1-10$, then we have
\begin{align}
\|\P_{k_4}\N(u_{1}, u_{2}, u_{3})\|_{N_{k_4}^\alpha}
& \les \min(2^{- \frac{1}{2}(\al - \eps) k_1^*},2^{-\frac{1}{2}(1-\eps)k_1^*})
\prod_{i = 1}^3\|u_{i}\|_{F_{k_i}^\al}
\label{HLL1}
\end{align}

\noi
for any $\eps >0$.
The estimate \eqref {HLL1} holds under permutation of indices $k_1, k_2$, and $k_3$.
\end{lemma}

\begin{proof}
Let $\g:\R \to [0, 1]$ be a smooth cutoff function supported on $[-1, 1]$
with $\g \equiv 1$ on $[-\frac{1}{4}, \frac{1}{4}]$
such that
\[ \sum_{m} \g^3(t-  m ) \equiv 1, \qquad t \in \R.\]

\noi
Then, there exist $c, C> 0$ such that
\begin{align*}
 \eta_0(2^{[\al k_4]} (t - t_{k_4}))
= \eta_0(2^{[\al k_4]} (t - t_{k_4})) \sum_{|m|\leq C} \g^3(2^{[\al k_4]+c} (t - t_{k_4})- m)
\end{align*}

\noi
and
\begin{align}
 \eta_0(2^{[\al k_i]}t) \cdot \g(2^{[\al k_4]+c}t)
= \g(2^{[\al k_4]+c}t)
\label{HLL1b}
\end{align}

\noi
 for $i = 1,2,  3$.
Then,  from the  definition and Lemma \ref{LEM:embed3},
 the left-hand side of \eqref{HLL1}
is estimated  by
\begin{align}\label{HLL11}
C \sup_{t_{k_4} \in \R}  &  \big\|(\tau-n^2+    i2^{[\alpha k_4]})^{-1}
\ind_{I_{k_4}}(n)\F[ \g(2^{[\alpha k_4]+c}(t-t_{k_4}) - m )\cdot u_1] \notag\\
& *\F[ 
\g(2^{[\al k_4]+c}(t-t_{k_4})- m )\cdot \cj{u}_{2}] 
*\F[ \g(2^{[\al k_4]+c}(t-t_{k_4}) -  m )\cdot u_3]\big\|_{X_{k_4}}.
\end{align}

\noi
Let $f_{k_i}=\F[ \g(2^{[\al k_4]+c}(t-t_{k_4}))\cdot u_i ]$, $i = 1, 3$,
and
$ \cj {f_{k_2}}=\F[ 
\g(2^{[\al k_4]+c}(t-t_{k_4}))\cdot  u_2 ]$.
Then,  from the definition \eqref{Xk} of $X_k$,
we have
\begin{align}
\eqref{HLL11}
\les  \sup_{t_{k_4}\in \R} \sum_{j_4=0}^{\infty}2^\frac{j_4}{2}
\sum_{j_1,j_2,j_3\geq [\al k_4]}
\big\|(2^{j_4}& +2^{[\alpha k_4]})^{-1} \notag \\
& \times \ind_{D_{k_4, j_4}}\cdot
(f_{k_1,j_1}*\wt{ f}_{k_2,j_2}*f_{k_3,j_3})\big\|_{\l^2_nL^2_\tau},
\label{HLL12}
\end{align}

\noi
where $\wt{f}(n,\tau)=f(-n,-\tau)$ and $f_{k_i, j_i}$, $i = 1, 2, 3$, is defined by
\[ f_{k_i, j_i}(n, \tau) =
\begin{cases}
f_{k_i}(n,\tau)\eta_{j_i}(\tau-n^2), &
\text{for }j_i>[\al k_4], \\
f_{k_i}(n,\tau)\eta_{\leq [\al k_4]}(\tau-n^2), & \text{for } j_i = [\al k_4].
\end{cases}
\]

Using the fact $\ind_{D_{k_4,j_4}}\leq \ind_{D_{k_4,\leq
j_4}}$, we have
\begin{align}
\eqref{HLL12}
& \les\sup_{t_{k_4} \in \R}\Big(\sum_{j_4<[\al k_4]}+\sum_{j_4\geq[\al k_4]}\Big)
2^{\frac{j_4}{2}}\sum_{j_1,j_2,j_3\geq [\alpha k_4]}
\big\|(2^{j_4}+2^{[\alpha k_4]})^{-1}
\ind_{D_{k_4, j_4}}\notag \\
& \hphantom{XXXXXXX} \times
f_{k_1,j_1}*\wt{f}_{k_2,j_2}*f_{k_3,j_3}
\big\|_{\l^2_n L^2_\tau }\notag\\
& \les \sup_{t_{k_4}\in \R}
\sum_{j_1,j_2,j_3,j_4\geq [\al k_4]}2^{-\frac{j_4}{2}}\big\|\ind_{D_{k_4,\leq j_4}}\cdot
(f_{k_1,j_1}*\wt{f}_{k_2,j_2}*f_{k_3,j_3})\big\|_{\l^2_n L^2_\tau }\notag \\
& \les \sup_{t_{k_4}\in \R}
\sup_{j_4 \geq [\al k_4]}
\sum_{j_1,j_2,j_3\geq [\al k_4]}2^{-(\frac{1}{2} - ) j_4}\big \|\ind_{D_{k_4,\leq j_4}}\cdot
(f_{k_1,j_1}*\wt{f}_{k_2,j_2}*f_{k_3,j_3})\big\|_{\l^2_n L^2_\tau }.
\label{HLL13}
\end{align}

\noi
From duality and Lemma \ref{LEM:L4L6}
along with
the support conditions \eqref{SC1} and \eqref{SC2}, 
we have
\begin{align*}
\eqref {HLL13}
& \les \sup_{t_{k_4}\in \R}
\min(2^{- \frac{1}{2}(\al- \eps)  k_1^*},2^{-\frac{1}{2}(1-\eps)k_1^*})
\sum_{j_1,j_2,j_3 \geq [\al k_4]}
\prod_{i = 1}^3 2^\frac{j_i}{2}\|f_{k_i,j_i}\|_{\l^2_n L^2_\tau }.
\end{align*}

\noi
Finally, by applying \eqref{Xk2} with \eqref{HLL1b}, we obtain \eqref{HLL1}.
When $k_2 \geq k_2^*$,  \eqref{SC2} actually provides
a better estimate with $2^{-(1-\eps) k_1^*}$.
\end{proof}

In the proof of Lemma \ref{LEM:HLL}, we carried the supremum over $t_{k_4}$
along the argument.
For simplicity, we may drop the supremum over $t_{k_4}$
in the following proofs 
without explicitly stating so.

\begin{lemma}[high $\times$ high $\times$ low $\to$ high] \label{LEM:HHLH}
Let $\alpha\geq 0$.
If $k_4\geq 20$, $|k_i-k_4|\leq 5$, $i = 1, 2$, and 
$k_3\leq k_1-10$, then we have
\begin{align}
\|\P_{k_4}\N(u_{1}, {u}_{2}, u_{3})\|_{N_{k_4}^\alpha}
\les 2^{-(1-\eps)k_1^*}
\|u_{1}\|_{F_{k_1}^\alpha}\|u_{2}\|_{F_{k_2}^\alpha}\|u_3\|_{F_{k_3}^\alpha}
 \label{HHLH1}
\end{align}

\noi
for any $\eps > 0$.
The estimate \eqref {HHLH1} holds under permutation of indices $k_1, k_2$, and $k_3$.

\end{lemma}
\begin{proof}
We proceed as in Lemma \ref{LEM:HLL}.
In this case, we have $j_1^* \geq 2k_1^* - 5$
from the support condition \eqref{SC3}.
Then, \eqref{HHLH1} follows from \eqref{L6} in Lemma \ref{LEM:L4L6}.
\end{proof}

\begin{lemma}[high $\times$ high $\times$ high $\to$ high] \label{LEM:HHHH}
Let $\al \geq 0$.
If $k_4\geq 20$, $|k_i-k_4|\leq 5$, $i = 1, 2, 3$, 
then we have
\begin{align}
\|\P_{k_4}\N(u_{1}, u_{2}, u_{3})\|_{N_{k_4}^\alpha}
+ \|\P_{k_4}\RR & (u_{1}, u_{2}, u_{3})\|_{N_{k_4}^\alpha} \notag\\
& \les
2^{-\frac{1}{2}(\alpha -\eps)k_1^*}
\|u_{1}\|_{F_{k_1}^\alpha}\|u_{2}\|_{F_{k_2}^\alpha}\|u_{3}\|_{F_{k_3}^\alpha}
\label{HHHH1}
\end{align}

\noi
for any $\eps > 0$.
\end{lemma}

\begin{proof}
In this case, we do not have any lower bound on the size of the phase function $\Phi(\bar n)$.
Proceeding as in Lemma \ref{LEM:HLL} with $j_i \geq [\al k_4]$, $i = 1, \dots, 4$,
 \eqref{HHHH1} follows from \eqref{L4} in Lemma \ref{LEM:L4L6}.
\end{proof}

\begin{lemma}[high $\times$ high $\times$ high $\to$ low]\label{LEM:HHHL}
Let $\al \geq 0$.
If $k_3\geq 20$, $|k_3-k_i|\leq 5$, $i = 1, 2$, 
and $k_4\leq k_1-30$, then we have
\begin{align}
\|\P_{k_4}\N(u_{1}, {u}_{2}, u_{3})\|_{N_{k_4}^\alpha}
& \les 
\min(M_1, M_2) 
\|u_{1}\|_{F_{k_1}^\alpha}\|u_2\|_{F_{k_2}^\alpha}\|u_3\|_{F_{k_3}^\alpha}
\label{HHHL1}
\end{align}

\noi
for any $\eps >0$,
where $M_1$ and $M_2$ are given by 
\[M_1 = 2^{-(1  - \al -\eps)k_1^* - \al k_4 }
\qquad \text{and}\qquad
M_2 = 2^{-(1 - \frac{3}{4}\al-\eps )k_1^* + (\frac{1}{2}-\al) k_4}.\]

\end{lemma}
\begin{proof}
Let $\g:\R\rightarrow [0,1]$ be
as in the proof of Lemma \ref{LEM:HLL}.
Then, there exist $c, C> 0$ such that
\[ \eta_0(2^{[\al k_4]} (t - t_{k_4}))
= \eta_0(2^{[\al k_4]} (t - t_{k_4}))
\sum_{|m|\leq C 2^{[\al k_1] - [\al k_4]}} \g^3(2^{[\al k_1]+c} (t - t_{k_4})- m)\]

\noi
and
$ \eta_0(2^{[\al k_i]}t) \cdot \g(2^{[\al k_1]+c}t)
= \g(2^{[\al k_1]+c}t)$ for $i = 1,2,  3$.
By Lemma \ref{LEM:embed3}, the left-hand side of \eqref{HHHL1} is estimated by
\begin{align*}
C\sup_{t_{k_4}\in \R}  \bigg\|(\tau-n ^2+ & i2^{[\alpha k_4]})^{-1}
 \ind_{I_{k_4}}(n)\sum_{|m|\leq C2^{[\alpha k_1]-[\alpha k_4]}}
F_{m, t_{k_4}}(n, \tau)\bigg\|_{X_{k_4}},
\end{align*}

\noi
where $F_{m, t_{k_4}}(n, \tau) = \ft v_{1, m}*  \ft{\cj v}_{2, m}
* \ft v_{3, m} (n, \tau)$ and
$ v_{i, m} = \g(  2^{{[\alpha k_1]+c}}(t-t_{k_4})-m)\cdot u_i$, $i = 1, 2, 3$.
Proceeding as in the proof of Lemma \ref{LEM:HLL}
with \eqref{Xk2},
it suffices to prove that
\begin{align*}
2^{\alpha (k_1-k_4)}\sum_{j_4\geq [\al k_4]}2^{-\frac{j_4}{2}}\|\ind_{\wt D_{k_4, j_4}}\cdot
(f_{k_1,j_1}*\wt{f}_{k_2,j_2}*f_{k_3,j_3})\|_{\l^2_n L^2_\tau}
\les M
\prod_{i = 1}^3 2^{\frac{j_i}{2}}\|f_{k_i,j_i}\|_{\l^2_n L^2_\tau}
\end{align*}

\noi
for any 
$f_{k_i,j_i}:\Z\times \R \to  \R_+$  
supported on $\wt D_{k_i, j_i}$
with $j_i \geq [\alpha k_1]$,   $i=1,2,3$, 
where
\[ \wt D_{k_i, j_i} = 
\begin{cases}
 D_{k_i,\leq j_i},
 & \text{when }j_i = [\alpha k_1],\\
 D_{k_i, j_i}, & 
\text{when }j_i > [\alpha k_1].
\end{cases}
\]

\noi
Here, we can assume that $j_i \geq [\al k_1]$, $i = 1, 2, 3$,
thanks to the time localization over an interval of size $\sim 2^{-[\al k_1] - c}$
and \eqref{Xk2}.
Hence, \eqref{HHHL1} 
with $M_1$
follows from
\eqref{L6} in Lemma \ref{LEM:L4L6}
with \eqref{SC3}: $j_1^* \geq 2k_1^* - 5$,
while
\eqref{HHHL1} 
with $M_2$
follows from Lemma \ref{LEM:L62}.
\end{proof}

\begin{lemma}[high $\times$ high $\times$ low $\to$ low] \label{LEM:HHLL}
Let $\al \in [0,1 ]$.
If $k_1\geq 20$, $|k_1-k_2|\leq 5$, and $k_3, k_4\leq k_1-10$, then we have
\begin{align}
\|\P_{k_4}\N(u_{1}, u _{2}, u_{3})\|_{N_{k_4}^\alpha}
\les \min(M_1, M_2)
\|u_{1}\|_{F_{k_1}^\alpha}\|u_{2}\|_{F_{k_2}^\al}\|u_3\|_{F_{k_3}^\alpha},\label{HHLL1}
\end{align}

\noi
where $M_j$, $j = 1, 2$, is given by 
\[ M_1 = 2^{(-\frac{1}{2}+ \alpha+\eps)k_1^* - \al k_4 -\beta_1}
\quad \text{with }
\be_1 = \begin{cases}
\frac{k_3^*}{2}, & \text{if }|k_3 -  k_4| \geq 2, \\
0, & \text{otherwise}, 
\end{cases}
\]

\noi
and

\[M_2 = \begin{cases}
2^{(-\frac{ 1}{2}+ \frac{1}{2} \al+\eps) k_1^* - \al k_4}, & 
\text{if }|k_3 - k_4 | \geq 2,\\
 2^{(-\frac{ 1}{2}+ \frac 12\al+\eps) k_1^*+ \frac{k_3^*}{4}- \al k_4} , &
 \text{otherwise}, 
\end{cases}\]

\noi
for any $\eps > 0$.

Moreover, 
when $j_3 = j_1^*$ and $k_3 \geq k_4 + 2$
 \eqref{HHLL1} holds with the constant given by 
\begin{equation}
M_3 = \min \big(
2^{-\frac{k_1^*}{2} + \frac{\al}{2}k_3+(\frac{1}{2}-\al+\eps) k_4},
2^{(-\frac{1}{2}+ \al)k_1^*-\frac{k_3}{2}+ (\frac{1}{2} - \al)k_4 -\frac{\min(\al k_1^*, k_3)}{2} }\big).
\label{HHLL3}
\end{equation}

\end{lemma}

While the estimate \eqref {HHLL1}
 holds under permutation of indices $k_1$, $k_2$, and $k_3$, 
 we have a better estimate when $k_1, k_3 \geq k_2 + 10$
 in view of \eqref{SC2}.

\begin{proof}
Proceeding as in the proof of Lemma \ref{LEM:HHHL},
\eqref{HHLL1} with $M_1$ follows from \eqref{L6} in Lemma \ref{LEM:L4L6}
with the support condition \eqref{SC1} and \eqref{SC2}.
Note that we have $j_1^* \geq k_1^* + k_3^* - 5$
when  $|k_3 - k_4 |\geq  2$.
Similarly, \eqref{HHLL1} with $M_2$ follows from Lemma \ref{LEM:L63}.

The argument above suffices to prove the main trilinear estimate
\eqref{3lin1} for $s > -\frac{1}{8}$
(see Summary (v) below)
except for the case $j_3 = j_1^*$ and $k_3 \geq k_4 + 2$,
where the argument above is  good only for $s> -\frac{1}{10}$.
See Remark \ref{REM:L66} and Summary (v) below.
In order to go below $s = -\frac{1}{10}$, 
we need a different argument
when $j_3 = j_1^*$ and $k_3 \geq k_4 + 2$.

Note that we have $j_3 \geq k_1^* + k_3 - 5$ from \eqref{SC1}.
In this case, we do not perform time-localization
corresponding to the highest (spatial) frequency.
Instead, 
we need to
perform time-localization
over time intervals of size $\sim 2^{-[\al k_3]}$.
Write $v_3 = \eta_0(2^{[\alpha k_4]}(t-t_{k_4})) \cdot u_3$ as a sum of $O(2^{\al(k_3 - k_4)})$ many
terms localized
over time intervals of size $\sim 2^{-[\al k_3]}$.
Then, by H\"older and Hausdorff-Young's inequalities with Lemmata \ref{LEM:infty}
and \ref{LEM:a.o.} (a),
we have
\begin{align}
\text{LHS of } \eqref{HHLL1}
& \sim \sum_{j_4 \geq [\al k_4]} 2^{-\frac{j_4}{2}}
\big\| \ind_{D_{k_4, \leq j_4}} \cdot ( \ft u_1 * \F(\cj u_2) * \ft v_{3}) \big\|_{\l^2_n L^2_\tau} \notag \\
& \les 2^{(\frac{1}{2}-\frac{\al}{2}+\eps) k_4}  \| u_1\|_{L^\infty_t L^2_x} \| u_2\|_{L^\infty_t L^2_x}
\|v_{3}\|_{L^2_t L^\infty_x}\notag \\
& \les2^{-\frac{j_3}{2}}
2^{\frac{k_3}{2}} 2^{(\frac{1}{2}-\frac{1}{2}\al+\eps) k_4}
 \| u_1\|_{F^\al_{k_1}} \| u_2\|_{F^\al_{k_2}}\|v_{3}\|_{F^\al_{k_3}} \notag \\
& \les
2^{-\frac{k_1^*}{2} + \frac{\al}{2}k_3+(\frac{1}{2}-\al+\eps) k_4}
 \| u_1\|_{F^\al_{k_1}} \| u_2\|_{F^\al_{k_2}}\|u_{3}\|_{F^\al_{k_3}},
\label{HHLL15}
\end{align}

\noi
yielding the first bound in \eqref{HHLL3}.

Note that the time localization step
affects the modulation.
In particular, that the modulation of $v_3$ may be much smaller than $2^{j_1^*}$.
Hence, 
strictly speaking,
we first
need to decompose $v_3$ as $v_3 = v_{3, +} + v_{3, -}$,
where $v_{3, +}$ is the high modulation piece defined
by $v_{3, +} = \F^{-1} [\eta_{\geq k_1^* + k_3 - 10}(\tau - n^2) \ft v_3 ]$
and $v_{3, -}= v_3 - v_{3, +}  $.
Then, apply the argument  in \eqref{HHLL15} to $v_{3, +}$.
As for $v_{3, -}$, we apply
 the time localization $ \g(2^{[\al k_1]+c}(t-t_{k_4})-m)$,
 then modulation localization,
and apply Lemma \ref{LEM:L63}
as before.
Note that  $\ft w_{3, m} = \F[\g(2^{[\al k_1]}(t-t_{k_4})-m)v_{3, -}]$
is essentially supported on $D_{k_3, \leq k_1^* + k_3 - 5}$,
and thus we can assume $\max(j_1, j_2, j_4)
\geq k_1^* + k_3 - 5 > j_3$.
Indeed, when $j_3 > k_1^* + k_3 - 5$,
then it follows from the support consideration that
\begin{align*}
|\ft w_{3, m} (n, \tau)|
& \leq \int 2^{-[\al k_1]} |\ft \g(2^{-[\al k_1]}(\tau - \tau ') )||\ft v_{3, -}(n, \tau')| d\tau'\\
& \les 2^{-10 k_1^*} \int 2^{-[\al k_1]} |\phi(2^{-[\al k_1]}(\tau - \tau ') )||\ft v_{3, -}(n, \tau')| d\tau'
\end{align*}

\noi
for some $\phi \in \S$.
Then, we can simply apply Cauchy-Schwarz inequality to
conclude this case.

The estimate \eqref{HHLL15} provides
a good bound for $s > -\frac{1}{8}$ when $k_3$ is small in comparison to $k_1$,
more precisely, when $k_3 \leq \frac{2}{3} k_1$.
Otherwise, i.e.~if $k_3 > \frac{2}{3} k_1$, we use the time localization argument
with the second claim in Lemma \ref{LEM:L63}.
Then, we obtain
\begin{align*}
\text{LHS of } \eqref{HHLL1}
& \les
2^{(-\frac{1}{2}+ \al)k_1^*-\frac{k_3}{2}+ (\frac{1}{2} - \al)k_4 -\frac{\min(\al k_1^*, k_3)}{2} }
 \| u_1\|_{F^\al_{k_1}} \| u_2\|_{F^\al_{k_2}}\|u_{3}\|_{F^\al_{k_3}}.
\end{align*}

\noi
This completes the proof of Lemma \ref{LEM:HHLL}.
\end{proof}

\begin{remark}\rm

When $j_4 \geq [\al k_1] $, we can use Lemma \ref{LEM:a.o.} (b)
to improve
the constant in \eqref{HHLL1} to
$2^{-(\frac{1}{2}-\eps)k_1^* -\frac{\al}{2}k_4}$
when $|k_3 - k_4|\geq  2$, 
and
$2^{-(\frac{1}{2}-\eps)k_1^* +\frac{k_3^*}{4} -\frac{\al}{2}k_4}$
when  $|k_3 - k_4| \leq 1$.
These estimates allow us to prove \eqref{3lin1}
 for $s > -\frac{1}{6}$
 in the high modulation case: $j_4 \geq [\al k_1] $.
\end{remark}

\begin{lemma}[low $\times$ low $\times$ low $\to$ low]\label{LEM:LLLL}
If $k_1,k_2,k_3,k_4\leq 200$, then we have
\begin{align*}
\|\P_{k_4}\N(u_{1}, u_{2}, u_{3})\|_{N_{k_4}^\alpha}
+ \|\P_{k_4}\RR & (u_{1}, u_{2}, u_{3})\|_{N_{k_4}^\alpha} 
 \les \|u_{1}\|_{F_{k_1}^\alpha}\|u_{2}\|_{F_{k_2}^\al}\|u_3\|_{F_{k_3}^\alpha}.
\end{align*}
\end{lemma}
\begin{proof}
This follows immediately from Young's inequality on \eqref{HLL13}
with \eqref{Xk3}.
\end{proof}

\noi
{\bf Summary:}
We summarize the regularity restrictions
from Lemmata \ref{LEM:HLL} - \ref{LEM:LLLL}.
In the following, we assume $s < 0$.

\begin{itemize}
\item[(i)] high $\times$ low $\times$ low $\to$ high:
\quad In view of Lemma \ref{LEM:HLL}, we need to have
\begin{align*}
 s k_4 - \tfrac{1}{2}(1-\eps)  k_1^* \leq s(k_1 + k_2 + k_3)\quad
\LLRA\quad   s (k_3^* + k_4^*) \geq - \tfrac{1}{2}(1-\eps) k_1^*.
\end{align*}

\noi
The latter holds for  $s\geq - \frac{1}{4}(1-\eps)$.

\medskip

\item[(ii)] high $\times$ high $\times$ low $\to$ high:
\quad In view of Lemma \ref{LEM:HHLH}, we need to have
\begin{align*}
 s k_4 - (1-\eps)  k_1^* \leq s(k_1 + k_2 + k_3)\quad
\LLRA\quad   s (k_1^* + k_4^*) \geq - (1-\eps) k_1^*.
\end{align*}

\noi
The latter holds for  $s\geq - \frac{1}{2}(1-\eps)$.

\medskip

\item[(iii)] high $\times$ high $\times$ high $\to$ high:
\quad In view of Lemma \ref{LEM:HHHH}, we need to have
$- \frac{1}{2}(\al - \eps)  \leq 2s$.
Hence, it suffices to choose
\begin{equation}
\al =   -4s + \eps
\label{alpha}
\end{equation}

\noi
for some sufficiently small $\eps = \eps(s) > 0$.

\medskip

\item[(iv)] high $\times$ high $\times$ high $\to$ low:
\quad
In this case, we use Lemma \ref{LEM:HHHL}.
In particular, from \eqref{HHHL1} with $M_1$,
we need to have
\begin{align}
-(1-\al -\eps) \leq 3s.
\label{S1a}
\end{align}

\noi
In view of \eqref{alpha},
this holds for $s \geq -\frac{1}{7}+\frac{2}{7}\eps$.
Next, we consider the case $s \leq -\frac{1}{7}$.
From \eqref{HHHL1} with $M_2$,
we need to have
\begin{align}
-(1-\tfrac{3}{4}\al -\eps)k_1^* + (s + \tfrac{1}{2} - \al) k_4^* \leq 3sk_1^*.
\label{S2}
\end{align}

\noi
In view of \eqref{alpha}, we have $s \geq -\frac{1}{6} + \frac{7}{24}\eps$
from the coefficients of $k_1^*$,
while $s\leq -\frac{1}{10} + \frac{\eps}{5}$ guarantees 
that the coefficient of $k_4^*$ is non-positive.
Hence, \eqref{S1a} or \eqref{S2} holds for $s \in (-\frac{1}{6}, 0)$.

\medskip

\item[(v)] high $\times$ high $\times$ low $\to$ low:
\quad
First, we consider $s> -\frac{1}{12}$.
In view of  \eqref{HHLL1} with $M_1$ of Lemma \ref{LEM:HHLL}, we need to have
\begin{align}
(-\tfrac{1}{2} + \al + \eps) k_1^* \leq 2sk_1^*
\qquad \text{and} \qquad
(s- \al) k_4 -\be \le s k_3.
\label{S2a}
\end{align}

\noi
From \eqref{alpha}, the first condition provides
$s \geq -\frac{1}{12} + \frac{1}{3}\eps$.
The second condition is trivially satisfied when $k_4 \geq k_3 - 5$.
When $k_3 \geq k_4 + 5$, it gives $s \geq -\frac{1}{2}$.

Next, we consider $s\leq -\frac{1}{12}$.
First, we consider the case  $ |k_3 - k_4| \leq 1$.
In view of  \eqref{HHLL1} with $M_2$ in  Lemma \ref{LEM:HHLL}, we need to have
\begin{align}
(-\tfrac{ 1}{2} + \tfrac{ 1}{2} \al+ \eps) k_1^* \le 2 s k_1^*
\quad \text{and}\quad
 (\tfrac{1}{4}- \al)k_3^*   \leq 0.
\label{S3}
\end{align}

\noi
From \eqref{alpha}, the first condition provides
$s \geq -\frac{1}{8} + \frac{3}{8} \eps$,
while the second condition provides $s\leq -\frac{1}{16} +\frac{\eps}{4}$.

Next, let us consider the case
when  $k_3 \geq k_4 +2$.
(The case $k_4 \geq k_3 +2$ is easier.
In this case, the condition yields $s \in (-\frac{1}{8}, 0)$.)
In this case,
from   \eqref{HHLL1} with $M_2$ in  Lemma \ref{LEM:HHLL}, we need to have
\begin{align}
s k_4+(-\tfrac{ 1}{2}+ \tfrac12 \al + \eps) k_1^*
- \al k_4 \leq 2s k_1^* + sk_3 .
\label{S4}
\end{align}

\noi
When $|n_3|\gg |n_4|$,
there is no gain of regularity on $n_3$ from $n_4$.
Namely, with \eqref{alpha}
we have $s\geq -\frac{1}{10} + \frac{3}{10}\eps$
by collecting the coefficients of $k_1^*$ and $k_3$.
The coefficients of $k_4$ give $s \leq \frac{1}{5}\eps$,
i.e.~no restriction by taking $\eps > 0$ sufficiently small.

Lastly, we consider $s\leq -\frac{1}{10}$.
Suppose that $k_3 \leq \frac{2}{3} k_1$.
Then, from \eqref{HHLL3},
we have
\begin{equation}
sk_4 -\tfrac{k_1^*}{2} + \tfrac{\al}{2}k_3+(\tfrac{1}{2}-\al+\eps) k_4\leq 2sk_1^* + sk_3.
\label{S5}
\end{equation}

\noi
By writing
$ -\frac{k_1^*}{2} \leq   -(\frac{1}{4}-\eps)k_1^* -\frac{3}{8}k_3
- \frac{3}{2}\eps k_4$,
the coefficients of $k_1^*$ and $k_3$ yield $s \geq  -\frac{1}{8}+ \frac{\eps}{2}$
and $s \geq  -\frac{1}{8}+ \frac \eps 6$, respectively, 
while the coefficients of $k_4$ yield $ s\leq -\frac{1}{10}+\frac{3}{10}\eps$.
When $k_3 > \frac{2}{3} k_1$,
the second term in \eqref{HHLL3} gives
\begin{equation}
sk_4 + (-\tfrac{1}{2}+ \al)k_1^*-\tfrac{k_3}{2}+ (\tfrac{1}{2} - \al)k_4 -\tfrac 12\min(\al k_1^*, k_3)
\leq 2sk_1^* + sk_3.
\label{S6}
\end{equation}

\noi
In view of  \eqref{S2}, we have 
$\min(\al k_1^*, k_3) = \al k_1^*$ 
for $s > -\frac 18$.
Thus, \eqref{S6} yields the condition $s \in (-\frac{1}{8}, -\frac{1}{40})$.

Hence, (v.a) \eqref{S2a}, or 
(v.b) \eqref{S3} and one of  \eqref{S4}, \eqref{S5} or \eqref{S6} holds for $s \in (-\frac{1}{8}, 0)$.

\medskip

\item[(vi)]
low $\times$ low $\times$ low $\to$ low:
There is no condition imposed in this case.
\end{itemize}

\medskip

We conclude this section by presenting
 the proof of Proposition \ref{PROP:3lin}.

\begin{proof}[Proof of Proposition \ref{PROP:3lin}]
First, we consider the first inequality \eqref{3lin1}.
From the definitions, we
have
\begin{align}
\|\N(u_1,u_2,u_3)\|_{N^{s,\alpha}(T)}^2
& =\sum_{k_4=0}^{\infty}2^{2sk_4}\|\P_{k_4}\N(u_1,u_2,u_3)\|_{N_{k_4}^\alpha(T)}^2.
\label{3lin3}
\end{align}

\noi
Applying a dyadic decomposition, we have
\begin{align}
\|\P_{k_4}\N(u_1,u_2,u_3)\|_{N_{k_4}^\alpha(T)}
\leq \sum_{k_1,k_2,k_3\in \Z_+}
\|\P_{k_4}\N(\P_{k_1}u_1,\P_{k_2}u_2,\P_{k_3}u_3)\|_{N_{k_4}^\alpha(T)}.
\label{3lin4}
\end{align}

\noi
Given  $k\in \Z_+$, let  $ u_{k_1}$ be an extension of $\P_{k_i}(u_i)$ such that
$\|u_{k_i}\|_{F_k^\alpha}\leq 2\|\P_{k_i} u_i\|_{F_k^\alpha(T)}$.
In view of   \eqref{Xk3}, we can assume that
$u_{k_i}$ is supported on $[-2T, 2T]$.
Then, we have
\begin{align}
\|\P_{k_4}\N(\P_{k_1}u_1,\P_{k_2}u_2,\P_{k_3}u_3)\|_{N_{k_4}^\alpha(T)}
\leq  \sum_{k_1,k_2,k_3\in \Z_+}
\|\P_{k_4}\N(u_{k_1},u_{k_2},u_{k_3})\|_{N_{k_4}^\alpha}.
\label{3lin5}
\end{align}

\noi
By Lemmata \ref{LEM:HLL} - \ref{LEM:LLLL} and Summary (i) - (vi)
along with Lemma \ref{LEM:timedecay},
we obtain
\begin{align}
\|\P_{k_4}\N(u_{k_1},u_{k_2},u_{k_3})\|_{N_{k_4}^\alpha}
\les T^\theta \prod_{i = 1}^3 2^{(s-\eps)k_i} \|u_{k_i}\|_{F_{k_i}^\al}
\label{3lin6}
\end{align}

\noi
for some small $\theta, \eps >0$.
Here, in order to obtain a small power of $T$, we actually applied
Lemmata \ref{LEM:HLL} - \ref{LEM:LLLL} with the right-hand sides
with $F^{b, \al}_{k_i}$ (see Lemma \ref{LEM:timedecay}) for some  $b = \frac{1}{2}-$
 at a slight loss of the regularity in the highest modulation.
Putting \eqref{3lin3} - \eqref{3lin6} together,
we obtain \eqref{3lin1}.
The second estimate \eqref{3lin2}
follows from a similar argument
with Lemmata \ref{LEM:HHHH} and \ref{LEM:LLLL}.
\end{proof}

\section{Energy estimates on smooth solutions}
\label{SEC:energy}

In this section, we establish an energy estimate for  (smooth) solutions
to the Wick ordered cubic NLS \eqref{NLS1}.
This argument is very close in spirit
to the $I$-method developed by Colliander-Keel-Staffilani-Takaoka-Tao
\cite{CKSTT3, CKSTT2, CKSTT1}.
Let $u \in C(\R; H^\infty(\T))$ be a smooth solution to \eqref{NLS1}.
Then,
by Fundamental Theorem of Calculus, we have
\begin{align}
\|u(t)\|_{H^s}^2 - \|u(0)\|_{H^s}^2
& = \|a(t)\|_{H^s}^2 - \|a(0)\|_{H^s}^2
= 2 \text{Re}\, \bigg( \int_0^t \sum_n \jb{n}^{2s} \dt a_n(t') \cj{a}_n(t') dt'\bigg) \notag \\
& =  2 \text{Re}\, i \bigg( \int_0^t \sum_n \jb{n}^{2s} \textsf{N} (a)_n   \cj{a}_n(t') dt'\bigg) \notag \\
& \hphantom{XXXX}- 2 \text{Re}\, i \bigg( \int_0^t \sum_n \jb{n}^{2s} \textsf{R} (a)_n \cj{a}_n(t') dt'\bigg) \notag \\
& =   2 \text{Re}\, i \bigg( \int_0^t
\sum_n \jb{n}^{2s}
\sum_{\substack{n = n_1 - n_2 + n_3  \\ n_2\ne n_1, n_3} }
e^{- i \Phi(\bar{n})t' }
a_{n_1} \cj{a}_{n_2}a_{n_3}\cj{a}_{n} (t') dt'\bigg) \notag \\
& \hphantom{XXXX} - 2 \text{Re}\, i \bigg( \int_0^t \sum_n \jb{n}^{2s} |a_n(t')|^4 dt'\bigg) \label{E1},
\end{align}

\noi
where $\textsf N(a)$ and $\textsf R(a)$ are as in \eqref{NLS2}
and $\Phi(\bar{n})$ is as in \eqref{Phi}.
Here, $a_n(t) = e^{-itn^2} \ft u_n(t)$
denotes (the Fourier coefficient of) the interaction representation of $u$
defined in Section \ref{SEC:notations}.
Clearly, the second term on the right-hand side of \eqref{E1} is 0.
Moreover,
letting $n_4 = n$ and
symmetrizing under the summation indices $n_1, \dots, n_4$,
we obtain
\begin{align}
\|u(t)\|_{H^s}^2 - \|u(0)\|_{H^s}^2
& =-  \frac{i}{2}  \int_0^t
\sum_{\substack{n_1 - n_2 + n_3-n_4 = 0 \\ n_2\ne n_1, n_3} }
\Psi_s(\bar{n})
e^{- i \Phi(\bar{n})t' }
a_{n_1} \cj{a}_{n_2}a_{n_3}\cj{a}_{n_4} (t') dt'\notag\\
& =: R_4(t),  \label{E2}
\end{align}

\noi
where $\Psi_s(\bar{n})$ is defined by
\begin{equation} \label{Psi}
\Psi_s(\bar{n}) = \jb{n_1}^{2s} - \jb{n_2}^{2s} + \jb{n_3}^{2s} - \jb{n_4}^{2s}
\end{equation}

\noi
and $\Phi(\bar{n})$ is as in \eqref{Phi} with $n$ replaced by $n_4$.

\begin{remark} \rm
It is this symmetrization process that fails when we try to establish an energy estimate
for a
difference of two solutions.
The symbol $\Psi_s(\bar{n})$  provides an extra decay
via Mean Value Theorem and Double Mean Value Theorem
(Lemmata 4.1 and 4.2 in \cite{CKSTT1}.)
See \eqref{DMT1}, \eqref{MVT1}, and \eqref{MVT2}.
This is crucial in estimating $R_4(t)$  in the {\it nearly resonant} case
(i.e. $|\Phi(\bar n)| \ll n_1^*$)
below $L^2(\T)$.
\end{remark}

With the $L^6$-Strichartz estimate \eqref{Stri3}
and its refinement (Lemma \ref{LEM:L64}),
we can estimate $R_4$ in \eqref{E2}
and obtain  an energy estimate for $s > -\frac{1}{10}$.
However, for $s \leq -\frac{1}{10}$,
we need to  add a ``correction term'' to \eqref{E2} as in the application of the $I$-method
\cite{CKSTT2, CKSTT1}.
In terms of the interaction representation $a_n(t)$, 
this process can be regarded as the Poincar\'e-Dulac normal form reduction
applied to the evolution equation satisfied by
$\|a(t)\|_{H^s}^2$.
In fact, integrating \eqref{E2} by parts, we have
\begin{align*}
\|u(t)\|_{H^s}^2 -  \|u(0)\|_{H^s}^2
&   =    \frac{1}{2}
\sum_{\substack{n_1 - n_2 + n_3-n_4 = 0 \\ n_2\ne n_1, n_3} }
\frac{\Psi_s(\bar{n})}{ \Phi(\bar{n})}
e^{- i \Phi(\bar{n})t' }
a_{n_1} \cj{a}_{n_2}a_{n_3}\cj{a}_{n_4} (t')  \bigg|_0^t \notag\\
& \hphantom{X}
-  \frac{1}{2}  \int_0^t
\sum_{\substack{n_1 - n_2 + n_3-n_4 = 0 \\ n_2\ne n_1, n_3} }
\frac{\Psi_s(\bar{n})}{\Phi(\bar{n})}
e^{- i \Phi(\bar{n})t' }
\dt (a_{n_1} \cj{a}_{n_2}a_{n_3}\cj{a}_{n_4} )(t') dt' \notag \\
& =: \Ld_4(t; u) - \Ld_4(0; u) + R_6(t; u), 
\end{align*}

\noi
where the correction term $\Ld_4(t;u)$ is given by
\begin{equation*}
\Ld_4(t) = \Ld_4(t;u)   =    \frac{1}{2}
\sum_{\substack{n_1 - n_2 + n_3-n_4 = 0 \\ n_2\ne n_1, n_3} }
\frac{\Psi_s(\bar{n})}{ \Phi(\bar{n})}
\ft{u}_{n_1} \cj{\ft{u}}_{n_2}\ft{u}_{n_3}\cj{\ft{u}}_{n_4} (t) .
\end{equation*}

\noi
For simplicity, we assume that the time derivative in $R_6(t)$
falls on the first factor.
The same comment applies to $R^M_6(t)$ defined in \eqref{EM4}.
Using the equation \eqref{NLS2}, we can then write $R_6(t):= R_6(t; u)$  as
\begin{align*}
R_6(t)
& = c \int_0^t
\sum_{\substack{n_1 - n_2 + n_3-n_4 = 0 \\ n_2\ne n_1, n_3} }
\frac{\Psi_s(\bar{n})}{\Phi(\bar{n})}
e^{- i \Phi(\bar{n})t' }  \notag \\
& \hphantom{XXXX} \times
\sum_{\substack{n_1 = n_5 - n_6 + n_7\\ n_6\ne n_5, n_7} }
e^{- i \wt{\Phi}(\bar{n})t' }
a_{n_5} \cj{a}_{n_6}a_{n_7}
 \cj{a}_{n_2}a_{n_3}\cj{a}_{n_4} (t') dt' \notag \\
&  \hphantom{XX}
+ c \int_0^t
\sum_{\substack{n_1 - n_2 + n_3-n_4 = 0 \\ n_2\ne n_1, n_3} }
\frac{\Psi_s(\bar{n})}{\Phi(\bar{n})}
e^{- i \Phi(\bar{n})t' }
|a_{n_1}|^2 a_{n_1}\cj{a}_{n_2}a_{n_3}\cj{a}_{n_4} (t') dt' \notag \\
& = c \int_0^t
\sum_{\substack{n_1 - n_2 + n_3-n_4 = 0 \\ n_2\ne n_1, n_3} }
\frac{\Psi_s(\bar{n})}{\Phi(\bar{n})}
\sum_{\substack{n_1 = n_5 - n_6 + n_7\\ n_6\ne n_5, n_7} }
\ft{u}_{n_5} \cj{\ft{u}}_{n_6}\ft{u}_{n_7}
 \cj{\ft{u}}_{n_2}\ft{u}_{n_3}\cj{\ft{u}}_{n_4} (t') dt' \notag \\
& \hphantom{XX}
+ c \int_0^t
\sum_{\substack{n_1 - n_2 + n_3-n_4 = 0 \\ n_2\ne n_1, n_3} }
\frac{\Psi_s(\bar{n})}{\Phi(\bar{n})}
|\ft{u}_{n_1}|^2 \ft{u}_{n_1}\cj{\ft{u}}_{n_2}\ft{u}_{n_3}\cj{\ft u}_{n_4} (t') dt', 
\end{align*}

\noi
where the phase function $\wt{\Phi}(\bar{n})$ is given by
\begin{align}\label{Phim}
\wt{\Phi}(\bar{n})& = \Phi(n_5, n_6, n_7, n_1)
= n_1^2 - n_5^2 + n_6^2- n_7^2 \notag \\
& = 2(n_6 - n_5) (n_6 - n_7)
= 2(n_1 - n_5) (n_1 - n_7).
\end{align}

The  boundary term $\Ld_4$ corresponds to the correction term in an application of  the $I$-method.
While $\Ld_4$ satisfies a good estimate in terms of the regularity,
it does not see the length of the time interval.
In order to gain a small power of $T>0$, we need to modify the argument above.

Let $R_4^M(t)$ be the part of $R_4(t)$, 
where all the frequencies $|n_j|\leq M$.
Namely, we have
\begin{align*}
R_4^M(t)
& = - \frac{i}{2}  \int_0^t
\sum_{\substack{n_1 - n_2 + n_3-n_4 = 0 \\ n_2\ne n_1, n_3\\|n_j|\leq M} }
\Psi_s(\bar{n})
e^{- i \Phi(\bar{n})t' }
a_{n_1} \cj{a}_{n_2}a_{n_3}\cj{a}_{n_4} (t') dt'\notag \\
& = - \frac{i}{2}   \int_0^t
\sum_{\substack{n_1 - n_2 + n_3-n_4 = 0 \\ n_2\ne n_1, n_3\\|n_j|\leq M} }
\Psi_s(\bar{n})
\ft u_{n_1} \cj{\ft u }_{n_2}\ft u _{n_3}\cj{\ft u }_{n_4} (t') dt'.
\end{align*}

\noi
Then, by applying a normal form reduction to  $R_4(t) - R_4^M(t)$, we have
\begin{align*}
\|u(t)\|_{H^s}^2 -  \|u(0)\|_{H^s}^2
&   =  R_4^M(t;u)
+ \Ld_4^M(t; u) - \Ld_4^M(0; u) + R_6^M(t; u), 
\end{align*}

\noi
where $\Ld_4^M(t) : = \Ld_4^M(t;u)$ and  $R_6^M(t) := R_6^M(t;u)$ are given by
\begin{equation*}
\Ld_4^M(t) =     \frac{1}{2}
\sum_{\substack{n_1 - n_2 + n_3-n_4 = 0 \\ n_2\ne n_1, n_3\\
\max |n_j| > M} }
\frac{\Psi_s(\bar{n})}{ \Phi(\bar{n})}
\ft{u}_{n_1} \cj{\ft{u}}_{n_2}\ft{u}_{n_3}\cj{\ft{u}}_{n_4} (t)
\end{equation*}

\noi
and
\begin{align}
R_6^M(t)
& = c \int_0^t
\sum_{\substack{n_1 - n_2 + n_3-n_4 = 0 \\ n_2\ne n_1, n_3\\\max |n_j| > M} }
\frac{\Psi_s(\bar{n})}{\Phi(\bar{n})}
\sum_{\substack{n_1 = n_5 - n_6 + n_7\\ n_6\ne n_5, n_7} }
\ft{u}_{n_5} \cj{\ft{u}}_{n_6}\ft{u}_{n_7}
 \cj{\ft{u}}_{n_2}\ft{u}_{n_3}\cj{\ft{u}}_{n_4} (t') dt' \notag \\
& \hphantom{X}
+ c \int_0^t
\sum_{\substack{n_1 - n_2 + n_3-n_4 = 0 \\ n_2\ne n_1, n_3\\\max |n_j| > M} }
\frac{\Psi_s(\bar{n})}{\Phi(\bar{n})}
|\ft{u}_{n_1}|^2 \ft{u}_{n_1}\cj{\ft{u}}_{n_2}\ft{u}_{n_3}\cj{\ft u}_{n_4} (t') dt'
=:\I(t)+\II(t). \label{EM4}
\end{align}

In the remaining part of this section, we establish  multilinear estimates on
$R_4^M$, $\Ld_4^M$,  and $R_6^M$.
In the following,  $u$ denotes a smooth solution to \eqref{NLS1}.
The main tool 
is the following refinement of the $L^6$-Strichartz estimate.

\begin{lemma} \label{LEM:L64}

Let  $f_i$ be supported in $D_{k_i,\leq j_i}$, $i = 1, 2, 3$.
Then, the following estimate holds:
\begin{align*}
\|\ind_{I_{k_4}}(n)\cdot f_1*\wt{f_2}*f_3\|_{\l^2_n L^2_\tau }
\les 
M
\prod_{i=1}^3 2^\frac{j_i}{2}\|f_i\|_{\l^2_n L^2_\tau},
\end{align*}

\noi
where 
 $\wt{f_2}(n_2, \tau_2) = f_2(- n_2, -\tau_2)$.
Here, 
$M$ is given as follows.

\smallskip

\noi
\textup{(a)} If $|k_1 - k_4|\geq 10$ and $\max(k_1, k_4) \geq k_1^* - 5$, 
then we have 
\begin{align}\label{L64a}
M = 
\min (2^\frac{k_1}{2}, 2^\frac{k_4}{2}) 
\min_{i = 2, 3}\big\{ (1+2^{j_i-k_1^*})^\frac{1}{2}2^\frac{-j_i}{2}\big\}.
\end{align}

\noi
\textup{(b)} If $k_2 \leq k_4 -10 \leq k_1^* - 20$
or  $|n_1 - n_3| \sim 2^{k_1^*}$, 
then we have 
\begin{align}\label{L64b}
M = 
2^\frac{k_2}{2}
\min_{i = 1, 3} \big\{(1+2^{j_i-k_1^*})^\frac{1}{2}2^\frac{-j_i}{2}\big\}. 
\end{align}

\end{lemma}

\begin{proof}
The first bound  \eqref{L64a} follows from duality 
and (a slight modification of) Cases (a) and (b) in the proof of Lemma \ref{LEM:L62}.
The second bound \eqref{L64b} follows from 
 \eqref{L63a1} and \eqref{L63a2} in the proof of Lemma \ref{LEM:L63}.
In either case, we have 
$|\dd_{n_3} h_1(n_3)| = |2(n_2 - 2n_3 + n_4)| 
= |2(n_1 - n_3)| \sim 2^{k_1^*}$,
where $h_1$ is as in \eqref{L63a2}.
Thus, we have
\begin{align*}
 \eqref{L63a1}
& \les (1+ 2^{j_1 - k_1^*})^\frac{1}{2}2^{-\frac{j_1}{2}}\int \sum_{n_4} |g_4(n_4, \tau_4)|
\sum_{n_2}
|g_2(n_2, \tau_2)| \bigg(\sum_{n_3}|g_3(n_3, \tau_3)|^2 \notag\\
& \hphantom{XXXXX}
\times
\big|g_1\big(n_2-n_3+n_4, h_1(n_2, n_3, n_4, \tau_2, \tau_3, \tau_4)\big)\big|^2\bigg)^\frac{1}{2}
d\tau_2d\tau_3d\tau_4\\
& \les (1+ 2^{j_1 - k_1^*})^\frac{1}{2}2^{-\frac{j_1}{2}}\cdot 2^\frac{k_2}{2}  \|g_4(n_4, \tau_4)\|_{\l^2_{n_4}L^2_{\tau_4}}
\int \|g_2(n_2, \tau_2)\|_{\l^2_{n_2}} \sup_{n_2}  \bigg(\sum_{n_3, n_4}|g_3(n_3, \tau_3)|^2 \notag\\
& \hphantom{XXXXX}
\times
\big\|g_1\big(n_2-n_3+n_4, h_1(n_2, n_3, n_4, \tau_2, \tau_3, \tau_4)\big)\big\|_{L^2_{\tau_4}}^2\bigg)^\frac{1}{2}
d\tau_2d\tau_3\\
& \les  
(1+ 2^{j_1 - k_1^*})^\frac{1}{2}2^{-\frac{j_1}{2}}\cdot 2^\frac{k_2}{2}
\|\F(u_4)\|_{\l^2_n L^2_\tau}
\prod_{i = 1}^3 2^\frac{ j_i}{2} \|\F(u_i)\|_{\l^2_n L^2_\tau}.
\end{align*}

\noi
Then, the second bound \eqref{L64b} follows from 
duality.
 \end{proof}

\begin{proposition} \label{LEM:R4M}
Let $s \in (-\frac 14, 0)$ and  $\al = -4s + \eps$ as in \eqref{alpha}.
Then, there exists $\theta > 0$ such that 
\begin{align}
|R_4^M(T; u)|
& \lesssim
T^\theta  M^{c(s)} \|u\|_{F^{s, \al}(T)}^4
\label{R4M1}
\end{align}

\noi
for $T\in (0, 1]$,
where $c(s)$ is defined by 
\begin{align}
 c(s) = \max(-\tfrac 12 - 5s+, 0).
\label{R4M1x}
\end{align}

\end{proposition}

\begin{remark}\label{REM:ene}\rm
It follows from \eqref{R4M1x}
that we have $c(s) = 0$ for  $s > -\frac{1}{10}$.
Namely,   for $s > -\frac{1}{10}$, 
we do not need to add
a correction term. 
In this case, 
\eqref{E2} and
Proposition \ref{LEM:R4M} (with $M = \infty$) yield
a good energy estimate:
\[
\big|\|u\|_{E^s(T)}^2 - \|u(0)\|_{H^s}^2\big|
\leq |R_4(T)|
 \les  T^\theta\|u\|_{F^{s, \al}(T)}^4,\]

\noi
allowing us to prove 
 the existence result (Theorem \ref{THM:1})
simply by
Proposition \ref{LEM:R4M}
along with the argument in Section \ref{SEC:existence}.
This 
is basically the periodic analogue of the result in \cite{CCT3}.

When $s \leq -\frac{1}{10}$, we have 
$c(s) > 0$ (see \eqref{R4M4} below)
and hence we need to add the correction term $\Ld_4^M$.
\end{remark}

\begin{proof}
Apply a dyadic decomposition on
the spatial frequencies
$|n_i| \sim 2^{k_i}$, $i = 1, \dots, 4$.
By symmetry, assume that 
$ |n_1|\sim n_1^*  : = \max (|n_1|,|n_2|,|n_3|,|n_4|) \leq M.$
By assuming that each factor $u_i$ has its Fourier support on $I_{k_i}\times \R$,
it suffices to prove
\begin{equation*}
|R_4^M(T)|\les T^\theta M^{c(s)}  \prod_{i=1}^4 2^{(s-)k_i}\|\P_{k_i}u\|_{F_{k_i}^\alpha(T)}.
\end{equation*}

\noi
Here, a slight extra decay is needed to sum over dyadic blocks.
Let  $\wt u_i$ be an extension of $u_i$ such that 
$\|\wt u_i\|_{F_{k_i}^\alpha}\leq 2 \|\P_{k_i}u\|_{F_{k_i}^\alpha(T)}$.
For
notational simplicity, we denote $\wt u_i$ by $u_i$,
$i = 1, \dots,  4$, in the following. 
Letting $\g:\R \to [0, 1]$ be a smooth cutoff function supported on $[-1, 1]$ such that
$\sum_{m\in \Z} \g^4(t - m) \equiv 1$ for all  $t \in \R$, 
we have
\begin{align*}
R_4^M(T)=\int_{\R} \ind_{[0,T]}(t)
& \sum_{|m|\leq  T2^{[\alpha K]}}\g^4(2^{[\alpha K ]}t - m)
\notag \\
& 
\times
\sum_{\substack{n_1 - n_2 + n_3-n_4 = 0 \\
n_2\ne n_1, n_3\\ |n_j| \leq M} }
\Psi_s(\bar{n})
\ft{u}_1({n_1})
 \cj{\ft{u}_2}({n_2})\ft{u}_3({n_3})\cj{\ft{u}_4}({n_4}) (t) dt .
\end{align*}

\noi
Here, $K = k_1^* + c$.
With  $ f_{i, m}(t)=
\g(2^{[\al  K]}t - m)u_i(t)$, 
let 
$ f_{i, j_i, m} = \F^{-1}\big[\eta_{j_i}(\tau - n^2) \ft f_{i, m}\big]$,
$ i = 1, \dots,  4$.
Then, 
 it suffices to prove 
\begin{align}\label{R43}
 \bigg|\int_{\R}    \ind_{[0,T]}(t)
&  \sum_{j_1, \dots, j_4}\sum_{|m|\leq T2^{[\alpha K]}}
 \sum_{\substack{n_1 - n_2 + n_3-n_4 = 0 \\
n_2\ne n_1, n_3\\ |n_j| \leq M} }
  \Psi_s(\bar{n})
 \ft{f}_{1, j_1, m}({n_1}) 
 \cj{\ft{f}_{2, j_2, m}}({n_2}) \notag \\
 & \times \ft{f}_{3, j_3, m}({n_3})\cj{\ft{f}_{4, j_4, m}}({n_4}) (t) dt \bigg|
  \les T^\theta M^{c(s)}  \prod_{i=1}^4 2^{(s-)k_i}\|u_i\|_{F_{k_i}^\alpha}.
\end{align}

\noi
In view of \eqref{Xk2}, we assume that 
$j_i \geq \al K$,  $i = 1, \dots, 4$.
In the following,  we prove \eqref{R43}
for fixed $j_i$, $i = 1, \dots, 4$.
For simplicity,  
we denote $f_{i, j_i, m}$  by $f_{i, j_i}$.
Lastly, note that 
from \eqref{Phi} with $\tau_1-\tau_2+\tau_3-\tau_4 = 0$, we have
\[\s_1^* := \max (|\s_1|, |\s_2|, |\s_3|, |\s_4|) \gtrsim |\Phi(\bar{n})|,\]

\noi
where $\s_j = \tau_j - n_j^2$.

Define the subsets $\mathcal{A}$ and $\mathcal{B}$ of 
$\big\{m\in \Z: |m| \leq T 2^{[\al K]}\big\}$ by 
\begin{align*}
  \mathcal{A} & = \big\{ m \in \Z: 
\ind_{[0,T]}(t)\g(2^{[\alpha K]}t - m) =  \g(2^{[\alpha K]}t - m) \big\}, \\
  \mathcal{B} & = \big\{ m \in \Z: 
\ind_{[0,T]}(t)\g(2^{[\alpha K]}t - m) \ne  \g(2^{[\alpha K]}t - m)
\text{ and } \ind_{[0,T]}(t)\g(2^{[\alpha K]}t - m) \not \equiv 0 \big\}.
\end{align*}

\medskip
\noi
 {\bf Part 1:}
First, we consider the terms
with $m \in \mathcal{A}$.
In this case, we can drop the sharp cut-off $\ind_{[0,T]}(t)$ on the left-hand side
of \eqref{R43}.
We prove \eqref{R43} with $\theta = 1$ in this case.

\medskip
\noi
$\bullet$ {\bf Case (a):}  $|n_4 - n_1| , |n_4 - n_3| \ll n_1^*$.
(This case includes the nearly resonant case:
$|\Phi(\bar{n})| \ll n_1^*$.)
\\
\indent
Since $n_1 - n_2 + n_3 - n_4 = 0$, it follows that
$|n_1|\sim |n_2|\sim |n_3|\sim |n_4|\sim n_1^*$ in this case.
Then, by Double Mean Value Theorem \cite[Lemma 4.2]{CKSTT1}, we have
\begin{equation} \label{DMT1}
|\Psi_s(\bar{n})|\lesssim (n_1^*)^{2s-2}|(n_4 -n_1)(n_4 - n_3)|
\sim (n_1^*)^{2s-2} |\Phi(\bar n)|.
\end{equation}

\noi
By crudely estimating $|\Phi(\bar n)| \les (n_1^*)^2$, 
it follows from \eqref{DMT1} with \eqref{alpha} that
\begin{align}
\sum_{|m|\leq T2^{[\alpha K]}}
(\s_1^*)^{-\frac{1}{2}}|\Psi_s(\bar n)|
\les T (n_1^*)^{-1-2s+}
\leq T M^{c(s)} (n_1^*)^{4s-},
\label{R4M1a}
\end{align}

\noi
where
\begin{equation}
c(s) = \max(-1 -6s + , 0).
\label{R4M2}
\end{equation}

\noi
Then, by Lemma \ref{LEM:L4L6} with \eqref{R4M1a}, we obtain 
\begin{align*}
|R_4^M(T)|
& \les
T M^{c(s)} \prod_{i = 1}^4 2^{(s-) k_i} \| u_i\|_{F^\al_{k_i}}.
\end{align*}

\medskip
\noi
$\bullet$  {\bf Case (b):} $|n_4 - n_1| \sim n_1^* \gg |n_4 - n_3| $.
\\
\indent
In this case, we have $|n_2|\sim|n_1|\sim n_1^*$. 
Then, by Mean Value Theorem, we have
\begin{equation} \label{MVT1}
|\jb{n_1}^{2s} - \jb{n_2}^{2s}| \lesssim (n_1^*)^{2s-1} |n_1 - n_2|
= (n_1^*)^{2s-1} |n_4 - n_3|.
\end{equation}

\noi
$\circ$ Subcase (b.i):
$|n_4 - n_3| \ll n_3^*$.
\\
\indent
In this case,  we  have $|n_3|\sim|n_4|$. Then, 
by Mean Value Theorem, we have
\begin{equation} \label{MVT2}
|\jb{n_3}^{2s} - \jb{n_4}^{2s}| \lesssim (n_3^*)^{2s-1} |n_4 - n_3|.
\end{equation}

\noi
From \eqref{MVT1} and \eqref{MVT2}, we have
\begin{align}
|\Psi_s(\bar{n})|\lesssim  (n_3^*)^{2s-1}|n_4 - n_3|.
\label{MVT2a}
\end{align}

\noi
Then, with \eqref{alpha} and \eqref{MVT2a}, we have
\begin{align}
\sum_{|m|\leq T2^{[\alpha K]}}
(\s_1^*)^{-\frac{1}{2}}|\Psi_s(\bar n)|
& \les T 
(n_1^*)^{-\frac{1}{2} - 4s+} (n_3^*)^{2s-1}|n_4 - n_3|^\frac{1}{2}
 \les T (n_1^*)^{-\frac{1}{2}-4s+} (n_3^*)^{2s- \frac{1}{2}} \notag \\
& \les T  
 (n_1^*)^{-\frac{1}{2} - 6s+} (n_3^*)^{- \frac{1}{2}+} 
\bigg(\prod_{i = 1}^4 2^{(s-) k_i}\bigg).
\label{R4M2a}
\end{align}

\noi
If $n_3^* \sim n_1^*$, 
then we have
$\eqref{R4M2a} 
\les T M^{c(s)} (n_1^*)^{4s-}$,
where $c(s)$ is as in \eqref{R4M2}.
Then,  the rest follows as in Case (a).

Now, suppose $n_3^* \ll n_1^*$.
We only consider the case $\s_1 = \s_1^*$.
(A similar argument holds for other cases.)
By Lemma \ref{LEM:L64} (a) and the time localization of size $\sim 2^{k_1^*}$, we have 
\begin{align}
 \sum_{j_2, \dots, j_4}
 \|\cj{f_{2, j_2}}f_{3, j_3}\cj{f_{4, j_4}}\|_{L^2_{x, t}}
\les (n_1^*)^{-\frac{\al}{2}} (n_3^*)^\frac{1}{2}
 \prod_{i = 2}^4  \| u_i\|_{F^\al_{k_i}}.
\label{R4M2b}
\end{align}

\noi
Here,  we used the fact $\al = -4s + \eps \leq 1$
for $s > -\frac 14$.
From \eqref{R4M2a} and \eqref{R4M2b} we obtain 
\begin{align*}
|R_4^M(T)|
& \les
\sum_{|m|\leq T2^{[\alpha K]}}
 \sum_{j_1, \dots, j_4}
(\s_1^*)^{-\frac{1}{2}}|\Psi_s(\bar n)|
\big\| \F^{-1} (\s_1^\frac{1}{2} \ft f_{1, j_1})\big\|_{L^2_{x, t}} 
\|\cj{f_{2, j_2}}f_{3, j_3}\cj{f_{4, j_4}}\|_{L^2_{x, t}}\\
& \lesssim
T M^{c(s)} \prod_{i = 1}^4 2^{(s-) k_i} \| u_i\|_{F^\al_{k_i}}, 
\end{align*}

\noi
where
\begin{equation*} 
c(s) = \max(-\tfrac{1}{2} -4s +, 0).
\end{equation*}

\medskip
\noi
$\circ$ Subcase (b.ii):
$|n_4 - n_3| \sim n_3^*$.
\\
\indent
In this case, we have  $|\Phi(\bar{n})| \gtrsim n_1^*n_3^*$.
Then, with \eqref{alpha}, we have 
\begin{align*}
\sum_{|m|\leq T2^{[\alpha K]}}
(\s_1^*)^{-\frac{1}{2}}|\Psi_s(\bar n)|
& \les T 
(n_1^*)^{-\frac{1}{2} - 4s+} (n_3^*)^{-\frac{1}{2}}
(n_4^*)^{2s}\notag \\
& \les T  
 (n_1^*)^{-\frac{1}{2} - 6s+} (n_3^*)^{- \frac{1}{2}-s} (n_4^*)^s
\bigg(\prod_{i = 1}^4 2^{(s-) k_i}\bigg).
\end{align*}

\noi
Note that we have $n_3^* \ll n_1^*$ in this case.
Then, 
the rest follows from Lemma \ref{LEM:L64} (a) as in Subcase (b.i), 
where $c(s)$ is given by 
\begin{equation} 
c(s) = \max(-\tfrac{1}{2} -5s +, 0).
\label{R4M4}
\end{equation}

\medskip
\noi
$\bullet$  {\bf Case (c):} $|n_4 - n_1|,   |n_4 - n_3|\sim n_1^*  $.
\\
\indent
In this case, we have $|\Phi(\bar{n})| \sim (n_1^*)^2$.
Then, with \eqref{alpha}, we have
\begin{align*}
\sum_{|m|\leq T2^{[\alpha K]}}
(\s_1^*)^{-\frac{1}{2}}|\Psi_s(\bar n)|
& \les T 
(n_1^*)^{-1 - 4s+} (n_4^*)^{2s}
 \les T  
 (n_1^*)^{-1 - 7s+} (n_4^*)^s
\bigg(\prod_{i = 1}^4 2^{(s-) k_i}\bigg).
\end{align*}

\noi
Then, the rest follows  from Lemma \ref{LEM:L4L6} as in Case (a)
with  $c(s)$ given by 
\begin{equation*} 
c(s) = \max(-1-7s+, 0).
\end{equation*}

\medskip
\noi
 {\bf Part 2:}
Next, we consider the terms
with $m \in \mathcal{B}$.
In this case, we need to handle the sharp cutoff $\ind_{[0, T]}$.
The modification is systematic and
thus we only discuss Case (a)
when $\s_1 = \s_1^*$.
The main point is that 
we do not need to sum over $m$
since there are only $O(1)$ many values of $m$ in $\mathcal{B}$.
In particular,  we gain  $(n_1^*)^{-\al}$ as compared to Cases (a), (b), and (c).

From \eqref{DMT1}, we have 
\begin{align*}
(\s_1^*)^{-\frac{1}{2}+2\theta+} |\Psi(\bar n)| \les (n_1^*)^{-1-2s}
\leq  M^{c(s)} (n_1^*)^{4s-}
\end{align*}

\noi
for $\theta <  -s$,
where $c(s)$ is as in \eqref{R4M2}.  Compare this with \eqref{R4M1a}.
Proceeding as in Case (a), we obtain 
\begin{align*}
|R_4^M(T)|
& \les
 M^{c(s)} (n_1^*)^{4s-}
\Big( \sum_{j_1} 2^{-\theta  j_1}\Big)
\sup_{j_1} 2^{(\frac{1}{2} - \theta -)j_1}
\big\|\F(\ind_{[0, T]} f_{1, j_1})\big\|_{\l^2_n L^2_\tau}
 \sum_{j_2, \dots, j_4}\prod_{i = 2}^4 \|f_{i, j_i}\|_{L^6_{x, t}}\\
& \lesssim
T^\theta
 M^{c(s)} \prod_{i = 1}^4 2^{(s-) k_i} \| u_i\|_{F^\al_{k_i}},
\end{align*}

\noi
where 
we used Lemmata \ref{LEM:sup} and \ref{LEM:timedecay}
in the last step.
 In all the other cases,
 we can save a small power of $\s_1^*$
 thanks to the gain of $(n_1^*)^{-\al}$,
Since
 the modifications are similar, we omit details.
This completes the proof of Proposition \ref{LEM:R4M}.
\end{proof}

Next, we estimate the correction term $\Ld_4^M$.

\begin{proposition} \label{LEM:LdM}
Let $s \in (-\frac{1}{2},  0)$ and  $\al > 0$. 
Then, there exists $d(s) > 0$ such that
\begin{equation} \label{LdM}
|\Ld_4^M(t; u)| \lesssim M^{-d(s)} \|u\|_{F^{s, \al}(T)}^4
\end{equation}

\noi
for $t\in [0, T]\subset [0, 1]$.
\end{proposition}

\begin{proof}
After applying a dyadic decomposition on
the spatial frequencies,  assume that
$ |n_1|\sim n_1^*  > M.$
In view of Lemma \ref{LEM:embed1}
and the fact that $\max(|n_j|) > M$,
it suffices to prove 
\begin{equation}
|\Ld_4^M(t; u)| \lesssim \|u(t)\|_{H^{s-}}^4
\label{LdM0}
\end{equation}

\noi
for $s \in (-\frac{1}{2}, 0)$.

\medskip
\noi
$\bullet$  {\bf Case (a):}  $|n_4 - n_1| , |n_4 - n_3| \ll n_1^*$.
\\
\indent
From \eqref{DMT1}, we have
\begin{equation}
\bigg|\frac{\Psi_s(\bar{n})}{\Phi(\bar{n})}\bigg|
\lesssim (n_1^*)^{2s-2} \les \prod_{i = 1}^4 \jb{n_i}^{\frac{s}{2} - \frac{1}{2}}.
\label{LdM1}
\end{equation}

\noi
Then, by H\"older (in $x$) and Sobolev inequalities,  we have
\begin{align*}
|\Ld_4^M(t;u)|
& \les  \| u(t)\|_{W^{\frac{s}{2} - \frac{1}{2}, 4}}^4
\lesssim \|u(t)\|_{H^{\frac{s}{2}-\frac{1}{4}}}^4.
\end{align*}

\noi Hence, \eqref{LdM0} follows once we note that
 $\frac{s}{2} - \frac{1}{4}  < s$ for $s > -\frac{1}{2}$.

\medskip
\noi
$\bullet$  {\bf Case (b):} $|n_4 - n_1| \sim n_1^* \gg |n_4 - n_3| $.

\smallskip

\noi
$\circ$ Subcase (b.i):
We first  consider the case $|n_4 - n_3| \ll n_3^*$.
From \eqref{MVT2a},  we have
\begin{equation}
\bigg|\frac{\Psi_s(\bar{n})}{\Phi(\bar{n})}\bigg|
\lesssim (n_1^*)^{-1}(n_3^*)^{2s-1}
\lesssim (n_1^*)^{-1}(n_3^*)^{s-\frac{1}{2}}
(n_4^*)^{s-\frac{1}{2}}.
\label{LdM2}
\end{equation}

\noi
Then, by H\"older (in $x$) and Sobolev inequalities 
(note that $\frac{4}{1-2s} \geq 2$ for $s \geq - \frac 12$),  we have
\begin{align*}
|\Ld_4(t;u)|
& \lesssim
\| u \|^2_{W^{-\frac 12, \frac{4}{1-2s}}}
\| u \|^2_{W^{s-\frac 12, \frac{4}{1+2s}}}
 \lesssim \|u(t)\|_{H^{\frac{s}{2}-\frac{1}{4}}}^4.
\end{align*}

\noi Hence, \eqref{LdM0} once we note that
 $\frac{s}{2} - \frac{1}{4}  < s$ for $s > -\frac{1}{2}$.

\medskip

\noi
$\circ$ Subcase (b.ii):
When $|n_4 - n_3| \sim n_3^*$,
 we have $|\Phi(\bar{n})| \gtrsim n_1^* n_3^*$.
Thus, we have
\begin{equation}
\bigg|\frac{\Psi_s(\bar{n})}{\Phi(\bar{n})}\bigg|
\lesssim (n_1^*)^{-1}(n_3^*)^{-1}(n_4^*)^{2s}
\lesssim (n_1^*)^{-1}(n_3^*)^{s-\frac{1}{2}}
(n_4^*)^{s-\frac{1}{2}}.
\label{LdM2a}
\end{equation}

\noi
The rest follows as in Subcase (b.i).

\medskip
\noi
$\bullet$  {\bf Case (c):} $|n_4 - n_1|,   |n_4 - n_3|\sim n_1^*  $.
\\
\indent
In this case, we have $|\Phi(\bar{n})| \sim (n_1^*)^2$
and thus
\begin{equation}
\bigg|\frac{\Psi_s(\bar{n})}{\Phi(\bar{n})}\bigg|
\lesssim (n_1^*)^{-2}(n_4^*)^{2s}.
\label{LdM3}
\end{equation}

\noi
Suppose $n_4^* = |n_4| $.
Then, by H\"older (in $x$) and Sobolev inequalities,  
we have
\begin{align*}
|\Ld_4(t;u)|
& \les \| u \|_{W^{-\frac 23, \frac 6{1-2s}+}}^3\| u \|_{W^{2s, \frac{2}{1+2s}-}}
\les \|u \|^3_{H^{ \frac s3 -\frac 13+}} \| u\|_{H^{s-}}.
\end{align*}

\noi 
Hence, \eqref{LdM0} once we note that
 $\frac{s}{3} - \frac{1}{3}  < s$ for $s > -\frac{1}{2}$.
\end{proof}

Finally, we estimate $R_6^M$ in \eqref{EM4}.

\begin{proposition}\label{LEM:R6M}
Let $s \in (-\frac{1}{8},  0)$ and  $\al  = -4s + \eps$ as in
\eqref{alpha}. 
Then, there exists $\theta > 0$ such that 
\begin{equation} \label{R6M1}
|R_6^M(T; u)| \lesssim T^\theta \|u\|_{F^{s, \al}(T)}^6
\end{equation}

\noi for $T\in (0, 1]$.

\end{proposition}

\begin{proof}

First, we estimate the contribution from $\II(T)$. Since $n_1 - n_2
+ n_3 - n_4 = 0$, we have $|n_j| \ges |n_1|$ for some $j \in \{2, 3,
4\}$, say $|n_2|\ges |n_1|$. 
Then,  \eqref{R6M1} follows (with $\theta = 1$)
 from Proposition \ref{LEM:LdM} and Lemma \ref{LEM:embed1}
by noting the following.
In Cases (a) and (b) of Proposition \ref{LEM:LdM}, we have
\begin{equation*}
\jb{n_1}^{\frac{s}{2} - \frac{1}{4}+}\jb{n_2}^{\frac{s}{2} -
\frac{1}{4}+} \les \jb{n_1}^{3s}\jb{n_2}^{s}
\end{equation*}

\noi as long as  $s > - \frac{1}{6}$.
In Case (c), by writing 
$(n_1^*)^{-2} = (n_1^*)^{-2 - 2s}(n_1^*)^{2s}$, 
it suffices to note that 
$- \frac{s}{3} - \frac 13 + < s$
for $ s > -\frac{1}{4}$.

Next, we consider  the contribution from $\I(T)$.  Apply a dyadic
decomposition on the spatial frequencies $|n_i|\sim 2^{k_i}$, $ i = 1, \dots,  7$. 
For convenience, let $n_1^*\geq n_2^*\geq n_3^*\geq
n_4^*$, $m_1^*\geq m_2^*\geq m_3^*$ be the 
decreasing rearrangements of
$\{|n_1|, |n_2|, |n_3|, |n_4|\}$, and $\{ |n_5|, |n_6|, |n_7|\}$, 
respectively. 
Then, 
by assuming that each factor $u_i$ has its Fourier support on $I_{k_i}\times \R$,
it suffices to prove
\begin{equation*}
|\I(T)|\les T^\theta \prod_{i=2}^7 2^{(s-)k_i}\|\P_{k_i}u\|_{F_{k_i}^\alpha(T)}
\end{equation*}

\noi
for $s>-\frac{1}{8}$.
Here, a slight extra decay is needed to sum over dyadic blocks.
Let  $\wt u_i$ be an extension of $u_i$ such that 
$\|\wt u_i\|_{F_{k_i}^\alpha}\leq 2 \|\P_{k_i}u\|_{F_{k_i}^\alpha(T)}$.
For
notational simplicity, we denote $\wt u_i$ by $u_i$,
$i = 2, \dots,  7$, in the following. 
Letting $\g:\R \to [0, 1]$ be a smooth cutoff function supported on $[-1, 1]$ such that
$\sum_m\g^6(t - m) \equiv 1$ for all  $t \in \R$, 
we have
\begin{align*}
\I(T)=\int_{\R} \ind_{[0,T]}(t)
& \sum_{|m|\leq T2^{[\alpha K]}}\g^6(2^{[\alpha K]}t - m)\sum_{\substack{n_1 - n_2 + n_3-n_4 = 0 \\
n_2\ne n_1, n_3\\\max (|n_j| )> M} }\frac{\Psi_s(\bar{n})}{\Phi(\bar{n})}\ind_{I_{k_1}}(n_1)
\notag \\
& 
\times
\sum_{\substack{n_1 = n_5 - n_6 + n_7\\ n_6\ne n_5, n_7} }
\ft{u}_5({n_5}) \cj{\ft{u}_6}({n_6})\ft{u}_7({n_7})
 \cj{\ft{u}_2}({n_2})\ft{u}_3({n_3})\cj{\ft{u}_4}({n_4}) (t) dt .
\end{align*}

\noi
Here, we choose $K$ to be a positive integer such that $2^K\sim \max(m_1^*,n_1^*)$
unless otherwise stated.
In the following, we only consider the terms
with 
$\ind_{[0,T]}(t)\g(2^{[\alpha K]}t - m) =  \g(2^{[\alpha K]}t - m)$
and thus 
we drop the sharp cut-off $\ind_{[0,T]}(t)$ on the left-hand side.
Note that there are $O(1)$ many values of  $m$ such that
$\ind_{[0,T]}(t)\g(2^{[\alpha K]}t - m)\ne \g(2^{[\alpha K]}t - m)$
and they are easier to handle.
 See the proof of Proposition \ref{LEM:R4M}.

With  $ f_i=
\g(2^{[\al  K]}t - m)u_i,$\footnote{Strictly speaking, 
we need to use $f_{i, m}$ instead of $f_i$.
However, 
all the estimate below are uniform in $m$.
Thus, for notational simplicity, we drop the subscript $m$.} 
$2\leq i\leq 7$, 
let 
$ f_{i, j_i} = \F^{-1}\big[\eta_{j_i}(\tau - n^2) \ft f_i\big]$.
Then, 
 it suffices to prove 
\begin{align}
 \bigg|\int_{\R}  
 \sum_{j_2, \dots, j_7}
 \sum_{|m|\leq T2^{[\alpha K]}}
 & \sum_{\substack{n_1 - n_2 + n_3-n_4 = 0 \\
n_2\ne n_1, n_3\\\max |n_j| > M} }
 \frac{\Psi_s(\bar{n})}{\Phi(\bar{n})}\ind_{I_{k_1}}(n_1)
\sum_{\substack{n_1 = n_5 - n_6 + n_7\\ n_6\ne n_5, n_7} }
\ft{f}_{5, j_5}({n_5}) \cj{\ft{f}_{6, j_6}}({n_6})\ft{f}_{7, j_7}({n_7})\notag \\
&  \times \cj{\ft{f}_{2, j_2}}({n_2})\ft{f}_{3, j_3}({n_3})\cj{\ft{f}_{4, j_4}}({n_4}) (t) dt \bigg|
 \les T^\theta \prod_{i=2}^7 2^{(s-)k_i}\|u_i\|_{F_{k_i}^\alpha}
\label{R63}
\end{align}

\noi
for $ s > -\frac{1}{8}$.
In view of \eqref{Xk2}, we assume that 
$j_i \geq \al K$,  $i = 2, \dots, 7$.
In the following,  we prove \eqref{R63}
for fixed $j_i$, $i = 2, \dots, 7$.
Since the summations over $j_i$ do not play any significant role in the following, 
we drop the summations over $j_i$
and 
we denote $f_{i, j_i}$  by $f_i$
for simplicity.

\medskip
Now, we proceed to prove \eqref{R63}. 
Gathering  \eqref{LdM1}, \eqref{LdM2}, \eqref{LdM2a}, and \eqref{LdM3}
in the proof of Proposition \ref{LEM:LdM}, 
 we have the bound
\begin{equation}
\Theta(\bar n): = \bigg|\frac{\Psi_s(\bar{n})}{\Phi(\bar{n})}\bigg|\les \frac{(n_4^*)^{2s}}{{n_3^*}{n_1^*}},
\label{R64}
\end{equation}

\noi
which suffices for many cases in the following.
Note that, however, each of  \eqref{LdM1}, \eqref{LdM2}, 
and \eqref{LdM3} is better than \eqref{R64}.
Let $\wt \Phi(\bar n)=n_1^2-n_5^2+n_6^2-n_7^2
= 2(n_1 - n_5) (n_1 - n_7)$
as in \eqref{Phim}. 
We prove
\eqref{R63} by performing case-by-case analysis.
By symmetry, we assume $|n_5|\geq |n_7|$ in the following.

\medskip

\noi
$\bullet $ {\bf Case (a):} $m_1^* \les n_1^*$.
\\
\indent
We first consider the case $n_1^* \sim n_3^*$.
From \eqref{R64}, we have 
\begin{align}
\sum_{|m| \leq T2^{[\al K]}}\Theta(\bar n)
\les T(n_1^*)^{-2 - 4s+} (n_4^*)^{2s}
\leq T (n_1^*)^{6s - }
\label{R4a1}
\end{align}

\noi
for $s > -\frac{1}{5}$.
Then, 
\eqref{R63} follows from the  $L^6$-Strichartz estimate \eqref{Stri3}
with \eqref{R4a1}:
\begin{align*}
|\I(T)|
& \les 
\sum_{|m| \leq T2^{[\al K]}}\Theta(\bar n)
\|f_{5} \cj{f_6}f_7\|_{L^2_{x, t}}\|\cj{f_2} f_3 \cj{f_4}\|_{L^2_{x, t}}\\
& \les 
\sum_{|m| \leq T2^{[\al K]}}\Theta(\bar n)
\prod_{i = 2}^7 \|f_i\|_{L^6_{x, t}}\\
& \les 
T \prod_{i=2}^7
2^{(s-)k_i}
\|u_i\|_{F_{k_i}^\alpha}
\end{align*}

\noi
for $s > -\frac{1}{5}$.
In the following, we assume   $n_1^* \gg n_3^*$.

\medskip

\noi
$\circ$
Subcase (a.i): $|n_1|\gg n_3^*$. 
\\
\indent
In this case, we have $|n_1| \sim n_1^*$.
In particular, we have $\min (|n_2|, |n_3|, |n_4|) = n_4^*$.
Then, 
from \eqref{R64}, we have
\begin{align}
\sum_{|m| \leq T2^{[\al K]}}\Theta(\bar n)
\cdot (n_1^*)^{-\frac{\al}{2}} (n_3^*)^\frac{1}{2}
&  \les T(n_1^*)^{-1-2s+} (n_3^*)^{-\frac{1}{2}}
(n_4^*)^{2s} \notag\\
&   \les T (n_1^*)^{4s -  }
(n_3^*)^{s }(n_4^*)^{s} 
\les T \prod_{i = 2}^7 2^{(s-)k_i}, 
\label{R4a2}
\end{align}

\noi
where the penultimate inequality holds for $s > -\frac{1}{6}$.
Noting that $\min(|n_2|, |n_4|)\les n_3^* \ll |n_1|$, 
it follows from 
the $L^6$-Strichartz estimate \eqref{Stri3}, 
Lemma \ref{LEM:L64} (a), and \eqref{R4a2} that 
\begin{align*}
|\I(T)|
& \les 
\sum_{|m| \leq T2^{[\al K]}}\Theta(\bar n)
\|f_{5} \cj{f_6}f_7\|_{L^2_{x, t}}\|\cj{f_2} f_3 \cj{f_4}\|_{L^2_{x, t}}\\
& \les 
\sum_{|m| \leq T2^{[\al K]}}\Theta(\bar n)
\cdot (n_1^*)^{-\frac{\al}{2}} (n_3^*)^\frac{1}{2}
\bigg(\prod_{i = 5}^7 \|f_i\|_{L^6_{x, t}}\bigg)
\bigg( \prod_{i=2}^4
\|u_i\|_{F_{k_i}^\alpha}\bigg)\\
& \les 
T \prod_{i=2}^7
2^{(s-)k_i}
\|u_i\|_{F_{k_i}^\alpha}
\end{align*}

\noi
for $s > -\frac{1}{6}$.

\medskip

\noi
$\circ$
Subcase (a.ii): $|n_1|\les n_3^* \ll n_1^*$. 

\smallskip

\underline{Subsubcase (a.ii.1):}  $m_1^* \sim m_2^* \gg |n_1|$.
\quad 
First suppose that  $m_1^* \ges (n_1^*)^\al$.

\begin{enumerate}
\item[(i)] Suppose $|n_1|\sim  n_4^*$.
Then, 
from \eqref{R64},
we have 
\begin{align*}
\sum_{|m| \leq T2^{[\al K]}}\Theta(\bar n)
\cdot (n_1^*)^{-\al}n_4^*
& \les T(n_1^*)^{-1} (n_3^*)^{-1} (n_4^*)^{1+ 2s}
   \les T (n_1^*)^{5s -  }(n_3^*)^{s }
\les T 
\prod_{i = 2}^7 2^{(s-)k_i}
\end{align*}

\noi
for $s > -\frac{1}{5}$.

\smallskip

\item[(ii)] Suppose $|n_1| \sim n_3^* \gg n_4^*$.
Then, 
from \eqref{R64},
we have 
\begin{align*}
\sum_{|m| \leq T2^{[\al K]}}\Theta(\bar n)
\cdot (n_1^*)^{-\al}n_3^*
& \les T(n_1^*)^{-1} (n_4^*)^{2s}
   \les T (n_1^*)^{5s -  }(n_4^*)^{s }
\les T 
\prod_{i = 2}^7 2^{(s-)k_i}
\end{align*}

\noi
for $s > -\frac{1}{5}$.

\end{enumerate}

\noi
Noting that $\max(|n_2|, |n_4|) \sim n_1^* \gg n_3^* \ges |n_1|$, 
we apply Lemma \ref{LEM:L64} (a) on 
 $f_5 \cj{f_6} f_7$
and 
$\cj{f_2}  f_3 \cj{f_4}$.
This yields
\eqref{R63} for $s > -\frac{1}{5}$.

Next, suppose that  $m_1^* \ll (n_1^*)^\al$. 
In this case, 
by applying  Lemma \ref{LEM:L64} (a) on 
 $f_5 \cj{f_6} f_7$, 
we only gain the factor of $(m_1^*)^{-\frac{1}{2}}$
instead of $(n_1^*)^{-\frac{\al}{2}}$.
Nonetheless, 
\eqref{R63} for $s> -\frac 16$ follows from  applying Lemma \ref{LEM:L64} (a) on 
 $f_5 \cj{f_6} f_7$
and 
$\cj{f_2}  f_3 \cj{f_4}$, 
since 
\begin{align*}
\sum_{|m| \leq T2^{[\al K]}}\Theta(\bar n)
\cdot (n_1^*)^{-\frac{\al}{2}}(m_1^*)^{-\frac 12}n_3^*
& \les T(n_1^*)^{-1-2s+} (n_4^*)^{2s}(m_1^*)^{-\frac 12}
\\
 &  \les T (n_1^*)^{3s -  }(m_1^*)^{3s }
\les T 
\prod_{i = 2}^7 2^{(s-)k_i}
\end{align*}

\noi
for $s > -\frac 16$.

\medskip

\underline{Subsubcase (a.ii.2):}  $m_1^*\sim m_2^* \les |n_1| \les n_3^*$.
\quad From \eqref{R64}, we have
\[\sum_{|m| \leq T2^{[\al K]}}\Theta(\bar n)
\les T(n_1^*)^{-1 - 4s+} (n_3^*)^{-1} (n_4^*)^{2s}
\leq T (n_1^*)^{2s -  }(n_3^*)^{4s }
\les T 
\prod_{i = 2}^7 2^{(s-)k_i}
\]

\noi
for $s > -\frac{1}{6}$.
Then, 
\eqref{R63} follows from 
the $L^6$-Strichartz estimate \eqref{Stri3}.

\medskip

\noi
$\bullet $ {\bf Case (b):} 
$m_1^* \gg n_1^*$, \ 
$|\wt \Phi(\bar n)|\les |\Phi(\bar n)| $.
\\
\indent
In the following, 
we choose $K$ such that  $2^K  \sim m_1^*$,
unless otherwise stated.

\smallskip

\noi
$\circ$
Subcase (b.i): $|n_1|\les n_3^*$. 

\smallskip

\underline{Subsubcase (b.i.1):}
$n_1^*\sim n_3^*$.
\quad
If $m_3^* \gg |n_1|$, 
then we have
\begin{align}
m_1^* m_3^* \les |\wt \Phi(\bar n)|
\les | \Phi(\bar n)| \les (n_1^*)^2.
\label{R4b1}
\end{align}

\noi
This in particular implies that 
$|n_1| = n_4^* \ll n_1^*$.
Otherwise, i.e.~if $|n_1| \sim n_3^* \sim n_1^*$, then
it follows from \eqref{R4b1} that 
\begin{align*}
m_1^* n_1^* \ll m_1^* m_3^*  \les (n_1^*)^2,
\end{align*}

\noi
which is a contradiction to the assumption $m_1^* \gg n_1^*$.

From \eqref{R4b1}, 
we have 
$m_1^*m_2^*m_3^* 
\les (n_1^*)^4$
and  $m_1^* \les (n_1^*)^2$.
Then, from \eqref{R64}  with $n_1^*\sim n_3^*$,  we have
\begin{align*}
\sum_{|m| \leq T2^{[\al K]}}\Theta(\bar n)
\cdot (m_1^*)^{-\frac{\al}{2}} (n_4^*)^\frac{1}{2}
& \les T(n_1^*)^{-\frac{3}{2}+2s} (m_1^*)^{-2s+}
\les T (n_1^*)^{3s}(m_1^*m_2^*m_3^*)^s
(n_1^*)^{-\frac{3}{2}-9s+}\\
& \les T 
\prod_{i = 2}^7 2^{(s-)k_i}
\end{align*}

\noi
for $s > -\frac{1}{6}$.
Noting that $\max(|n_5|, |n_7|)\sim m_1^*$, 
we apply  Lemma \ref{LEM:L64} (a) on $f_5\cj{f_6}f_7$.
Then, with the $L^6$-Strichartz estimate \eqref{Stri3}, we have
\begin{align*}
|\I(T)|
& \les 
\sum_{|m| \leq T2^{[\al K]}}\Theta(\bar n)
\|f_5\cj{f_6}f_7\|_{L^2_{x, t}}\|\cj{f_2} f_3 \cj{f_4}\|_{L^2_{x, t}}\\
& \les 
\sum_{|m| \leq T2^{[\al K]}}\Theta(\bar n)
\cdot (m_1^*)^{-\frac{\al}{2}} (n_4^*)^\frac{1}{2}
\bigg( \prod_{i=5}^7
\|u_i\|_{F_{k_i}^\alpha}\bigg)
\bigg(\prod_{i = 2}^4 \|f_i\|_{L^6_{x, t}}\bigg)\\
& \les 
T \prod_{i=2}^7
2^{(s-)k_i}
\|u_i\|_{F_{k_i}^\alpha}
\end{align*}

\noi
for $s > -\frac{1}{6}$.

Next, suppose that $m_3^* \les |n_1|$.
Then,  
we have
$m_1^*  \les |\wt \Phi(\bar n)|
\les | \Phi(\bar n)| \les (n_1^*)^2.$
In particular, 
we have $m_1^*m_2^*m_3^* \les (n_1^*)^4 |n_1|$.
In this case,
we choose $K$ such that  $2^K  \sim n_1^*$,
although $n_1^* \ll m_1^*$.
We  estimate 
$\|f_5\cj{f_6}f_7\|_{L^2_{x, t}}$
by further dividing 
the interval of length $\sim 2^{-[\al K]}\sim (n_1^*)^{-\al}$ 
into $O( (m_1^*)^{\al}(n_1^*)^{-\al})$-many subintervals 
of length $ \sim (m_1^*)^{-\al}$.
By the almost orthogonality of the contributions from these subintervals under the $L^2_t$-norm, 
this adds an extra factor of 
 $(m_1^*)^\frac{\al}{2}(n_1^*)^{-\frac{\al}{2}}$.
\smallskip

\begin{itemize}
\item[(i)] Suppose that $|n_1|  \sim  n_4^*$. Then,
from \eqref{R64},  we have
\begin{align*}
 (m_1^*)^\frac{\al}{2}(n_1^*)^{-\frac{\al}{2}}
& \sum_{|m| \leq T2^{[\al K]}}\Theta(\bar n)
\cdot (m_1^*)^{-\frac{\al}{2}} (n_4^*)^\frac{1}{2}
 \les T
  (n_1^*)^{-2-2s+ }(n_4^*)^{\frac{1}{2}+2s}\\
&  \les T(n_1^*)^{-\frac{3}{2} + }
  \les T(n_1^*)^{-\frac{3}{2} - 8s + }
\prod_{i = 2}^7 2^{(s-)k_i}
 \les T
\prod_{i = 2}^7 2^{(s-)k_i}
\end{align*}

\noi
for $s > -\frac 3{16}$.

\smallskip

\item[(ii)] Suppose that $|n_1| \gg n_4^*$.
Then, 
from \eqref{R64},  we have
\begin{align*}
 (m_1^*)^\frac{\al}{2}(n_1^*)^{-\frac{\al}{2}}
& \sum_{|m| \leq T2^{[\al K]}}\Theta(\bar n)
\cdot (m_1^*)^{-\frac{\al}{2}} (n_3^*)^\frac{1}{2}
 \les T(n_1^*)^{-\frac{3}{2}-2s+}(n_4^*)^{2s}\\
& \les T (n_1^*)^{-\frac 32 - 9 s+} (n_4^*)^{s}
\prod_{i = 2}^7 2^{(s-)k_i}
\les T 
\prod_{i = 2}^7 2^{(s-)k_i}
\end{align*}

\noi
for $ s > -\frac{1}{6}$.

\end{itemize}

\noi
In both cases, we obtain \eqref{R63} 
by applying 
 Lemma \ref{LEM:L64} (a) on $f_5\cj{f_6}f_7$
and  the $L^6$-Strichartz estimate \eqref{Stri3}
on $\cj{f_2} f_3 \cj{f_4}$.

\medskip

\underline{Subsubcase (b.i.2):}
Next, we consider the case $n_1^*\gg n_3^*$.
Note that we have $m_1^*\les (n_1^*)^2$.
Moreover, we have either $m_3^* \sim |n_1|$
or $m_1^*m_3^* \les  (n_1^*)^2$.
In either case, we have
\[m_1^* m_2^* m_3^* \les (n_1^*)^4 |n_1|.\]

In this case,
we choose $K$ such that  $2^K  \sim n_1^*$,
although $n_1^* \ll m_1^*$.
As in Subsubcase (b.i.1), we lose
 an extra factor of 
 $(m_1^*)^\frac{\al}{2}(n_1^*)^{-\frac{\al}{2}}$
 by 
estimating the contribution of 
$\|f_5\cj{f_6}f_7\|_{L^2_{x, t}}$
over  subintervals 
of length $ \sim (m_1^*)^{-\al}$.

\smallskip

\begin{itemize}
\item[(i)] $|n_4 - n_1| \sim n_1^* \gg |n_4 - n_3| = |n_1 - n_2|$.
\quad
Since $|n_1| \les n_3^*\ll n_1^*$, it follows that 
$|n_2| \les n_3^*\ll n_1^*$.
From \eqref{R64}, we have
\begin{align*}
 (m_1^*)^\frac{\al}{2}(n_1^*)^{-\frac{\al}{2}}
& \sum_{|m| \leq T2^{[\al K]}}\Theta(\bar n)
 \cdot (m_1^*)^{-\frac{\al}{2}}(n_1^*)^{-\frac{\al}{2}} (n_3^*)^\frac{1}{2}(n_4^*)^\frac{1}{2}
\les T(n_1^*)^{-1-2s+ }(n_3^*)^{-\frac{1}{2}}
(n_4^*)^{\frac{1}{2}+2s}\\
& \les T(n_1^*)^{-1 - 6s+} 
(n_3^*)^{-\frac12-s} (n_4^*)^{\frac12+s}
\prod_{i = 2}^7 2^{(s-)k_i}
\les T
\prod_{i = 2}^7 2^{(s-)k_i} 
\end{align*}

\noi
for $ s > -\frac{1}{6}$.
Then,  \eqref{R63} 
follows from 
applying 
 Lemma \ref{LEM:L64} (a) on $f_5\cj{f_6}f_7$
and  $\cj{f_2} f_3 \cj{f_4}$.

\smallskip
\item[(ii)] $|n_4 - n_1| ,  |n_4 - n_3| \sim n_1^*$.
\quad
From \eqref{LdM3}, we have
\begin{align*}
 (m_1^*)^\frac{\al}{2}(n_1^*)^{-\frac{\al}{2}}
& \sum_{|m| \leq T2^{[\al K]}}\Theta(\bar n)
  \cdot (m_1^*)^{-\frac{\al}{2}}(n_1^*)^{-\frac{\al}{2}} 
    (n_3^*)^\frac{1}{2}(n_4^*)^\frac{1}{2}
 \les T(n_1^*)^{-2 -2s+} 
(n_3^*)^\frac{1}{2}(n_4^*)^{\frac{1}{2}+2s}\\
& \les T(n_1^*)^{-2 - 6s} 
(n_3^*)^{\frac{1}{2}-s}(n_4^*)^{\frac{1}{2}+s}
\prod_{i = 2}^7 2^{(s-)k_i}
\les T
\prod_{i = 2}^7 2^{(s-)k_i} 
\end{align*}

\noi
for $ s > -\frac{1}{6}$.
Then,  \eqref{R63} 
follows from 
applying 
 Lemma \ref{LEM:L64} (a) on $f_5\cj{f_6}f_7$
and  
 Lemma \ref{LEM:L64} (b) on 
$\cj{f_2} f_3 \cj{f_4}$.

\end{itemize}

\medskip

\noi
 $\circ$   Subcase (b.ii): $|n_1|\sim n_1^* \gg n_3^*$.

Note that  we have $|n_1 - n_7|  \ll |n_1|$.
Otherwise, i.e.~if $|n_1 - n_7|  \ges |n_1|$,
then we have $|\wt \Phi(\bar n)|\ges m_1^* n_1^* \gg (n_1^*)^2 \ges |\Phi(\bar n)|$,
which is a contradiction.
Hence, we have $|n_1 - n_7|  \ll |n_1|$.
In particular, $m_1^* \sim m_2^* \gg m_3^* \sim n_1^*$.

\smallskip

\underline{Subsubcase (b.ii.1):}
$|n_4 - n_1| \sim n_1^* \gg |n_4 - n_3| \sim n_3^*$.
\quad 
From \eqref{Psi} and \eqref{R64},  we have 
\begin{align*}
\Theta(\bar n) \les
|\wt \Phi(\bar n)|^{-\g} (n_1^*n_3^*)^{-(1-\g)}(n_4^*)^{2s}
\end{align*}

\noi
for any $\g \in [0, 1]$.
With $|\wt \Phi(\bar n)| \sim m_1^* |n_1 - n_7| \les |\Phi(\bar n)|$, we have 
\begin{align}
\Theta(\bar n) \cdot (m_1^*)^{\frac{\al}{4}} (n_4^*)^\frac{1}{2}
& \les
|\wt\Phi(\bar n)|^{-\frac{3}{4}-s} 
(m_1^*)^{-s+}
(n_1^*n_3^*)^{-\frac{1}{4}+s}(n_4^*)^{\frac{1}{2} + 2s}\notag \\
& \leq  |n_5 - n_6|^{-\frac{3}{4}-s} 
(m_1^*)^{-\frac{3}{4}-2s+} (n_1^*)^{2s}(n_3^*)^{s}(n_4^*)^{s}\notag \\
& \leq  |n_5 - n_6|^{-\frac{3}{4}-s} 
(m_1^*)^{2s-}(n_1^*)^{2s}(n_3^*)^{s}(n_4^*)^{s}\notag\\
&  \les  |n_5 - n_6|^{-\frac{3}{4}-s} \prod_{i = 2}^7 2^{(s-)k_i}
\label{R6a1}
\end{align}

\noi
for $s > -\frac{3}{16}$.

Fix $n_6 \in \Z$.
Then, by Cauchy-Schwarz inequality, we have
\begin{align*}
 \sum_{n_5\ne n_6}
\frac{g(n_5) }{|n_5 - n_6|^\be}
\les \|g\|_{\l^2_n} 
\end{align*}

\noi
for $\beta > \frac{1}{2}$.
Combining this with 
the proof, in particular \eqref{L62b},  of Case (a) of Lemma \ref{LEM:L62}
under $\{n_5, n_6, n_7, n_1\} \leftrightarrow \{n_1, n_2, n_3, n_4\}$, 
we obtain
\begin{align}
\bigg\| \sum_{\substack{n_5, n_6\\n_6\ne n_5, n_7}} 
\frac{\ft f_5(n_5, t)  \cj{\ft f_6} (n_6, t) \ft f_7(n_1 - n_5+n_6, t)}
{|n_5 - n_6|^{\be}}\bigg\|_{L^2_t \l^2_{n_1}}
\les (m_1^*)^{-\frac{\al}{2}}
 \prod_{i=5}^7
\|u_i\|_{F_{k_i}^\alpha}.
\label{R6a2}
\end{align}

\noi
Also, by $L^4_{x, t}, L^4_{x, t}, L^\infty_{x, t}$-H\"older inequality, we have  
\begin{align}
\|\cj{f_2} f_3 \cj{f_4}\|_{L^2_{x, t}}
& \les (m_1^*)^{-\frac{\al}{4}} (n_4^*)^\frac{1}{2}
 \prod_{i=2}^7
\|u_i\|_{F_{k_i}^\alpha}.
\label{R6a2a}
\end{align}

\noi
Hence, the desired estimate \eqref{R63} follows
from \eqref{R6a1}, \eqref{R6a2} (with $\be = \frac 34 + s$), and \eqref{R6a2a}.

\medskip

\underline{Subsubcase (b.ii.2):}
$|n_4 - n_1| \sim n_1^* \gg n_3^* \gg |n_4 - n_3|$.
\quad 
From \eqref{MVT1} and \eqref{MVT2},  we have 
\begin{align*}
\Theta(\bar n) \les
|\wt \Phi(\bar n)|^{-\g} (n_1^*)^{-(1-\g)}(n_3^*)^{2s-1}|n_4 - n_3|^\g
\ll
|\wt \Phi(\bar n)|^{-\g} (n_1^*)^{-(1-\g)}(n_3^*)^{2s-1+\g}
\end{align*}

\noi
for any $\g \in [0, 1]$.
Then, with $\g = \frac 9{14}$,  we have 
\begin{align*}
\Theta(\bar n) \cdot  (n_3^*)^\frac{1}{2}
& \les  |n_5 - n_6|^{-\frac{9}{14}} 
(m_1^*)^{-\frac{9}{14}}  (n_1^*)^{-\frac{5}{14}} (n_3^*)^{\frac{1}{7} + 2s}\notag \\
& \leq  |n_5 - n_6|^{-\frac{9}{14}} 
(m_1^*)^{3s-}(n_1^*)^{s}(n_3^*)^{2s}\notag\\
&  \les  |n_5 - n_6|^{-\frac{9}{14}}  \prod_{i = 2}^7 2^{(s-)k_i}
\end{align*}

\noi
for $s > -\frac{3}{14}$.
Then, 
the rest 
follows from 
\eqref{R6a2} and  
 Lemma \ref{LEM:L64} (a) on 
$\cj{f_2} f_3 \cj{f_4}$.

\medskip

In the following, we 
consider the remaining case: 
\[m_1^* \gg n_1^*
\qquad\text{and}\qquad
 |\wt \Phi(\bar n)| \gg |\Phi(\bar n)|.\]

\noi
With $\s_j = \tau_j - n_j^2$, we have
\begin{align*}
& \s_1 - \s_5 + \s_6 - \s_7 = -\wt \Phi(\bar n),
\qquad 
 \s_1 - \s_2 + \s_3 - \s_4 =  \Phi(\bar n).
\end{align*}

\noi
Hence,  we conclude that 
$\max(|\s_2|,|\s_3|,|\s_4|,|\s_5|,|\s_6|,|\s_7|)\ges |\tilde \Phi(\bar n)|$.

\medskip

\noi
$\bullet $ {\bf Case (c):}
$\max(|\s_2|, |\s_3|, |\s_4|)
\ges |\wt \Phi(\bar n)|$.

\medskip

\noi
$\circ$ 
Subcase (c.i):  $ |n_7| \sim m_3^* \gg |n_1|$.
\\
\indent
We have $|\wt \Phi(\bar n) | \ges m_1^* m_3^*$. 
In this case, we set $2^K \sim n_1^*$, although $n_1^* \ll m_1^*$.
By duality (in $x$) and Cauchy-Schwarz inequality, we have 
\[\|\P_{k_1}(f_5\cj{f_6}f_7)\|_{L^\infty_tL^2_{x}}
\les 
2^\frac{k_1}{2}
(m_3^*)^\frac{1}{2} \prod_{i = 5}^7 \|f_i\|_{L^\infty_tL^2_x}.\]

\noi
Then, by Lemma \ref{LEM:infty}, we have 
\begin{align}
|\I(T)|  & \les 
\sum_{|m|\leq T 2^{[\al K]}}\Theta(\bar n )
2^\frac{k_1}{2}
(m_3^*)^\frac{1}{2} \bigg(\prod_{i = 5}^7 \|f_i\|_{L^\infty_tL^2_x}\bigg)
\|\cj{f_2}f_3\cj{f_4}\|_{L^1_t L^2_x}
\notag \\
& \les 
T (n_1^*)^{-1 + \al}(n_3^*)^{-1}(n_4^*)^{2s}
2^\frac{k_1}{2}
(m_3^*)^\frac{1}{2} 
\bigg(\prod_{i = 5}^7 \|f_i\|_{F^{\al}_{k_i}}\bigg)
\|\cj{f_2}f_3\cj{f_4}\|_{L^1_t L^2_x}.
\label{R6c1}
\end{align}

\noi
By H\"older inequality,  we 
place the factor in $\cj{f_2}f_3\cj{f_4}$ with the highest modulation in $L^2_t$
and others in $L^2_t$ and $L^\infty_t$,
while we place the factors with two low spatial frequencies in $L^\infty_x$
and the other in $L^2_x$.
For example, if $k_2 \geq k_3 \geq k_4$
and $j_2 \geq j_3 \geq j_4$, 
we have 
\begin{align}
\|\cj{f_2} f_3 \cj{f_4}\|_{L^1_t L^2_{x}}
& \leq 
\| f_2\|_{L^2_{x, t}} 
\| f_3\|_{L^2_t L^\infty_{x}} 
\| f_4\|_{L^\infty_{x, t}}\notag \\ 
& \les 
(m_1^*m_3^*)^{-\frac{1}{2}} 
(n_1^*)^{-\frac{\al}{2}} 2^\frac{k_3}{2}2^\frac{k_4}{2}
\prod_{i = 2}^4 \|u_i\|_{F^{\al}_{k_i}}.
\label{R6c2}
\end{align}

\noi
Here, we used the fact that 
$2^{j_2} \ges m_1^*m_3^*$
and $2^{j_3} \ges (n_1^*)^{\al}$.
By collecting the weights from \eqref{R6c1} and \eqref{R6c2}, 
 we have 
\begin{align}
(m_1^*)^{-\frac{1}{2}} 
(n_1^*)^{-1-2s+} 
& (n_3^*)^{-1} (n_4^*)^{2s}
2^\frac{k_1+k_3+k_4}{2}\notag \\
&  \les
(m_1^*)^{-\frac{1}{2}} 
(n_1^*)^{-\frac{1}{2}-2s+} 
(n_3^*)^{2s} 
 \les
(m_1^*)^{3s-}
(n_1^*)^{2s}(n_3)^{s},
\label{R6c2a}
\end{align}
	
\noi
where
the last inequality holds as long as $s > -\frac{1}{8}$.
Hence, \eqref{R63} holds for $s > -\frac{1}{8}$ in this case.
Note that the same proof holds
when $|n_5| \sim |n_7| \sim m_1^*$,
since we have $|\wt \Phi(\bar n)| \sim (m_1^*)^2$ in this case.

\medskip

\noi
$\circ$ 
Subcase (c.ii):  $|n_7| \sim m_3^* \ll |n_1|$.
\\
\indent
In this case, we have $|\wt \Phi(\bar n) | \ges m_1^* |n_1| 
\gg  m_1^* m_3^*$.
Proceeding as in  Subcase (c.i), 
we obtain  \eqref{R63}  for $s > -\frac{1}{8}$.

\medskip

\noi
$\circ$ 
Subcase (c.iii):  $|n_7| \sim |n_1|$.
\\
\indent
In this case, we first estimate the contribution
from $|n_7-n_1|\sim 2^L$
for each $L$
and sum over $L$ in the end.
Note that $L \leq k_1+10$.
When $|n_7-n_1|\sim 2^L$, we have 
\begin{align}
|\wt \Phi(\bar n) | \sim 2^L m_1^*.
\label{R6c3a} 
\end{align}

\noi
Moreover, 
we have $|n_5 - n_6| \sim 2^L$.

As in the proof of Lemma \ref{LEM:L63}, 
 write 
 $I_{k_i} = \bigcup_{\l_i} J_{i\l_i}$, $i = 5, 6$, 
where $|J_{i\l_i}| \sim 2^L$.
Then, if $n_5 \in J_{5\l_5}$ for some $\l_5$, 
there are $O(1)$ many possible values for $\l_6 = \l_6(\l_5)$
such that  $n_6\in J_{6\l_6}$.
Then, by writing 
\[\sum_{n_5} \sum_{n_6} = \sum_{\l_5}\sum_{\l_6 = \l_6(\l_5)} \sum_{n_5 \in J_{5\l_5} } 
\sum_{n_6 \in J_{6\l_6}},\]

\noi
we only lose $2^L$ by applying Cauchy-Schwarz inequality
in $n_5$ and $n_6$.
Then, by Young's inequality, we have 
\begin{align}
\|f_5 \cj{f_6}f_7\|_{L^\infty_t L^2_x}
& \les
\sum_{\l_5} \sum_{\l_6 = \l_6(\l_5)}
\|\ft{f}_{5,\l_5}\|_{L^\infty_t \l^1_{n}}
\| \ft{ f}_{6,\l_6(\l_5)}\|_{L^\infty_t \l^1_{n}}
\|  \ft{f}_7\|_{L^\infty_t \l^2_{n}}\notag\\
& \les 2^L \prod_{i = 5}^7 \|f_i\|_{L^\infty_t L^2_x}.
\label{R6c3}
\end{align}

\noi
In the last step, 
after Cauchy-Schwarz inequality in $n_5$ and $n_6$, 
we also applied Cauchy-Schwarz inequality in $\l_5$
to sum  the contribution from $\ft {f}_{5, \l_5}(n, t):= \ind_{J_{5\l_5}}(n) \cdot \ft{f}_5(n, t)   $ and 
$ \ft{f}_{6, \l_6(\l_5)}$ over $\l_5$.

As before, we set $2^K \sim n_1^*$.
Without loss of generality, assume $k_2 \geq k_3 \geq k_4$.
Then, by proceeding as in Subcase (c.i) with \eqref{R6c3a} and \eqref{R6c3}, 
we have 
\begin{align*}
|\I(T)|  & \les 
\sum_{L= 0}^{k_1+10}T 
\Theta(\bar n )|\wt \Phi(\bar n)|^{-\frac{1}{2}}
(n_1^*)^{-2s+}
2^L 2^\frac{k_3+k_4}{2}
\prod_{i = 2}^7 \|u_i\|_{F^{\al}_{k_i}}\\
 & \les 
T 
(m_1^*)^{-\frac{1}{2}}
(n_1^*)^{-1-2s+}
(n_3^*)^{-1}
(n_4^*)^{2s}
2^\frac{k_1+ k_3+k_4}{2}
\prod_{i = 2}^7 \|u_i\|_{F^{\al}_{k_i}}.
\end{align*}

\noi
Hence, by \eqref{R6c2a}, we obtain \eqref{R63} for $s > -\frac{1}{8}$.

\medskip

\noi
$\bullet $ {\bf Case (d):}
$\max( |\s_5|, |\s_6|, |\s_7|)
\ges |\wt \Phi(\bar n)|$.
\\
\indent
In this case, 
we set $K$ by $2^K \sim m_1^*$.
Recall that we assume $|n_5| \geq |n_7|$.
In particular, we have $|n_5|\sim \max(|n_6|, |n_7|) \sim m_1^*$.

\medskip

\noi
$\circ$ 
Subcase (d.i):  $|n_7| \sim m_3^* \ll |n_1|$.
\\
\indent
In this case, we have $|\wt \Phi(\bar n)| \ges 2^{k_1} m_1^* $.
Write $I_{k_i^*} = \bigcup_{\l_i}J_{i\l_i}$
with $|J_{i\l_i}|\sim n_3^*$, $i = 1, 2$.
Then, 
with $n_1 - n_2 + n_3 - n_4 = 0$, 
write the terms corresponding to $n_1^*$ and $n_2^*$
as sums of their restrictions 
on $J_{i\l_i} $ as in Subcase (c.iii).

\begin{itemize}
\item[(i)]
$|n_1| \sim n_4^*$.\quad Proceeding as in \eqref{R6c3}, we have
\begin{align}
\|\P_{k_1}(\cj{f_2} f_3 \cj{f_4})\|_{L^\infty_{x, t}}
\les (n_4^*)^\frac{1}{2}
\|\cj{f_2} f_3 \cj{f_4}\|_{L^\infty_{t}L^2_x}
\les 
 n_3^* (n_4^*)^\frac{1}{2}
\prod_{i = 2}^4 \|f_i\|_{L^\infty_tL^2_x}.
\label{R6d0a}
\end{align}

\smallskip
\item[(ii)]
$|n_1| \gg n_4^*$.\quad 
By Young's inequality followed by Cauchy-Schwarz inequality, we have
\begin{align}
\|\cj{f_2} f_3 \cj{f_4}\|_{L^\infty_{x, t}}
\les 
\big\|\F_x(\cj{f_2} f_3 \cj{f_4})\big\|_{L^\infty_{t} \l^1_n}
\les 
 n_3^* (n_4^*)^\frac{1}{2}
\prod_{i = 2}^4 \|f_i\|_{L^\infty_tL^2_x}.
\label{R6d0}
\end{align}

\end{itemize}

\noi
Now, 
we can basically proceed as in Subcase (c.i),
after switching the indices $\{5, 6, 7\} \leftrightarrow \{2, 3, 4\}$.
With  \eqref{R6d0a},  \eqref{R6d0}, \eqref{R64},  and Lemma \ref{LEM:infty}, we have 
\begin{align*}
|\I(T)|  & \les 
\sum_{|m|\leq T 2^{[\al K]}}\Theta(\bar n )
 n_3^*( n_4^*)^\frac{1}{2}
\bigg(\prod_{i = 2}^4 \|f_i\|_{L^\infty_tL^2_x}\bigg)
\|f_5 \cj{f_6}f_7\|_{ L^1_{x, t}}
\notag \\
& \les 
\sum_{|m|\leq T 2^{[\al K]}}\Theta(\bar n )
n_3^* (n_4^*)^\frac{1}{2}
|\wt \Phi(\bar n)|^{-\frac{1}{2}}(m_1^*)^{-\frac{\al}{2}}(m_3^*)^\frac{1}{2}
\prod_{i = 2}^7 \|f_i\|_{F^{\al}_{k_i}}\notag \\
& \les T
(n_1^*)^{-\frac{1}{2}+2s}
(m_1^*)^{-\frac{1}{2}-2s + }
\prod_{i = 2}^7 \|u_i\|_{F^{\al}_{k_i}}
 \les T
\prod_{i = 2}^7 2^{(s-)k_i}\|u_i\|_{F^{\al}_{k_i}},
\end{align*}

\noi
where the last inequality holds for $s > -\frac{1}{8}$.

\medskip

\noi
$\circ$ 
Subcase (d.ii):  $|n_7| \sim   |n_1|$.
\\
\indent
In this case, we first estimate the contribution
from $|n_7-n_1|\sim 2^L$
for each $L \leq k_1 + 10$
as in Subcase (c.iii).
When 
$|n_1| \sim n_4^*$, it follows from 
 Young's inequality that 
\begin{align}
\|\P_{k_1}(\cj{f_2} f_3 \cj{f_4})\|_{L^\infty_{t}L^2_x}
\les (n_4^*)^\frac{1}{2}
\|\F_x(\cj{f_2} f_3 \cj{f_4})\|_{L^\infty_{t}\l_n^\infty}
\les 
( n_3^*)^\frac{1}{2} (n_4^*)^\frac{1}{2}
\prod_{i = 2}^4 \|f_i\|_{L^\infty_tL^2_x}.
\label{R6d0b}
\end{align}

\noi
When 
$|n_1| \sim n_3^*$, a similar computation yields \eqref{R6d0b}, 
while \eqref{R6d0b} trivially follows
when $|n_1| \sim n_1^*$.

With \eqref{R6c3} (where $L^\infty_t$ on the left-hand side is replaced
by $L^1_t$) and \eqref{R6d0b}, we have
\begin{align*}
|\I(T)|  & \les 
\sum_{L= 0}^{k_1+10}T 
\Theta(\bar n )|\wt \Phi(\bar n)|^{-\frac{1}{2}}
(m_1^*)^{-2s+}
2^L 
(n_3^*)^{\frac{1}{2}} 
(n_4^*)^{\frac{1}{2}} 
\prod_{i = 2}^7 \|u_i\|_{F^{\al}_{k_i}}\\
 & \les 
T 
(m_1^*)^{-\frac{1}{2}-4s+}
(n_1^*)^{-\frac{1}{2}-3s}
(n_3^*)^{-\frac{1}{2}-s} 
(n_4^*)^{\frac{1}{2}+2s} 
\prod_{i = 2}^7 2^{(s-)k_i}\|u_i\|_{F^{\al}_{k_i}}.
\end{align*}

\noi
Hence, \eqref{R63} holds for $s > -\frac{1}{8}$.

\medskip

\noi
$\circ$ 
Subcase (d.iii):  $|n_7| \gg  |n_1|$.
\\
\indent
In this case, we have $|\wt \Phi(\bar n)| \ges  m_1^*  m_3^* $.
We first assume that $m_3^* \les (n_1^*)^3$.
Then, 
a slight modification of Subcase (d.i) yields
\begin{align*}
|\I(T)|  & \les 
\sum_{|m|\leq T 2^{[\al K]}}\Theta(\bar n )
 n_3^*( n_4^*)^\frac{1}{2}
\bigg(\prod_{i = 2}^4 \|f_i\|_{L^\infty_tL^2_x}\bigg)
\|f_5 \cj{f_6}f_7\|_{ L^1_{x, t}}
\notag \\
& \les 
\sum_{|m|\leq T 2^{[\al K]}}\Theta(\bar n )
n_3^* (n_4^*)^\frac{1}{2}
|\wt \Phi(\bar n)|^{-\frac{1}{2}}(m_1^*)^{-\frac{\al}{2}}(m_3^*)^\frac{1}{2}
\prod_{i = 2}^7 \|f_i\|_{F^{\al}_{k_i}}\notag \\
& \les T
(n_1^*)^{-\frac{1}{2}+2s}
(m_1^*)^{-\frac{1}{2}- 2s +  }
\prod_{i = 2}^7 \|u_i\|_{F^{\al}_{k_i}}
 \les T
\prod_{i = 2}^7 2^{(s-)k_i}\|u_i\|_{F^{\al}_{k_i}},
\end{align*}

\noi
where the last inequality holds for $s > -\frac{1}{8}$.

Next, we consider the case $m_3^* \gg (n_1^*)^3$.
Suppose $|\s_5| \ges  |\wt \Phi(\bar n)| $.
By Lemma \ref{LEM:infty}, we have
\begin{align}
\|\cj{f_2} f_5 \cj{f_4}\|_{L^2_{x, t}}
\leq \|f_5\|_{L^2_{x, t}}\|f_2\|_{L^\infty_{x, t}}\|f_4\|_{L^\infty_{x, t}}
\les
|\wt \Phi(\bar n)|^{-\frac{1}{2}}
 2^\frac{k_2+k_4}{2}\prod_{i = 2, 4, 5} \|f_i\|_{F^{\al}_{k_i}}.
\label{R6d2}
\end{align}

\noi
By Lemma \ref{LEM:L64} (a), we have 
\begin{align}
\|f_3 \cj{f_6}f_7\|_{L^2_{x, t}}
\les 2^\frac{k_3}{2} 
(m_1^*)^{-\frac{\al}{2}}
\prod_{i = 3, 6, 7} \|f_i\|_{F^{\al}_{k_i}}.
\label{R6d3}
\end{align}

\noi
Indeed, \eqref{R6d3} follows from 
Lemma \ref{LEM:L64} (a)
under $\{n_1, n_2, n_3, n_4\}
\leftrightarrow \{n_3, n_6, n_7, n\},$
where $n$ is the frequency of the duality variable given by $n = n_3 - n_6 + n_7$.
Here, we used the fact that 
 $|n|\sim m_1^* \gg |n_3|$,  since
$|n-n_3| = |n_7-n_6|
= |n_5-n_1|
 \sim m_1^*$.
Then, with \eqref{R64}, we have
\begin{align}
\sum_{|m|\leq T 2^{[\al K]}}\Theta(\bar n )
& |\wt \Phi(\bar n)|^{-\frac{1}{2}}
 2^\frac{k_2+k_3+k_4}{2}
(m_1^*)^{-\frac{\al}{2}}\notag \\
&  \les T (m_1^*)^{-\frac{1}{2}-2s+} (m_3^*)^{-\frac{1}{2}}
(n_3^*)^{-\frac{1}{2}}(n_4^*)^{2s}
 \les T
\prod_{i = 2}^7 2^{(s-)k_i}
\label{R6d4}
\end{align}

\noi
for $s > -\frac{1}{8}$.
From \eqref{R6d2}, \eqref{R6d3}, and \eqref{R6d4},  
we obtain \eqref{R63} for $s > -\frac{1}{8}$.

If $|\s_7| \ges  |\wt \Phi(\bar n)| $
and $|n_7| \sim m_1^*$, 
then the argument follows by switching the indices $5 \leftrightarrow 7$.
Namely, 
we apply
Lemma \ref{LEM:L64} (a)
on  $f_3\cj{f_6}f_5$
under $\{n_1, n_2, n_3, n_4\}
\leftrightarrow \{n_3, n_6, n_5, n\},$
where $n$ is the frequency of the duality variable given by $n = n_3 - n_6 + n_5$.
If  $|\s_7| \ges  |\wt \Phi(\bar n)| $
and $|n_7| \sim m_3^* \ll m_1^*$, 
then 
Lemma \ref{LEM:L64} (a)
on  $f_3\cj{f_6}f_5$
only yields a gain of $\max\big((m_1^*)^{-\frac{\al}{2}}, (m_3^*)^{-\frac{1}{2}}\big)$,
since 
$|n-n_3| = |n_5-n_6|
= |n_7-n_1|
 \sim m_3^*$.
In this case, 
before proceeding as in \eqref{R6d2} with $5\leftrightarrow 7$,
we group
$\cj{f_{2}} f_{7} \cj{f_{4}}= \cj{f_{2, m}} f_{7, m} \cj{f_{4, m}}$, 
which are localized on time intervals of length $\sim (m_1^*)^{-\al}$, 
into components 
localized on time intervals of length $\sim (m_3^*)^{-\al} $.
Then, by Cauchy-Schwarz inequality, 
we have 
\begin{align}
|\I(T)|
& \les 
\Theta(\bar n)
\Big(\sum_{|m| \leq T2^{[\al K]}}
\|\cj{f_{2, m}} f_{7, m} \cj{f_{4, m}}\|_{L^2_{x, t}}^2\Big)^\frac{1}{2}
\Big(\sum_{|m| \leq T2^{[\al K]}} \|
f_{3, m}\cj{f_{6, m}}f_{5, m}\|_{L^2_{x, t}}^2\Big)^\frac{1}{2} \notag \\
& \les 
T \Theta(\bar n)
(m_1^*m_3^*)^\frac{\al}{2}
\Big(\sup_{\wt m} \|\cj{f_{2, \wt m}} f_{7, \wt m} \cj{f_{4, \wt m}}
\|_{L^2_{x, t}}\Big)
\Big(\sup_m \|f_{3, m}\cj{f_{6,m }}f_{5, m}\|_{L^2_{x, t}}\Big).
\label{R6d5}
\end{align}

\noi
With \eqref{R64}, we have
\begin{align}
T \Theta(\bar n)
(m_1^*m_3^*)^\frac{\al}{2}
& |\wt \Phi(\bar n)|^{-\frac{1}{2}}
 2^\frac{k_2+k_3+k_4}{2}
(m_3^*)^{-\frac{1}{2}}\notag \\
&  \les T (m_1^*)^{-\frac{1}{2}-2s+} (m_3^*)^{-1-2s+}
(n_3^*)^{-\frac{1}{2}}(n_4^*)^{2s}
 \les T
\prod_{i = 2}^7 2^{(s-)k_i}
\label{R6d6}
\end{align}

\noi
for $s > -\frac{1}{8}$.
Hence, 
from  \eqref{R6d2}  (with $5 \leftrightarrow 7$), 
Lemma \ref{LEM:L64} (a)
on  $f_3\cj{f_6}f_5$, 
\eqref{R6d5},   
and \eqref{R6d6}, 
we obtain \eqref{R63} for $s > -\frac{1}{8}$.

Lastly, suppose that  $|\s_6| \ges  |\wt \Phi(\bar n)| $.
When $|n_5 - n_7|\sim m_1^*$, 
we group $f_j$'s into
$\cj{f_2} f_3\cj{f_6}$
and 
$f_5 \cj{f_4} f_7$.
Then, 
we use (a slight modification of) \eqref{R6d2}
and \eqref{R6d3}
for $\cj{f_2} f_3\cj{f_6}$
and 
 $f_5 \cj{f_4} f_7$, respectively.
In particular, 
the analogue of \eqref{R6d3}
for  $f_5 \cj{f_4} f_7$
follows from 
Lemma \ref{LEM:L64} (b)
under $\{n_1, n_2, n_3, n_4\}
\leftrightarrow \{n_7, n_4, n_5, n\},$
where $n  = n_5 - n_4 + n_7$.
Hence, \eqref{R63} holds for $s > -\frac{1}{8}$.

When $|n_5 - n_7|\ll m_1^*$, 
it follows from $|n_5|\sim m_1^*$
and $|n_5 - n_6 + n_7| = |n_1| \ll m_1^*$
that 
$|n_5| \sim |n_6| \sim |n_7| \sim m_1^*$.
In particular, we have 
 $|\wt \Phi(\bar n)| \ges  (m_1^* )^2$.
Then, 
a slight modification of Subcase (d.i) with  $L^4_{x, t}, L^2_{x, t}, L^4_{x, t}$-H\"older's inequality  on $f_5 \cj{f_6}f_7$ yields
\begin{align*}
|\I(T)|  & \les 
\sum_{|m|\leq T 2^{[\al K]}}\Theta(\bar n )
 n_3^*( n_4^*)^\frac{1}{2}
\bigg(\prod_{i = 2}^4 \|f_i\|_{L^\infty_tL^2_x}\bigg)
\|f_5 \cj{f_6}f_7\|_{ L^1_{x, t}}
\notag \\
& \les 
\sum_{|m|\leq T 2^{[\al K]}}\Theta(\bar n )
n_3^* (n_4^*)^\frac{1}{2}
|\wt \Phi(\bar n)|^{-\frac{1}{2}}(m_1^*)^{-\frac{\al}{4}}
\prod_{i = 2}^7 \|f_i\|_{F^{\al}_{k_i}}\notag \\
& \les T
(m_1^*)^{-1- 3s +  }
(n_1^*)^{-\frac{1}{2}+2s}
\prod_{i = 2}^7 \|u_i\|_{F^{\al}_{k_i}}
 \les T
\prod_{i = 2}^7 2^{(s-)k_i}\|u_i\|_{F^{\al}_{k_i}},
\end{align*}

\noi
where the last inequality holds for $s > -\frac{1}{6}$.

This completes the proof of Proposition \ref{LEM:R6M}.
\end{proof}

\begin{remark}\label{REM:smoothing}
\rm 
We point out that these energy estimates
have a certain smoothing property.
From the proofs of 
Propositions \ref{LEM:R4M}, \ref{LEM:LdM}, 
and \ref{LEM:R6M}, 
it is easy to see that 
the energy estimates 
\eqref{R4M1}, \eqref{LdM}, and \eqref{R6M1}
hold even if we replace
the $F^{s, \al}(T)$-norms on the right-hand sides 
by $F^{s-\dl, \al}(T)$-norms
for some small $\dl > 0$.
Namely, we can take the regularity on the right-hand sides 
 to be slightly less than that on the left-hand sides.
\end{remark}


\section{Existence of  solutions to the Wick ordered cubic NLS}
\label{SEC:existence}
In this section, we present the proof of Theorem \ref{THM:1}.
First, we establish an a priori bound of smooth solutions
to the Wick ordered cubic NLS \eqref{NLS1}.
Fix $s \in (-\frac{1}{8}, 0)$
and $\al = -4s +$.
Let $u\in C(\R; H^\infty(\T))$ be a smooth global solution to \eqref{NLS1}
with initial condition $u_0 \in H^\infty(\T)$.
Then, 
by Propositions  \ref{PROP:linear},  \ref{PROP:3lin}, 
\ref{LEM:R4M}, \ref{LEM:LdM}, 
and \ref{LEM:R6M}, we have\footnote{When $s>-\frac1{10}$, we do not need the last two terms in \eqref{K3}.
Moreover, we can choose $c(s) = 0 $ in this case.  See Remark \ref{REM:ene}.}
\begin{align}
&  \|u\|_{F^{s, \al}(T)}   
\les \|u\|_{E^s(T)} + \|\mathfrak N(u)\|_{N^{s, \al}(T)}, 
\label{K1}\\
 & \|\mathfrak N(u)\|_{N^{s, \al}(T)} 
  \les  T^\theta \|u\|_{F^{s, \al}(T)}^3,
 \label{K2}\\
&  \|u\|_{E^s(T)}^2   \leq \|u_0\|_{H^s}^2 +  
C\Big( T^\theta M^{c(s)} \|u\|_{F^{s, \al}(T)}^4
+ M^{-d(s)} \|u\|_{F^{s, \al}(T)}^4
+  T^\theta\|u\|_{F^{s, \al}(T)}^6\Big),
\label{K3}
\end{align}

\noi
for $T\in (0, 1]$, $M \in \mathbb{N}$, $c(s) \geq 0$, $d(s)>0$, and $\theta>0$, 
where $\mathfrak{N}(u) = \N(u) + \RR(u)$.
In the following, $T$ and $M$ will be chosen in terms of $\|u_0\|_{H^s}$.

First, we establish an a priori bound on the smooth solution $u$
in terms of the $H^s$-norm of the initial condition $u_0$.
Let $X(T) = \|u\|_{E^s(T)} + \|\mathfrak N(u) \|_{N^{s, \al}(T)}$.
Then, we have the following lemma.

\begin{lemma}\label{LEM:Kconti}
Let $u\in C([-1, 1] ; H^\infty(\T))$.
Then, 
$X(T)$ is non-decreasing and continuous in $T \in [0, 1]$.
Moreover, we have
\[ \lim_{T\to 0} X(T) = \|u(0)\|_{H^s}.\]
\end{lemma}

\begin{proof}
From the definition, $X(T)$ is non-decreasing in $T$.
Let $I_T = [-T, T]$.
Then, given $ 0 \leq  T_1 < T_2 \leq1$, we have 
\begin{align*}
\|u\|_{E^s(T_2)}^2- \|u\|_{E^s(T_1)}^2
& = \sum_{k \geq 1} 2^{2sk} \bigg( \sup_{t_k \in I_{T_2}}  \|\P_ku(t_k)\|_{L^2}^2
- \sup_{\wt t_k \in I_{T_1}}  \|\P_ku(\wt t_k)\|_{L^2}^2\bigg)\\
\intertext{By letting $\mathcal{A} = \big\{ k\geq1 : 
\sup_{t_k \in I_{T_2}}  \|\P_ku(t_k) \|_{L^2}^2
= \sup_{t_k \in I_{T_2}\setminus I_{T_1}}  \|\P_ku(t_k)\|_{L^2}^2\big\}$, 
we have}
& \leq  \sum_{k \in \mathcal{A}} 
2^{2sk} \bigg( \sup_{t_k \in I_{T_2}\setminus I_{T_1}} \|\P_ku(t_k)\|_{L^2}^2
-  \|\P_ku(T_1)\|_{L^2}^2\bigg)\\
& \leq \sup_{t_* \in I_{T_2}\setminus I_{T_1}}
\sum_{k \in \mathcal{A}}  2^{2(s+\eps)k} \bigg(  \|\P_ku(t_*)\|_{L^2}^2
-  \|\P_ku(T_1)\|_{L^2}^2\bigg)\\
& \leq 2
\| u\|_{C([-1, 1]; H^{s+\eps})}
\sup_{t_* \in I_{T_2}\setminus I_{T_1}}
\| u (t_*) - u(T_1) \|_{ H^{s+\eps}} \longrightarrow 0, 
\end{align*}

\noi
as $T_2 - T_1 \to 0$.
This shows the continuity of $\|u\|_{E^s(T)}$.
In particular, we have
\begin{align*}
\lim_{T\to 0} \|u\|_{E^s(T)} = \|u(0)\|_{H^s}.
\end{align*}

Let $Y(T) = \|u\|_{N^{s, \al}(T)}$.
Then, with Lemma \ref{LEM:embed3}, it is easy to see that 
\begin{align*}
Y(T) \les T^\frac{1}{2} \|u\|_{C([-T, T]: H^s)} \longrightarrow 0, 
\end{align*}

\noi
as $T \to 0$.
Hence, it remains to prove the  continuity of $Y(T)$.
Let  $T_0 \in (0, 1]$.
First, by the monotonicity in $T$ 
of the $N^{s, \al}(T)$-norm and $\|u\|_{N^{s, \al}((1+\dl)T_0)} <\infty$
(with some small $\dl > 0$\footnote{Set $\dl = 0$ when $T_0 = 1$.}
 such that $(1+\dl) T_0 \leq 1$),
it follows from Monotone Convergence Theorem
that given $\eps > 0$, there exists $K \in \mathbb{N}$ such that 
$\|\P_{>K}u\|_{N^{s, \al}(T)} < \eps$
for all $T < (1+ \dl) T_0$.
Then, it suffices to show that there exists $\dl_0 = \dl_0(T_0, \eps)>0$
 such that 
\begin{align}
\big|\|\P_{\leq K}u\|_{N^{s, \al}(T_0)} - 
\|\P_{\leq K}u\|_{N^{s, \al}(rT_0)}\big| < \eps
\label{CT1}
\end{align}
	
\noi
for all $r \in \R$ with $|r-1| < \dl_0$.
In the following, we restrict our attention
to $\P_{\leq K}u$ for this fixed $K = K(T_0, \eps)$.
For simplicity, we denote $\P_{\leq K}u$ by $u$,
i.e.~we assume that the spatial frequencies are supported on $\{|n| < 2^K\}$
in the following.

As in \cite{IKT}, we introduce the dilation operator $D_r$ given by
 $D_r(u)(x, t) = u(x, r^{-1}t)$
 for $r$ close to $1$.
Then, with $r = T/T_0$, by the triangle inequality
and Lemma \ref{LEM:embed3},  we have
\begin{align*}
\big|\|u\|_{N^{s, \al}(T)}
- \|D_{r} (u)\|_{N^{s, \al}(T)}\big|
& \leq \|u - D_r (u)\|_{N^{s, \al}(T)} \notag \\
& \les (T)^\frac{1}{2}\|u - D_r (u)\|_{L^\infty_t ([-T, T]; H^s)}
\to 0
\end{align*}

\noi
as $T \to T_0$, i.e. as $r \to 1$, since 
$u \in C([-1, 1]; H^\infty(\T))$.
Then, it remains to show that
\begin{equation}
\lim_{r\to 1}  \|D_{r} (u)\|_{N^{s, \al}(rT_0)} = \|u\|_{N^{s, \al}(T_0)}.
\label{CT2}
\end{equation}

\noi
Note that \eqref{CT2} follows once we show
\begin{align}
&   \|u\|_{N^{s, \al}(T_0)} \leq \liminf_{r\to 1}  \|D_{r} (u)\|_{N^{s, \al}(rT_0)}
 \label{CT3}
 \\
 \intertext{and} & \limsup_{r\to 1}  \|D_{r} (u)\|_{N^{s, \al}(rT_0)} \leq  \|u\|_{N^{s, \al}(T_0)}.
 \label{CT4}
\end{align}

We first prove \eqref{CT3} in the following.
Given small $\wt \eps > 0$, 
let $u_r$ be an extension of $D_r(u)$ outside $[-rT_0, rT_0]$\footnote{Strictly speaking,
we construct extensions 
$u_r^{k}$ for $\P_k D_r(u)$, $k \in \Z_+\cap [0, K]$
and then construct $u_r$ by setting $\P_k(u_r):= u_r^k$.
By construction, we have $\P_{\leq K}(u_r) = u_r$.}
such that 
\begin{equation*}
\|u_r\|_{N^{s, \al}} \leq  \|D_{r} (u)\|_{N^{s, \al}(rT_0)} + \wt \eps.
\end{equation*}

\noi
Then, $D_{\frac{1}{r}}(u_r)$ is an extension of  $u$
outside $[-T_0, T_0]$ and by definition,  we have
\begin{equation*}
\|u\|_{N^{s, \al}(T_0)}
\leq \|D_{\frac{1}{r}}(u_r)\|_{N^{s, \al}} .
\end{equation*}

\noi
Hence, it suffices to prove
\begin{align}
\|D_{\frac{1}{r}}(u_r)\|_{N^{s, \al}} 
\leq \psi(r) \|u_r\|_{N^{s, \al}}
\label{CT7}
\end{align}
	
\noi
for some continuous function $\psi(r)$ defined in a neighborhood of $r = 1$
such that $\lim_{r \to 1} \psi(r) = 1$.

\medskip

Fix $k \in \Z_+\cap [0, K]$.
With $\eta_0^r(t) = \eta_0(r^{-1}t)$, we have
\[ \eta_0(2^{[\alpha k]}(t -t_k)) \cdot 
\P_k D_\frac{1}{r}(u_r)
= \P_k  D_\frac{1}{r}\big( \eta_0^r(2^{[\alpha k]}(t -r t_k)) \cdot 
u_r\big).
\]

\noi
Hence, we have
\[\F[ \eta_0(2^{[\alpha k]}(t-t_k)) \cdot 
\P_k D_\frac{1}{r}(u_r) ](n, \tau)
=r^{-1}
\ind_{I_k} (n) \,  \F[ \eta_0^r (2^{[\alpha k]}(t-rt_k)) \cdot 
u_r ](n, r^{-1} \tau). \]

\noi
Then, 
we have
\begin{align}
\|\P_k & D_{\frac{1}{r}}(u_r)\|_{N^{\al}_k} 
 = \sup_{t_k\in\R} 
\big\|(\tau-n^2+i2^{[\alpha k]})^{-1}
\F[ \eta_0(2^{[\alpha k]}(t-t_k)) \cdot 
\P_k D_\frac{1}{r}(u_r) ]\big\|_{X_k}\notag\\
& = r^{-\frac{1}{2}} \sup_{t_k\in\R}
\sum_{j = 0}^\infty 2^\frac{j}{2}
\bigg\| \frac{ \eta_j(r \tau - n^2)}{|r\tau - n^2 + i 2^{[\al k]}|}
\ind_{I_k} (n) \,  \F[ \eta_0^r (2^{[\alpha k]}(t-t_k)) \cdot 
u_r ]
\bigg\|_{\l^2_n L^2_\tau}.
\label{CT8}
\end{align}

We first examine the integrand in \eqref{CT8}.
If $|\tau| \gg n^2$, we have
$|\tau| \sim |\tau - n^2| \sim |r\tau - n^2|  \sim 2^j $
for $r$ sufficiently close to $1$ on the support of $\eta_j(r \tau - n^2)$.
Otherwise, we have $|\tau| \les n^2 \sim 2^{2k}$.
Thus, we have
\begin{align}
\bigg|\frac{ 1}{(r\tau - n^2)^2 +  2^{2[\al k]}}
- \frac{ 1}{(\tau - n^2)^2 +  2^{2[\al k]}}\bigg| 
& = 
\bigg|\frac{ (1-r) \tau \big((1+r)\tau -2n^2\big)}{((r\tau - n^2)^2 +  2^{2[\al k]})((\tau - n^2)^2 +  2^{2[\al k]})} \bigg|
\notag \\
& \les |r-1| 2^{4k}
\frac{1}{(\tau - n^2)^2 +  2^{2[\al k]}}.
\label{CT9}
\end{align}

\noi
On the other hand, by Mean Value Theorem, we have
\begin{align}
\big|\eta_j(r \tau - n^2)
- \eta_j( \tau - n^2)\big|
& \les |r-1| \sup_{s\in J_r}2^{-j}|\tau|
|\eta'(2^{-j}(s \tau - n^2))| \notag \\
& \les \begin{cases}
|r-1| \sum_{|j' - j| \leq 1} \eta_{j'}( \tau - n^2), & \text{if } j \geq 2k + 5,\\
|r-1| 2^{-j} 2^{2k} \sum_{j' \leq 2k + 10} \eta_{j'}( \tau - n^2), & \text{otherwise}, 
\end{cases}
\label{CT10}
\end{align}

\noi
for $r$ sufficiently close to $1$,
where $J_r = [1, r]$ if $r > 1$ and $= [r, 1]$ if $r < 1$.

From \eqref{CT8}, \eqref{CT9}, and \eqref{CT10}, we have
\begin{align}
\|\P_k  D_{\frac{1}{r}}(u_r)\|_{N^{\al}_k} 
& \leq \varphi_0(r)  \sup_{t_k\in\R}
\sum_{j = 0}^\infty 2^\frac{j}{2}
\bigg\| \frac{ \eta_j( \tau - n^2)}{|\tau - n^2 + i 2^{[\al k]}|}
\ind_{I_k} (n) \,  \F[ \eta_0^r (2^{[\alpha k]}(t-t_k)) \cdot 
u_r ] \bigg\|_{\l^2_n L^2_\tau}\notag \\
& \hphantom{X}
 + \I_k + \II_k + \III_k,
\label{CT11}
\end{align}

\noi
where 
$\varphi_0(r) = r^{-\frac{1}{2}}(1+ c |r-1|) \to 1$
as $r\to 1$, 
and 
$\I_k$, $\II_k$, $\III_k$ are  given by 
\begin{align*}
\I_k
& \sim  \varphi_1(r) 2^{2k}\sup_{t_k\in\R}
\sum_{j = 0}^\infty 2^\frac{j}{2}
\bigg\| \frac{ \eta_j( \tau - n^2)}{|\tau - n^2 + i 2^{[\al k]}|}
\ind_{I_k} (n) \,  \F[ \eta_0^r (2^{[\alpha k]}(t-t_k)) \cdot 
u_r ] \bigg\|_{\l^2_n L^2_\tau}, \\
\II_k
& \sim \varphi_1(r)  2^{2k}\sup_{t_k\in\R}
\sum_{j = 0}^{2k+10} 
\bigg\| \frac{ \eta_j( \tau - n^2)}{|\tau - n^2 + i 2^{[\al k]}|}
\ind_{I_k} (n) \,  \F[ \eta_0^r (2^{[\alpha k]}(t-t_k)) \cdot 
u_r ] \bigg\|_{\l^2_n L^2_\tau},\\
\III_k
& \sim \varphi_2(r)  2^{4k}\sup_{t_k\in\R}
\sum_{j = 0}^{2k+10} 
\bigg\| \frac{ \eta_j( \tau - n^2)}{|\tau - n^2 + i 2^{[\al k]}|}
\ind_{I_k} (n) \,  \F[ \eta_0^r (2^{[\alpha k]}(t-t_k)) \cdot 
u_r ] \bigg\|_{\l^2_n L^2_\tau}.
\end{align*}

\noi
Here, $\varphi_1(r)$ and $\varphi_2(r)$ 
are given by 
\[\varphi_1(r) = r^{-\frac{1}{2}} |r-1|^\frac{1}{2}
\qquad
\text{and} \qquad 
\varphi_2(r) = r^{-\frac{1}{2}} |r-1|^\frac{3}{2}\]

\noi
such that $\lim_{r\to 1} \varphi_i(r) = 0$ for $i = 1, 2$.

It remains to deal with $\eta_0^r$ in the integrand.
Let $\g: \R \to [0, 1]$ be a smooth cutoff function supported on $[-1, 1]$
such that $\g(t) \equiv 1$ on $[-\frac{1}{4}, \frac{1}{4}]$
and
$\sum_{m\in \Z} \g(t - m) \equiv 1$.
Note that we have $\g(t) = \g(t) \cdot \eta_0( t) $.

Fix $t_k \in \R$ in the following.
By Fundamental Theorem of Calculus, we have
\begin{align}
\eta_0^r(2^{[\al k]}(t - t_k)) - \eta_0(2^{[\al k]}(t - t_k))
= \int_1^{r^{-1}} \zeta_s (2^{[\al k ]} (t - t_k)) ds,
\label{CT12}
\end{align}

\noi
where $\zeta_s(t) = s^{-1} \zeta(st)$
with $\zeta(t)= 
 t \eta_0'(t) \in \mathcal{S}(\R_t)$.
Note that 
the left-hand side of \eqref{CT12} is supported
on 
$\{t: t - t_k = O(2^{-[\al k]+2})\}$
for $r$ close to 1.
Then, denoting the right-hand side of \eqref{CT12} by $F(t-t_k)$,
we have
\begin{align}
 F(t- t_k) 
 & = F(t- t_k) \sum_{|m|\leq C} \g (2^{[\al k]}(t-t_k) - m) \notag \\
 & = F(t- t_k) \sum_{|m|\leq C} \g (2^{[\al k]}(t-t_k) - m)
\cdot  \eta_0 (2^{[\al k]}(t-t_k) - m).
\label{CT13}
 \end{align}

\noi
Then,  by Minkowski's integral inequality and Lemma \ref{LEM:embed3}\footnote{
The implicit constant in Lemma \ref{LEM:embed3} is independent
of $s \in [1, r^{-1}]$ (or $[r^{-1}, 1]$) for $r$ close to 1.}
twice 
(for $\zeta$ after Minkowski's integral inequality and for $\g$),
it follows from \eqref{CT11}, \eqref{CT12}, and \eqref{CT13} that 
\begin{align}
\|\P_k  D_{\frac{1}{r}}(u_r)\|_{N^{\al}_k} 
& \leq \wt{\varphi}_0(r) 
\|\P_ku_r\|_{N^{\al}_k} 
 + \wt \varphi_1 (r) 2^{4k} \|\P_ku_r\|_{N^{\al}_k} \notag \\
& \leq \wt{\varphi}_0(r) 
\|\P_ku_r\|_{N^{\al}_k} 
 + \wt \varphi_1 (r) 2^{4K} \|\P_ku_r\|_{N^{\al}_k}, 
\label{CT14}
\end{align}

\noi
where 
$\wt \varphi_0$ and $\wt \varphi_1$ are continuous functions defined 
on a neighborhood of 1 such that 
\[\lim_{r\to 1} \wt \varphi_0(r) = 1
\quad \text{and} \quad \lim_{r\to 1} \wt \varphi_1(r) = 0.\]

\noi
Here, we used our assumption
that  $\P_{\leq K} u_r = u_r$.
Finally, by summing over $k \in \Z_+ \cap [0, K]$, we obtain \eqref{CT7}.

\medskip

Next, we prove \eqref{CT4}.
Given $\wt \eps > 0$, 
let $\wt u$ be an extension of $u$ outside $[-T_0, T_0]$
such that 
\[ \|\wt u \|_{N^{s, \al}} \leq \|u\|_{N^{s, \al}(T_0)} + \wt \eps.\]

\noi
Note that we have $\P_{\leq K} \wt u = \wt u$ as before.
Then, 
$D_r(\wt u)$ is an extension of $D_r(u)$ outside $[-rT_0, rT_0]$
and thus we have
\[ \|D_r(u)\|_{N^{s, \al}(rT_0)}
\leq \|D_r(\wt u) \|_{N^{s, \al}}.\]

\noi
Hence, it suffices to prove
\begin{align}
\limsup_{r\to 1}\|D_r(\wt u) \|_{N^{s, \al}} \leq \|\wt u \|_{N^{s, \al}}.
\label{CT15}
\end{align}

\noi
Noting that \eqref{CT14} holds for functions independent of $r$ (in place of $u_r$), 
we see that \eqref{CT15} follows from \eqref{CT14}.
This completes the proof of Lemma \ref{LEM:Kconti}.
\end{proof}

From \eqref{K1}, \eqref{K2}, and \eqref{K3}, we have
\begin{align*}
X(T)^2 \le C_1 \|u_0\|_{H^s}^2 + 
C_2\Big( T^\theta M^{c(s)} X(T)^4
+ M^{- d(s)} X(T)^4 + T^\theta X(T)^6\Big)
\end{align*}

\noi
for $T \in (0, 1]$.
With $R = C_1^\frac{1}{2}\|u_0\|_{H^s}$, 
first choose  $M = M(R)$ sufficiently large such that 
\begin{align} \label{K3a}
C_2 M^{-d(s)} (2R)^2 < \tfrac{1}{4}.
\end{align}

\noi
Then, choose $T_0 = T_0(R)\leq1$ sufficiently small such that 
\begin{align}
 C_2T_0^\theta \big(M^{c(s)}(2R)^2 + (2R)^4\big) < \tfrac{1}{4}.
\label{K3b}
 \end{align}

\noi
In view of Lemma \ref{LEM:Kconti}, 
a continuity argument  yields
$X(T) \leq 2R$ for $T \leq T_0$.
Hence, from \eqref{K1}, we obtain 
\begin{equation}
\|u\|_{F^{s, \al}(T)}
\les \|u_0\|_{H^s}
\label{K4a}
\end{equation}

\noi
for $T \leq T_0(\|u_0\|_{H^s})$.

\medskip

We are now ready to prove  existence of solutions to \eqref{NLS1}
for $u_0 \in H^s(\T)$, $s > -\frac{1}{8}$.
Given
 $u_0 \in H^s(\T)$, $s > -\frac{1}{8}$,
 let $u_{0, n} = \P_{\leq n} u_0$ for $n \in \mathbb{N}$.
Then, we have $u_{0, n} \in H^\infty(\T)$
and $u_{0, n} \to u_0$ in $H^s(\T)$.
Moreover, by the classical well-posedness theory,
there exist  smooth global solutions $u_n \in C(\R; H^\infty(\T))$ to \eqref{NLS1} with  $u_n|_{t = 0}= u_{0, n}$.
Then, we have the following compactness lemma
for these solutions $\{u_n\}_{n \in \mathbb{N}}$.

\begin{lemma} \label{LEM:Kcpt}
Given $u_0 \in H^s(\T)$ for some $s > -\frac{1}{8}$,
let $u_n$ be the smooth solutions to \eqref{NLS1}
with  $u_n|_{t = 0}= u_{0, n}$ as above.
Then,
the set $\{ u_n\}_{n\in \mathbb{N}}$ is precompact in $C_TH^s:= C([-T, T]; H^s(\T))$
for $T \leq T_0$.
Here, $T_0 = T_0(\|u_0\|_{H^s})>0$ is given by \eqref{K3b}.
\end{lemma}

\begin{proof}
It follows from Lemma \ref{LEM:embed1} and \eqref{K4a}
that $\{ u_n\}_{n\in\mathbb{N}}$
is uniformly bounded in $C([-T, T]; H^s(\T))$.
Given  $\eps>0$, we claim that there exists $ N_0 \in \mathbb{N}$ such that
\begin{equation}
\|(I-\P_{\leq N_0})u_n\|_{C_TH^s}<\eps
\label{K5a}
\end{equation}

\noi
for all $n \in \mathbb{N}$.
In order to prove \eqref{K5a},
we need to use the smoothing effect of the energy estimates.
Namely,  there exists small $\dl > 0$
such that 
the estimates \eqref{R4M1}, \eqref{LdM}, and \eqref{R6M1}
hold even  with less regularity  of $r = s- \dl< s$ on the right-hand sides.
See Remark \ref{REM:smoothing}.
Let $M$ and $T< T_0$ be as in \eqref{K3a} and \eqref{K3b}.
Then, from (the proofs of) Propositions \ref{LEM:R4M}, \ref{LEM:LdM}, 
and \ref{LEM:R6M}
with \eqref{K4a}, we have
\begin{align*}
\big|\| & \P_{>N}  u_n\|^2_{E^s(T)}  -\|\P_{>N}u_{0, n}\|_{H^s}^2\big|\notag \\
&  \les
(T^\theta M^{c(s)}
+ M^{-d(s)})
\|\P_{>c N}u_n\|_{F^{r,\al}(T)}^2\|u_n\|_{F^{r,\al}(T)}^2
+T^\theta \|\P_{>cN}u_n\|_{F^{r,\al}(T)}^2\|u_n\|_{F^{r,\al}(T)}^4
\notag \\
& 
\les C(\|u_0\|_{H^s}) N^{-2\dl}
\longrightarrow 0
\end{align*}

\noi
as $N\to \infty$, uniformly in $n\in \mathbb N$. 
Hence, we have
\begin{align}
\| \P_{>N} u_n\|_{C_T H^s}^2
& \leq \|\P_{>N} u_n \|_{E^s(T)}^2 
\les \| \P_{>N} u_0\|_{ H^s}^2
+ C(\|u_0\|_{H^s}) N^{-2\dl}
\longrightarrow 0
\label{K5b}
\end{align}

\noi
as $N\to \infty$, uniformly in $n \in \mathbb{N}$.
This proves \eqref{K5a}.

Fix $\eps >0$. By \eqref{K5a}, there exists $N_0>0$ such that
$\|\P_{> N_0}u_n\|_{C_TH^s}<\frac{\eps}{3}$ for all $n\in \mathbb N$.
Clearly, 
$\{\P_{\leq {N_0}}u_n(t)\}_{n \in \mathbb{N}}$
is precompact in $H^s(\T)$
for each $t$.
Moreover, by Lemma \ref{LEM:embed1}, \eqref{linear2}, 
Proposition \ref{PROP:3lin}, and \eqref{K4a}, there exists $\theta > 0$ such that 
\begin{align*}
\|\P_{\leq {N_0}} u_n(t+\dl) - 
& \P_{\leq {N_0}} u_n(t)\|_{H^s}\\
&  \leq  \|(S(\dl) - 1)\P_{\leq {N_0}} u_n(t)\|_{H^s}  +
\bigg\|\int_t^{t+\dl} S(t + \dl - t') \P_{\leq {N_0}} \mathfrak N( u_n ) (t') dt'\bigg\|_{H^s}\\
& \les \dl N_0^2  \|u_n \|_{C_T H^s} + \dl^\theta  \|u_0 \|_{H^s}^3
\leq C(N_0, \|u_0\|_{H^s}) \dl^\theta
\end{align*}

\noi
for all $n\in \mathbb N$.
Namely,
$\{\P_{\leq {N_0}}u_n(t)\}_{n \in \mathbb{N}}$
is equicontinuous with values in $H^s(\T)$.
By Ascoli-Arzel\`a compactness theorem,
$\{\P_{\leq {N_0}}u_n\}_{n \in \mathbb{N}}$  is precompact in $C([-T, T]; H^s(\T))$.
Hence, there exists a finite cover by balls
of radius $\frac{\eps}{3}$
(in $C_TH^s$)
centered at $\{\P_{\leq {N_0}}u_{n_k}\}_{k = 1}^K$.
Then, the balls of radius $\eps$ (in $C_TH^s$)
centered at $\{u_{n_k}\}_{k = 1}^K$
cover $\{u_{n}\}_{n\in \mathbb N}$.
\end{proof}

In view of Lemma \ref{LEM:Kcpt},
we can extract a subsequence,
which we still denote by
 $\{u_n\}_{n \in \mathbb{N}}$,  converging to some $u$ in $C([-T, T]; H^s(\T))$.
On the one hand, in view of the uniform tail estimate \eqref{K5b},
this subsequence  $\{u_n\}_{n \in \mathbb{N}}$ also converges in $E^s(T)$.
On the other hand, by possibly making $T$ smaller,
\eqref{K1} and \eqref{K2} with \eqref{K4a}
yield 
\begin{align*}
\|u_n-u_m\|_{F^{s,\alpha}(T)}\les \|u_n-u_m\|_{E^s(T)}.
\end{align*}

\noi
Hence, $\{u_n\}$ converges to $u$ in $F^{s,\alpha}(T)$.
Finally, by applying Proposition \ref{PROP:3lin}
to
\begin{align*}
\Nf(u_n) - \Nf(u)
& = \Nf(u_n , u_n, u_n) - \Nf(u, u, u) \\
& = \Nf(u_n - u, u_n, u_n)
+ \Nf( u, u_n - u, u_n)
+ \Nf( u, u , u_n - u) ,
\end{align*}

\noi
we see that $\{\Nf(u_n)\}$ converges to $\Nf(u)$ in $N^{s, \al}(T)$.
Hence, the limit $u$ satisfies \eqref{NLS1} as a distribution.
This completes the proof of Theorem \ref{THM:1}.

\section{Non-existence of  weak solutions to the cubic NLS below $L^2$}
\label{SEC:non-existence}

In this section, we present the proof of Theorem \ref{THM:2}.
We prove this by contradiction.
Fix $s \in (-\frac{1}{8}, 0)$
and $u_0 \in H^s(\T)\setminus L^2(\T)$.
Suppose that there exist $T>0$
and a solution $u \in C([-T, T];H^s(\T))$ to  the standard cubic NLS \eqref{NLS0}
such that 
\begin{itemize}
\item[\textup{(i)}] $u|_{t = 0} = u_0$

\smallskip

\item[\textup{(ii)}] There exist smooth global solutions $\{u_n\}_{n\in \mathbb{N}}$ to \eqref{NLS0} such that 
$u_n \to u$ in $ C([-T, T]:H^s(\T))$ as $n \to \infty$. 
\end{itemize}

\noi
The main idea is to use the a priori bound
to the solutions to the Wick ordered NLS \eqref{NLS1}
from Section \ref{SEC:existence}
and 
exploit the fast oscillation introduced in 
the transformation \eqref{K6c} below.

By setting 
\begin{align}
v_n(t) = e^{- 2 i t \fint_\T |u_{0, n}|^2 dx} u_n(t), 
\label{K6c}
\end{align}

\noi
we see that  $v_n$ is a solution to  the Wick ordered NLS \eqref{NLS1}.
Moreover, 
we have 
 $v_n|_{t = 0}=  u_{n}(0) \to u_0$ in $H^s(\T)$.
By the a priori estimate \eqref{K4a},
and a slight modification of the proof of Lemma \ref{LEM:Kcpt}, 
there exists a subsequence $\{v_{n_k}\}$ converging to  
some $v$ in $C([-T, T]:H^s)$, 
where $T = T(\|u_0\|_{H^s})$.

Indeed, 
by setting $R = \|u_0\|_{H^s} + 1$, 
we obtain the a priori estimate \eqref{K4a}:
\begin{align*}
\| v_n\|_{F^{s, \al}(T)} \les R
\end{align*}

\noi
for $n \geq N_1$.
As for the proof of Lemma \ref{LEM:Kcpt}, we 
need to modify \eqref{K5a}.
In particular, we claim that,  
given  $\eps>0$,  there exist $ N_0, \cj{N} \in \mathbb{N}$ such that
\begin{equation}
\|(I-\P_{\leq N_0})v_n\|_{C_TH^s}<\eps
\label{K6a}
\end{equation}

\noi
for all $n \geq \cj N$.
Indeed, we have 
the following
 in place of \eqref{K5b};
given $\eps > 0$, 
there exist $N_0, \cj N \in \mathbb{N}$
such that 
\begin{align}
\| \P_{>N} v_n\|_{C_T H^s}^2
& 
\les \| \P_{>N} u_0\|_{ H^s}^2
+ \big(\| \P_{>N} u_{n}(0)\|_{ H^s}^2 - \| \P_{>N} u_0\|_{ H^s}^2\big)
+ C(\|u_0\|_{H^s}) N^{-2\dl}\notag \\
& < \eps, 
\label{K6b}
\end{align}

\noi
for $N \geq N_0$ and $n \geq \cj N$.
Then, 
\eqref{K6a} follows
from \eqref{K6b}.
Then, we can repeat the second half of the proof of Lemma \ref{LEM:Kcpt}
on $\{v_n\}_{n \geq \cj N}$
and find 
a finite $\eps$-cover 
$\{v_{n_k}: n_k \geq \cj N \}_{k = 1}^K$
of $\{v_n\}_{n \geq \cj N}$.
Finally, 
$\{v_n \}_{n = 1}^{\cj N - 1} \cup \{v_{n_k}: n_k \geq \cj N \}_{k = 1}^K$
forms a finite $\eps$-cover of 
$\{v_n\}_{n \in \mathbb{N}}$.
Therefore, $\{v_n \}$ is precompact in $C([-T, T];H^s(\T))$
and 
there exists a subsequence $\{v_{n_k}\}$ 
converging to  
some $v$ in $C([-T, T];H^s(\T))$.
In particular, note that 
\begin{align}
v(0) = \lim_{k\to \infty} v_{n_k}(0)
= \lim_{n\to \infty} u_{n}(0) = u_0 \qquad \text{in }H^s(\T).
\label{K6e}
\end{align}

Now, we are ready to prove a contradiction.
The main idea is to exploit
the faster and faster oscillation in \eqref{K6c}
as $n \to \infty$.
Let   $\phi \in \mathcal D(\T\times [-T, T])$
be a test function. 
Then, we have 
$\jb{u_n(\cdot, t), \phi(\cdot, t)}_{L^2_x}
\to F(t): = \jb{u(\cdot, t), \phi(\cdot, t)}_{L^2_x}$.
Moreover, we have $F(t) \in L^1(\R)$ 
since it is continuous and supported on $[-T, T]$.
Hence, by Riemann-Lebesgue Lemma, we have 
\begin{align*}
\bigg|\iint v_n \phi \, dx dt\bigg|
& = \bigg|\int e^{- 2 i t \fint_\T |u_{0, n}|^2 dx} 
\jb{u_n(\cdot, t), \phi(\cdot, t)}_{L^2_x} dt \bigg|\\
& \leq  \bigg|\int e^{- 2 i t \fint_\T |u_{0, n}|^2 dx} 
\jb{u(\cdot, t), \phi(\cdot, t)}_{L^2_x} dt \bigg|\\
& \hphantom{XXX}
 + \int | \jb{u (\cdot, t)- u_n(\cdot, t), \phi(\cdot, t)}_{L^2_x}| dt 
 \longrightarrow 0,
\end{align*}

\noi
as $n \to \infty$.
Note that the second term tends to 0 in view 
of the assumption (ii): 
$ \| u - u_n\|_{C_T H^s}\to 0$.

On the one hand, $ v_n$ converges
to $0$ in  $\mathcal{D}'(\T\times [-T, T])$.
On the other hand, 
there exists a subsequence $\{v_{n_k}\}$
converging to $v$ in $C([-T, T]; H^s(\T))$.
Hence, we conclude that $v\equiv 0$.
In particular, from \eqref{K6e}, we have
$u_0 = v(0) = 0$.
This is clearly a contradiction to $u_0 \in H^s(\T) \setminus L^2(\T)$.
This completes the proof of Theorem \ref{THM:2}.

\end{document}